\documentclass[a4paper,twosided,11pt]{article}
\usepackage{amsfonts}
\usepackage{latexsym}
\usepackage{amsmath}
\usepackage{amssymb}
\usepackage{amsthm}
\usepackage{amsmath}
\usepackage{hyperref}
\usepackage{verbatim}
\usepackage{bibspacing}
\usepackage{tikz-cd}
\usepackage{enumitem} \usepackage{pdfpages}

\newlist{enumerateb}{enumerate}{1}
\setlist[enumerateb]{label=(\arabic*$b$),ref=\arabic*$b$}

\newtheorem{thm}{Theorem}[section]
\newtheorem*{thm*}{Theorem}

\newtheorem{corollary}[thm]{Corollary}

\newtheorem{lemma}[thm]{Lemma}
\newtheorem{prop}[thm]{Proposition}
\newtheorem*{prop*}{Proposition}
\newtheorem{proposition}[thm]{Proposition}

\newtheorem*{conj*}{Conjecture}

\newtheorem*{defn*}{Definition}
\theoremstyle{definition}
\newtheorem{rem}[thm]{\textbf{Remark}}
\newtheorem*{rmk*}{Remark}
\newtheorem*{fact*}{Fact}

\theoremstyle{proof}

\newcommand{\Hess}{\text{Hess}}

\newcommand{\vol}{\textrm{Vol}}
\newcommand{\norm}[1]{\left\Vert#1\right\Vert}
\newcommand{\snorm}[1]{\Vert#1\Vert}

\newcommand{\set}[1]{\left\{#1\right\}}
\newcommand{\brac}[1]{\left(#1\right)}
\newcommand{\scalar}[1]{\left \langle #1 \right \rangle}
\newcommand{\sscalar}[1]{\langle #1 \rangle}

\newcommand{\R}{\mathbb{R}}

\newcommand{\EE}{\mathcal{E}}

\newcommand{\II}{\textrm{II}}
\newcommand{\Id}{\text{Id}}
\newcommand{\M}{\mathcal{M}}
\newcommand{\F}{\mathcal{F}}

\renewcommand{\H}{\mathcal{H}}
\newcommand{\eps}{\epsilon}
\newcommand{\K}{\mathcal{K}}
\renewcommand{\S}{\mathcal{S}}

\newcommand{\lin}{\text{lin}}

\newcommand{\Leb}{\mathfrak{m}}

\newcommand{\Det}{\text{\rm Det}}
\newcommand{\Ric}{\text{\rm Ric}}
\newcommand{\RR}{\text{\rm R}}
\newcommand{\tr}{\text{\rm tr}}
\newcommand{\D}{\mathcal{D}}
\renewcommand{\S}{\mathbb{S}}
\renewcommand{\div}{\text{\rm div}}
\newcommand{\grad}{\text{\rm grad}}
\newcommand{\ee}{{\mathsf e}}
\newcommand{\n}{\mathfrak{n}}

\def\inup{\rotatebox[origin=c]{90}{$\in$}}

\newlength{\defbaselineskip}
\numberwithin{equation}{section}

\begin{document}

\title{Centro-Affine Differential Geometry\\and the Log-Minkowski Problem}\date{}

\author{Emanuel Milman\textsuperscript{1}}

\footnotetext[1]{Department of Mathematics, Technion - Israel
Institute of Technology, Haifa 32000, Israel. 
The research leading to these results is part of a project that has received funding from the European Research Council (ERC) under the European Union's Horizon 2020 research and innovation programme (grant agreement No 637851). Email: emilman@tx.technion.ac.il.}

\begingroup    \renewcommand{\thefootnote}{}    
\endgroup

\maketitle

\begin{abstract}
We interpret the log-Brunn--Minkowski conjecture of B\"or\"oczky--Lutwak--Yang--Zhang as a spectral problem in centro-affine differential geometry. In particular, we show that the Hilbert--Brunn--Minkowski operator coincides with the centro-affine Laplacian, thus obtaining a new avenue for tackling the conjecture using insights from affine differential geometry. 
As every strongly convex hypersurface in $\R^n$ is a centro-affine unit-sphere, it has constant centro-affine Ricci curvature equal to $n-2$, in stark contrast to the standard weighted Ricci curvature of the associated metric-measure space, which will in general be negative. In particular, we may use the classical argument of Lichnerowicz and a centro-affine Bochner formula to give a new proof of the Brunn--Minkowski inequality. 
For origin-symmetric convex bodies enjoying fairly generous curvature pinching bounds (improving with dimension), we are able to show global uniqueness in the $L^p$- and log-Minkowski problems, as well as the corresponding global $L^p$- and log-Minkowski conjectured inequalities. 
 As a consequence, we resolve the \emph{isomorphic} version of the log-Minkowski problem: for any origin-symmetric convex body $\bar K$ in $\R^n$, there exists an origin-symmetric convex body $K$ with $\bar K \subset K \subset 8 \bar K$, so that $K$ satisfies the log-Minkowski conjectured inequality, and so that $K$ is uniquely determined by its cone-volume measure $V_K$. If $\bar K$ is not extremely far from a Euclidean ball to begin with, 
  an analogous \emph{isometric} result, where $8$ is replaced by $1+\eps$, is obtained as well. 
\end{abstract}

\section{Introduction}

A central question in contemporary Brunn--Minkowski theory is that of existence and uniqueness in the $L^p$-Minkowski problem for $p \in (-\infty,1)$: given a finite non-negative Borel measure $\mu$ on the Euclidean unit-sphere $\S^* = S^{n-1}$, determine conditions on $\mu$ which ensure the existence and/or  uniqueness of a convex body $K$ in $\R^n$ so that:
\begin{equation} \label{eq:intro-Lp-Minkowski}
S_p K := h_K^{1-p} S_K  = \mu  . 
\end{equation}
Here $h_K$ and $S_K$ denote the support function and surface-area measure of $K$, respectively -- we refer to Section \ref{sec:prelim} for standard missing definitions. When $h_K \in C^2(\S^*)$, 
\[
S_K = \det(D^2 h_K) \Leb,
\]
where $\Leb$ is the induced Lebesgue measure on $\S^*$, $D^2 h_K = {}^{\S^*} \nabla^2 h_K + h_K \delta^{\S^*}$ and ${}^{\S^*} \nabla$ is the Levi-Civita connection on $\S^*$ with its standard Riemannian metric $\delta^{\S^*}$. Consequently, (\ref{eq:intro-Lp-Minkowski}) is a Monge--Amp\`ere-type equation. 
It describes self-similar solutions to the (anisotropic) $\alpha$-power-of-Gauss-curvature flow for $\alpha = \frac{1}{1-p}$ \cite{Andrews-GaussCurvatureFlowForCurves,Andrews-FateOfWornStones,Andrews-PowerOfGaussCurvatureFlow,AndrewsGuanNi-PowerOfGaussCurvatureFlow,BCD-PowerOfGaussCurvatureFlow,ChoiDaskalopoulos-UniquenessInLpMinkowski,Chow-PowerOfGaussCurvatureFlow,Urbas-PositivePowersOfGaussCurvatureFlow,Urbas-NegativePowersOfGaussCurvatureFlow}.

The case $p=1$ above corresponds to the classical Minkowski problem of finding a convex body with prescribed surface-area measure; when $\mu$ is not concentrated on any hemisphere and its barycenter is at the origin, existence and uniqueness (up to translation) of $K$ were established by 
Minkowski, Alexandrov and Fenchel--Jessen (see \cite{Schneider-Book-2ndEd}), and regularity of $K$ was studied by Lewy \cite{Lewy-RegularityInMinkowskiProblem}, Nirenberg \cite{Nirenberg-WeylAndMinkowskiProblems}, Cheng--Yau \cite{ChengYau-RegularityInMinkowskiProblem}, Pogorelov \cite{Pogorelov-MinkowskiProblemBook}, Caffarelli \cite{CaffarelliHigherHolderRegularity,Caffarelli-StrictConvexity} and many others. 
The extension to general $p$ was put forth and publicized by E.~Lutwak \cite{Lutwak-Firey-Sums} as an $L^p$-analog of the Minkowski problem for the $L^p$ surface-area measure $S_p K = h_K^{1-p} S_K$ which he introduced. 
Existence and uniqueness in the class of origin-symmetric convex bodies, when the measure $\mu$ is even and not concentrated in a hemisphere, was established for $n \neq p > 1$ by Lutwak \cite{Lutwak-Firey-Sums} and for $p = n$ by Lutwak--Yang--Zhang \cite{LYZ-LpMinkowskiProblem}. A key tool in the range $p \geq 1$ is the prolific $L^p$-Brunn--Minkowski theory, initiated by Lutwak \cite{Lutwak-Firey-Sums,Lutwak-Firey-Sums-II} following Firey \cite{Firey-Sums},
and developed by Lutwak--Yang--Zhang (e.g. \cite{LYZ-LpAffineIsoperimetricInqs,LYZ-SharpAffineLpSobolevInqs,LYZ-LpJohnEllipsoids}) and others, which extends the classical $p=1$ case. Further existence, uniqueness and regularity results in the range $p > 1$ under various assumptions on $\mu$ were obtained in \cite{ChouWang-LpMinkowski,GuanLin-Unpublished,HuangLu-RegularityInLpMinkowski,HLYZ-DiscreteLpMinkowski,LutwakOliker-RegularityInLpMinkowski,Zhu-ContinuityInLpMinkowski}.

The case $p < 1$ turns out to be more challenging because of the lack of an appropriate $L^p$-Brunn--Minkowski theory. Existence, (non-)uniqueness and regularity under various conditions on $\mu$ were studied by numerous authors when $p<1$ (from either side of the critical exponent $p=-n$), especially after the important work by Chou--Wang \cite{ChouWang-LpMinkowski}, see e.g. \cite{BBC-SmoothnessOfLpMinkowski,BBCY-SubcriticalLpMinkowski,BHZ-DiscreteLogMinkowski,BoroczkyHenk-ConeVolumeMeasure,ChenEtAl-LocalToGlobalForLogBM,ChenLiZhu-LpMongeAmpere,ChenLiZhu-logMinkowski,HeLiWang-MultipleSupercriticalLpMinkowski, JianLuZhu-UnconditionalCriticalLpMinkowski,LuWang-CriticalAndSupercriticalLpMinkowski,EMilman-Isospectral-HBM,StancuDiscreteLogBMInPlane,Stancu-UniquenessInDiscretePlanarL0Minkowski,Stancu-NecessaryCondInDiscretePlanarL0Minkowski,Zhu-logMinkowskiForPolytopes,Zhu-CentroAffineMinkowskiForPolytopes,Zhu-LpMinkowskiForPolytopes,Zhu-LpMinkowskiForPolytopesAndNegativeP}. 
The case $p=0$ is of particular importance as it corresponds to the \emph{log-Minkowski problem} for the cone-volume measure
\[
V_K := \frac{1}{n} h_K S_K = \frac{1}{n} S_0 K ,
\]
described next. Note that $V_K$ is obtained as the push-forward of the cone-measure on $\partial K$ onto $\S^*$ via the Gauss map, and that the total mass of $V_K$ is $V(K)$, the volume of $K$. 
Being a self-similar solution to the isotropic Gauss curvature flow, the case $p=0$ and $\mu = \Leb$ of (\ref{eq:intro-Lp-Minkowski}) describes the ultimate fate of a worn stone in a model proposed by Firey \cite{Firey-ShapesOfWornStones} and further studied in \cite{Andrews-FateOfWornStones,AndrewsGuanNi-PowerOfGaussCurvatureFlow,BCD-PowerOfGaussCurvatureFlow,ChoiDaskalopoulos-UniquenessInLpMinkowski,Kolesnikov-OTOnSphere}. 
      
 \medskip
 
Let  $\K$ denote the collection of convex bodies in $\R^n$ containing the origin in their interior, and let $\K_e$ denote the subset of origin-symmetric elements. 
In \cite{BLYZ-logMinkowskiProblem}, B\"or\"oczky--Lutwak--Yang--Zhang showed that an \emph{even} measure $\mu$ is the cone-volume measure $V_K$ of an \emph{origin-symmetric} convex body $K \in \K_e$ if and only if it satisfies a certain subspace concentration condition, thereby completely resolving the existence part of the \emph{even} log-Minkowski problem. 
As put forth by B\"or\"oczky--Lutwak--Yang--Zhang in their influential work \cite{BLYZ-logMinkowskiProblem,BLYZ-logBMInPlane} and further developed in \cite{KolesnikovEMilman-LocalLpBM}, the uniqueness question is intimately related to the validity of a conjectured $L^0$- (or log-)Brunn--Minkowski inequality for origin-symmetric convex bodies $K,L \in \K_e$, which would constitute a remarkable strengthening of the classical $p=1$ case. The restriction to origin-symmetric bodies is natural, and necessitated by the fact that no $L^p$-Brunn--Minkowski inequality nor uniqueness in the $L^p$-Minkowski problem can hold for general convex bodies when $p < 1$ \cite{Andrews-ClassificationOfLimitingShapesOfIsotropicCurveFlows,ChenLiZhu-LpMongeAmpere,ChenLiZhu-logMinkowski,ChouWang-LpMinkowski,HeLiWang-MultipleSupercriticalLpMinkowski,JianLuWang-NonUniquenessInSubcriticalLpMinkowski,KolesnikovEMilman-LocalLpBM,Li-NonUniquenessInCriticalLpMinkowskiProblem,LLL-NonUniquenessInDualLpMinkowskiProblem,EMilman-Isospectral-HBM,Stancu-UniquenessInDiscretePlanarL0Minkowski}.

\smallskip

The following equivalence may be shown by following the arguments of \cite{BLYZ-logMinkowskiProblem,BLYZ-logBMInPlane} (see Subsection \ref{subsec:A} for a more general statement and further details). We denote by $\K^{2,\alpha}_{+,e}$ the subset of $\K_e$ having $C^{2,\alpha}$-smooth boundary and strictly positive curvature, and 
refer to Section \ref{sec:prelim} for the definition of the $L^p$-Firey--Minkowski sum $(1-\lambda) \cdot K+_p \lambda \cdot L$. 

\begin{thm}[after B\"or\"oczky--Lutwak--Yang--Zhang] \label{thm:intro-equiv}
The following statements are equivalent for any fixed $p \in (-n,1)$:
\begin{enumerate}
\item \label{it:main1}
For any $q \in (p,1)$, uniqueness holds in the even $L^q$-Minkowski problem for any $K \in \K^{2,\alpha}_{+,e}$:
\begin{equation} \label{eq:intro-equiv-uniqueness}
\forall L \in \K_e \;\;, \; \; S_q L = S_q K \; \Rightarrow \; L = K . 
\end{equation}
\item \label{it:main2}
The even $L^p$-Brunn--Minkowski inequality holds:
\begin{equation} \label{eq:intro-equiv-BM}
\forall \lambda \in [0,1] \;\;\; \forall K,L \in \K_e \;\;\; V((1-\lambda) \cdot K+_p \lambda \cdot L) \geq \brac{(1-\lambda) V(K)^{\frac{p}{n}} + \lambda V(L)^{\frac{p}{n}}}^{\frac{n}{p}} . 
\end{equation}
The case $p=0$, called the even log-Brunn--Minkowski inequality, is interpreted in the limiting sense as:
\[
 V((1-\lambda) \cdot K+_0 \lambda \cdot L) \geq V(K)^{1-\lambda} V(L)^{\lambda} .
 \]
\item \label{it:main3}
The even $L^p$-Minkowski inequality holds:
\begin{equation} \label{eq:intro-equiv-M}
\forall K,L \in \K_e \;\;\; \frac{1}{p} \int_{\S^*} h_L^{p} dS_p K  \geq \frac{n}{p} V(K)^{1-\frac{p}{n}} V(L)^{\frac{p}{n}} .
\end{equation}
The case $p=0$, called the even log-Minkowski inequality, is interpreted in the limiting sense as:
\[
\frac{1}{V(K)} \int_{\S^*} \log \frac{h_L}{h_K} dV_K \geq  \frac{1}{n} \log \frac{V(L)}{V(K)} .
\]
\end{enumerate}
\end{thm}

Using Jensen's inequality in formulation (\ref{eq:intro-equiv-M}) (or (\ref{eq:intro-equiv-BM})), it is immediate to check that the above (equivalent) statements become stronger as $p$ decreases (i.e. that their validity for $p_1$ implies their validity for $p_2$ whenever $p_1 < p_2 < 1$).

\begin{conj*}[B\"or\"oczky--Lutwak--Yang--Zhang, ``Even log-Brunn--Minkowski Conjecture"]
Any (and hence all) of the above statements hold for origin-symmetric convex bodies in the ``logarithmic case" $p=0$ (and hence for all $p \in [0,1)$ as well).  
\end{conj*}

A confirmation of this conjecture
would constitute a dramatic improvement over the classical Brunn--Minkowski theory for the subfamily of origin-symmetric convex bodies, which had gone unnoticed for over a century. 
The importance of this conjecture to the Brunn--Minkowski theory for general measures has been further expounded in several subsequent works \cite{LMNZ-BMforMeasures,HKL-LogBMForSubsets,Kolesnikov-OTOnSphere,KolesnikovLivshyts-ParticularFunctionsInLogBM,Saroglou-logBM1,Saroglou-logBM2};
see below for additional information and partial results. The even log-Brunn--Minkowski conjecture (a.k.a. log-Minkowski conjecture) is known to hold in the plane \cite{BLYZ-logBMInPlane} (see also \cite{GageLogBMInPlane,StancuDiscreteLogBMInPlane,MaLogBMInPlane,XiLeng-DarAndLogBMInPlane,Putterman-LocalToGlobalForLpBM}), but remains open in general for $n \geq 3$.

\smallskip

It is easy to show that (\ref{eq:intro-equiv-BM}) or (\ref{eq:intro-equiv-M}) are false for any $p < 0$ (see e.g. \cite{KolesnikovEMilman-LocalLpBM}).  
Moreover, uniqueness in (\ref{eq:intro-equiv-uniqueness}) does not hold for general $K,L \in \K_e$ and $q=0$, as may be verified by testing two different centered parallelepipeds with appropriately chosen parallel facets. The latter example is known to be the only exception to uniqueness in the log-Minkowski problem in the plane \cite{BLYZ-logBMInPlane}, but in higher dimension there are additional conjectured (non-smooth) cases of equality in (\ref{eq:intro-equiv-uniqueness}) when $q=0$ \cite{BoroczkyDe-StableLogBMWithSymmetries,BoroczkyKalantz-LogBMWithSymmetries}. In particular, when $K \in \K^{2,\alpha}_{+,e}$, the (somewhat stronger) conjecture is that  (\ref{eq:intro-equiv-uniqueness}) should hold not only for $q \in (0,1)$ but for $q=0$ as well -- we will refer to this as the ``uniqueness in the even log-Minkowski problem" conjecture.

\subsection{Main Results}

We now turn to describe the main results of this work. As our first main result, we obtain the following uniqueness result for the even $L^p$-Minkowski problem, with corresponding even $L^p$-Minkowski inequality, under a fairly generous curvature bound assumption on $K$. We denote 
by $\II^{\partial K}$ the second fundamental form on $\partial K \subset \R^n$. 

\begin{thm} \label{thm:main-Lp-Minkowski}
Let $K \in \K^{2,\alpha}_{+,e}$ have a centro-affine image $\tilde K$ so that the following curvature pinching bounds hold:
\begin{equation} \label{eq:main-strong-curvature-bounds}
\frac{1}{R} |X|^2 \leq \II^{\partial \tilde K}(X,X) \leq \frac{1}{r} |X|^2 \;\;\; \forall X \in T \partial \tilde K ,
 \end{equation}
for some $R \geq r > 0$; in other words, all radii of curvature of $\tilde K$ are bounded between $r$ and $R$. Then for any:
\begin{equation} \label{eq:main-rR-cond}
3 - \frac{n-1}{2} \frac{r^2}{R^2} < p < 1 ,
\end{equation}
the even $L^p$-Minkowski problem for $K$ has a unique solution:
\begin{equation} \label{eq:main-Lp-Minkowski-Uniqueness}
\forall L \in \K_e \;\;, \; \;  S_p L = S_p K  \; \Rightarrow \; L = K ,
\end{equation}
and the even $L^p$-Minkowski inequality holds:
\begin{equation} \label{eq:main-Lp-Minkowski-Inq}
\forall L \in \K_e  \;\;\; \frac{1}{p} \int_{\S^*} h_L^{p} dS_p K  \geq \frac{n}{p} V(K)^{1-\frac{p}{n}} V(L)^{\frac{p}{n}}  ,
\end{equation}
with equality if and only if $L = c K$ for some $c > 0$. 
\end{thm}

The term ``centro-affine image" above is synonymous with ``(non-degenerate) linear image", but we prefer to use the former to emphasize the centro-affine differential geometry underlying our results. Of course, the case $p=0$ above is interpreted in the limiting sense as follows:

\begin{corollary} \label{cor:main-log-Minkowski}
With the same conditions as above, whenever
\[
\frac{R^2}{r^2}  < \frac{n-1}{6},
\]
 the even log-Minkowski problem for $K$ has a unique solution:
\begin{equation} \label{eq:main-log-Minkowski-Uniqueness}
\forall L \in \K_e \;\;, \; \; V_L = V_K  \; \Rightarrow \; L = K ,
\end{equation}
 and the even log-Minkowski inequality holds:
 \begin{equation} \label{eq:main-log-Minkowski-Inq}
 \forall L \in \K_e \;\;\; \frac{1}{V(K)} \int_{\S^*} \log \frac{h_L}{h_K} dV_K \geq  \frac{1}{n} \log \frac{V(L)}{V(K)} ,
 \end{equation}
 with equality if and only if $L = c K$ for some $c > 0$. 
\end{corollary}

\begin{rem} \label{rem:intro-weakly-cont}
It is well-known that the measures $S_p K$ and in particular $V_K$ are weakly continuous (i.e. in duality with $C(\S^*)$) with respect to convergence of $K$ in the Hausdorff metric \cite[pp. 212-215]{Schneider-Book-2ndEd}. Consequently, for the purpose of deducing (\ref{eq:main-Lp-Minkowski-Inq}) or (\ref{eq:main-log-Minkowski-Inq}) without characterization of equality cases, it is enough to assume that $K \in \K_e$ can be approximated in the Hausdorff metric by $K_i \in \K^{2,\alpha}_{+,e}$ as above. 
\end{rem}

In fact, we will prove a strengthening of Theorem \ref{thm:main-Lp-Minkowski}, involving two-sided bounds on $\II^{\partial \tilde K}/h_{\tilde K}(\n^{\partial \tilde K})$, and producing a linear dependence on these bounds (instead of quadratic as above) -- see Theorem \ref{thm:Lp-Minkowski}. Using this strengthened version, we can give a positive answer to the \emph{isomorphic versions} of the uniqueness question for the even log-Minkowski problem and the even log-Minkowski inequality (curiously, we do not know how to establish these isomorphic results using the weaker formulation of Theorem \ref{thm:main-Lp-Minkowski} -- see Remark \ref{rem:isomorphic-conjecture}). 
Our \emph{isomorphic} nomenclature stems from Banach-space theory, where two Banach spaces are called isomorphic if, up to a linear bijection, their corresponding norms are equivalent up to constants. To better quantify this in the finite-dimensional geometric context, it is convenient to introduce the following distances for pairs of origin-symmetric convex bodies $K,L \in \K_e$ -- the geometric distance:
\[
d_G(K,L) := \inf \{ a b > 0 \; ; \; \frac{1}{b} K \subset L \subset a K \} ,
\]
and the Banach-Mazur distance:
\[
d_{BM}(K,L) := \inf \{ d_G(K , T(L)) \; ; \; T \in GL_n \} . 
\]
Clearly $d_G(K,L) , d_{BM}(K,L) \geq 1$. Note that the classical John's theorem \cite[Section 10.12]{Schneider-Book-2ndEd} asserts that $d_{BM}(K , B_2^n) \leq \sqrt{n}$ for any $K \in \K_e$, where $B_2^n$ denotes the Euclidean unit-ball in $\R^n$.

\begin{thm}[Isomorphic $L^p$-Minkowski] \label{thm:main-iso-Lp}
Let $\bar K \in \K_e$, and denote $D := d_{BM}(\bar K,B_2^n)$. Given $\gamma > 0$, define:
\begin{equation} \label{eq:main-p-gammaD}
p_{\gamma,D} :=  \frac{7}{3} - \frac{n-1}{24} \frac{\gamma^2}{D^2} .
\end{equation}
Then for any $8 \leq \gamma \leq D/2$, there exists $\tilde K \in \K^{\infty}_{+,e}$ so that:
\[
d_G(\bar K, \tilde K) \leq  \gamma ,
\]
and so that for any $p \in (p_{\gamma,D} , 1)$ and for any $T \in GL_n$, the even $L^p$-Minkowski problem for $K = T(\tilde K)$ has a unique solution (\ref{eq:main-Lp-Minkowski-Uniqueness}), and the even $L^p$-Minkowski inequality (\ref{eq:main-Lp-Minkowski-Inq}) holds for $K$, with equality if and only if $L = c K$ for some $c > 0$. 
\end{thm}

The above theorem agrees with the intuition that the smaller the Banach-Mazur distance $D$ of $\bar K$ to $B_2^n$ is, the smaller we can select $\gamma$ (controlling the distance between the original $\bar K$ and the modified $\tilde K$) in order to hit a particular value of $p_{\gamma,D}$. Note that for $\gamma = D$, we can simply select $\tilde K$ to be the John ellipsoid $\EE$ of $\bar K$ (so that $d_G(\bar K,\EE) = D$), for which it is known (see the next subsection) that the above conclusion holds with $p_{\gamma,D}  = -n$.
On the other extreme, when $\bar K = [-1,1]^n$ (so that $D = \sqrt{n}$), it follows from the results of \cite{KolesnikovEMilman-LocalLpBM} that for any fixed $p < 0$ there is no uniqueness in the even $L^p$-Minkowski problem for any $\tilde K$ with $\gamma = d_{G}(\bar K, \tilde K)$ close enough to $1$, so one cannot expect an estimate for $p_{1,\sqrt{n}}$ better than $0$. 
In this sense, the formula (\ref{eq:main-p-gammaD}) for $p_{\gamma,D}$ captures the correct order of magnitudes (proportional to $-n$ and to $1$) when $D/\gamma$ is of the order of $1$ and $\sqrt{n}$, respectively. 

\medskip

Of particular interest is the logarithmic case. Specializing to $p=0$ above, we resolve the isomorphic version of the conjecture regarding uniqueness in the even log-Minkowski problem:
\begin{corollary}[Isomorphic Log-Minkowski] \label{cor:main-iso-log}
For any $\bar K \in \K_e$, there exists $\tilde K \in \K^{\infty}_{+,e}$ with:
\[
d_G(\bar K , \tilde K) \leq 8 ,
\]
so that for any $T \in GL_n$, the even log-Minkowski problem for $K = T(\tilde K)$ has a unique solution  (\ref{eq:main-log-Minkowski-Uniqueness}), and the even log-Minkowski inequality (\ref{eq:main-log-Minkowski-Inq}) holds for $K$, with equality if and only if $L = c K$ for some $c > 0$.
\end{corollary}

This is an immediate corollary of Theorem \ref{thm:main-iso-Lp} since $D \leq \sqrt{n}$ by John's theorem, and so when $n \leq 64$ we can simply use $\tilde K = \EE$ John's ellipsoid (for which $D = d_G(\bar K,\EE) \leq \sqrt{n} \leq 8$), and when $n \geq 65$ the formula (\ref{eq:main-p-gammaD}) ensures that $p_{8,\sqrt{n}} < 0$. 

\bigskip

The constant $8$ obtained in the isomorphic version above is the worst case behavior for a general $\bar K \in \K_e$, when $D = d_{BM}(\bar K,B_2^n)$ may be as large as John's upper bound $\sqrt{n}$. However, whenever $D \ll \sqrt{n}$, a slightly finer analysis yields an \emph{isometric} version of the above results, where one only perturbs $\bar K$ by at most $\gamma = 1+\eps$. We only state the $p=0$ case below:

\begin{thm}[Isometric Log-Minkowski] \label{thm:main-isometric}
Let $\bar K \in \K_e$, and denote $D := d_{BM}(\bar K,B_2^n)$. There exists $\tilde K \in \K^{\infty}_{+,e}$ satisfying the conclusion of Corollary \ref{cor:main-iso-log}, so that:
\[
d_G(\bar K, \tilde K) \leq  1 + C \frac{\sqrt{D}}{\sqrt[4]{n}} ,
\]
where $C > 1$ is a universal constant.
\end{thm}

\subsection{Comparison with previous work}

As already mentioned, the validity of the log-Minkowski inequality (\ref{eq:main-log-Minkowski-Inq}) for all $K \in \K_e$, including characterization of its equality cases, as well as the uniqueness in the even log-Minkowski problem (\ref{eq:main-log-Minkowski-Uniqueness}) for $K \in \K_e$ which is not a parallelogram, was established when $n=2$ by B\"or\"oczky--Lutwak--Yang--Zhang \cite{BLYZ-logBMInPlane} (see also \cite{MaLogBMInPlane,Putterman-LocalToGlobalForLpBM,XiLeng-DarAndLogBMInPlane} for alternative derivations).

In our previous joint work with A.~Kolesnikov \cite{KolesnikovEMilman-LocalLpBM}, following the work of \cite{CLM-LogBMForBall},  we embarked on a systematic study of the validity of the \emph{local} $L^p$-Brunn--Minkowski inequality for origin-symmetric convex bodies and $p < 1$; by ``local" we mean on an infinitesimal scale, or equivalently, for pairs of bodies which are close enough to each other in an appropriate sense. 
  To that end, we introduced an elliptic second-order differential operator on $C^2(\S^*)$, called the Hilbert--Brunn--Minkowski operator $\Delta_K$, defined for $K \in \K^2_+$, which up to gauge transformations coincides with the operator introduced by Hilbert in his proof of the Brunn--Minkowski inequality (see \cite{BonnesenFenchelBook}).
Here $\K^m_+$ denotes the subset of $\K$ having $C^m$-smooth boundary and strictly positive curvature, and $\K^m_{+,e}$ denotes the subset of origin-symmetric elements. 
The operator $-\Delta_K$  is symmetric and positive semi-definite on $L^2(V_K)$, admitting a unique self-adjoint extension with compact resolvent. Its spectrum thus consists of a countable sequence of eigenvalues of finite multiplicity starting at $0$ and tending to $\infty$. It was shown in \cite{KolesnikovEMilman-LocalLpBM} that $\Delta_{K}$ enjoys a remarkable centro-affine equivariance property, stating that for any $T \in GL_n$, $\Delta_{T(K)}$ and $\Delta_K$ are conjugates modulo an isometry of Hilbert spaces; in particular, the spectrum $\sigma(-\Delta_{T(K)})$ is the same for all $T$. 
One way to define $(\Delta_K+(n-p) \Id) z$ is by linearizing $\log( h_K^{1-p} \det(D^2 h_K))$ appearing in the left-hand-side of (\ref{eq:intro-Lp-Minkowski}) under a logarithmic variation $h_{K_\eps} = h_K(1 + \eps z)$. Consequently, understanding whether $n-p$ is in the spectrum of $-\Delta_K$ is of fundamental importance to the uniqueness question in the $L^p$-Minkowski problem.

It was Hilbert who realized that the classical Brunn--Minkowski inequality (the case $p=1$) \cite{Schneider-Book-2ndEd} is equivalent to the statement that $\sigma(-\Delta_{K}) \cap (0,n-1) = \emptyset$, and proved that indeed $\lambda_1(-\Delta_K) = n-1$ where $\lambda_1$ denotes the first non-zero eigenvalue \cite{BonnesenFenchelBook}. Similarly, given $K \in \K^2_{+,e}$, we denote the first non-zero \emph{even} eigenvalue of $-\Delta_K$ (corresponding to an \emph{even} eigenfunction) by $\lambda_{1,e}(-\Delta_K)$. 
It was shown in \cite{KolesnikovEMilman-LocalLpBM} that for any $p < 1$, the statement $\lambda_{1,e}(-\Delta_K) \geq n-p$ is equivalent to the \emph{local} $L^p$-Brunn--Minkowski inequality for origin-symmetric perturbations of $K$, and implies the \emph{local} uniqueness for the even $L^q$-Minkowski problem for any $q > p$. The fact that a \emph{local} verification of these problems is enough to imply the \emph{global} one was subsequently shown by Chen--Huang--Li--Liu for the uniqueness of the $L^p$-Minkowski problem \cite{ChenEtAl-LocalToGlobalForLogBM} and by Putterman for the $L^p$-Brunn--Minkowski inequality \cite{Putterman-LocalToGlobalForLpBM}. 
We conjecture that $\lambda_{1,e}(-\Delta_K) > n$ for all $K \in \K^2_{+,e}$, which would confirm for all $p \in [0,1)$ the $L^p$-Brunn--Minkowski and $L^p$-Minkowski inequalities on $\K_e$ and the uniqueness in the $L^p$-Minkowski problem on $\K^2_{+,e}$.

Our main result in \cite{KolesnikovEMilman-LocalLpBM} was showing that $\lambda_{1,e}(-\Delta_K) \geq n-p_0$ for $p_0 :=1 - \frac{c}{n^{3/2}}$ and all $K \in \K^2_{+,e}$, where $c > 0$ is a universal constant, yielding local uniqueness in the even $L^p$-Minkowski problem for all $p \in (p_0 , 1)$. 
In \cite{ChenEtAl-LocalToGlobalForLogBM}, Chen--Huang--Li--Liu established their local-to-global principle for the uniqueness question, and deduced (\ref{eq:main-Lp-Minkowski-Uniqueness}) and (\ref{eq:main-Lp-Minkowski-Inq}) for all $K \in \K^{2,\alpha}_{+,e}$ and $p \in (p_0,1)$. 
In fact, thanks to recent progress on the KLS conjecture due to Y.~Chen \cite{Chen-AlmostKLS}, our estimate from \cite[Corollary 6.8 and Theorem 6.9]{KolesnikovEMilman-LocalLpBM} immediately improves to $p_0 = 1 - \frac{c_\eps}{n^{1+\eps}}$ for any $\eps > 0$, which together with the results of \cite{ChenEtAl-LocalToGlobalForLogBM} yields the presently best known range of $p$'s for which (\ref{eq:main-Lp-Minkowski-Uniqueness}) and (\ref{eq:main-Lp-Minkowski-Inq}) are known to hold. Furthermore, Chen--Huang--Li--Liu established in \cite{ChenEtAl-LocalToGlobalForLogBM} the validity of (\ref{eq:main-log-Minkowski-Uniqueness}) and (\ref{eq:main-log-Minkowski-Inq}) for the class $H_{\eps_n} := \{ K \in \K^{2}_{+,e} \; ; \; \norm{h_K - 1}_{L^\infty} \leq \eps_n \}$ and a sufficiently small $\eps_n > 0$, employing a corresponding local uniqueness result for $H_{\eps_n}$ established in \cite{KolesnikovEMilman-LocalLpBM}; in fact, the same argument applies to any centro-affine image $K = T(\tilde K)$, $T \in GL_n$ and $\tilde K \in H_{\eps_n}$. 

As already mentioned, it is known that for any $p < 0$ there exist $K \in \K^2_{+,e}$ for which (\ref{eq:main-Lp-Minkowski-Uniqueness}) and (\ref{eq:main-Lp-Minkowski-Inq}) are false (see \cite{EMilman-Isospectral-HBM} for additional information, and \cite{Andrews-ClassificationOfLimitingShapesOfIsotropicCurveFlows,ChouWang-LpMinkowski,HeLiWang-MultipleSupercriticalLpMinkowski,JianLuWang-NonUniquenessInSubcriticalLpMinkowski,Li-NonUniquenessInCriticalLpMinkowskiProblem,LLL-NonUniquenessInDualLpMinkowskiProblem,KolesnikovEMilman-LocalLpBM} for previously known non-uniqueness results). Consequently, the logarithmic case $p=0$ is precisely the conjectured threshold between the range $p \in [0,1)$ where (\ref{eq:main-Lp-Minkowski-Uniqueness}) and (\ref{eq:main-Lp-Minkowski-Inq}) are expected to hold for all $K \in \K^2_{+,e}$, and the range $p < 0$ where it is known that they fail in general. 

However, for a specific $K \in \K^2_{+,e}$, it is certainly possible for (\ref{eq:main-Lp-Minkowski-Uniqueness}) and (\ref{eq:main-Lp-Minkowski-Inq}) to hold with $p < 0$. For example, it is possible to show that these statements hold for all centered ellipsoids $K$ and for all $p \in (-n,1)$.  
Even in the simplest case when $K = B_2^n$, the Euclidean unit-ball, uniqueness in the $L^p$-Minkowski problem (\ref{eq:main-Lp-Minkowski-Uniqueness}) was until recently a major open problem in the latter range of $p$'s. As already eluded to above, this particular case is especially important because it describes 
self-similar solutions to the \emph{isotropic} $\alpha$-power-of-Gauss-curvature flow (for $\alpha = \frac{1}{1-p}$),  a model proposed by Firey \cite{Firey-ShapesOfWornStones} for $\alpha = 1$ ($p=0$). In the general anisotropic model, $x : \S^* \times [0,T) \rightarrow \R^n$ evolves according to:
\[
\frac{\partial x}{\partial t} = -  \brac{\rho(\n^{\partial L_t}_x)  \kappa^{\partial L_t}_x}^{\alpha}  \n^{\partial L_t}_x ~,~ \rho := \frac{dS_{p} K}{d \Leb} ,
\]
where $\n^{\partial L_t}$ is the outer unit-normal to $\partial L_t := x(\S^*,t)$ and $\kappa^{\partial L_t}$ is the corresponding Gauss-curvature. 
 Following contributions in \cite{Andrews-FateOfWornStones,AndrewsGuanNi-PowerOfGaussCurvatureFlow,ChoiDaskalopoulos-UniquenessInLpMinkowski,Chow-PowerOfGaussCurvatureFlow,Firey-ShapesOfWornStones,HuangLiuXu-UniquenessInLpMinkowski}, uniqueness in (\ref{eq:main-Lp-Minkowski-Uniqueness}) for the general isotropic case $K = B_2^n$ (without origin-symmetry, only assuming $L \in \K$) in the full range $p \in (-n,1)$ was resolved by Brendle--Choi--Daskalopoulos in \cite{BCD-PowerOfGaussCurvatureFlow}. In the origin-symmetric case, an extension of their uniqueness result from $B_2^n$ to arbitrary centered ellipsoids $\EE$ may be shown by following the arguments of \cite{KolesnikovEMilman-LocalLpBM,ChenEtAl-LocalToGlobalForLogBM} -- see Remark \ref{rem:ellipsoids}. Our uniqueness result of Theorem \ref{thm:main-Lp-Minkowski} thus extends the results of \cite{BCD-PowerOfGaussCurvatureFlow} in the origin-symmetric setting, from Euclidean balls (the isotropic case) to centro-affine images of convex bodies $K$ enjoying a curvature pinching condition (pinched anisotropic case). Specializing to ellipsoids $\EE$, while our general formula (\ref{eq:main-rR-cond}) does not recover the sharp exponent $p>-n$, we obtain the right order of magnitude ($p > 3 - \frac{n-1}{2}$). Note that uniqueness no longer holds below the critical exponent $p=-n$ due to the $SL_n$ equivariance of the centro-affine Gauss curvature $h_K^{1+n} \det(D^2 h_K)$ \cite{Tzitzeica1908,ChouWang-LpMinkowski}; in particular, the centro-affine Gauss curvature of any centered ellipsoid $\EE$ in $\R^n$ is constant and depends only on its volume: $S_{-n} \EE = c_n V(\EE)^2 \Leb$.

Strictly speaking, we are not aware of any other results establishing (\ref{eq:main-Lp-Minkowski-Uniqueness}) or (\ref{eq:main-Lp-Minkowski-Inq}) for a given $K \in \K_e$, $p < 1$ and \emph{all} $L \in \K_e$. Various additional results establish (\ref{eq:main-Lp-Minkowski-Uniqueness}), (\ref{eq:main-Lp-Minkowski-Inq}) or the corresponding $L^p$-Brunn--Minkowski inequality for particular pairs of convex bodies $K,L \in \K_e$, which are typically perturbations of a well-understood example, or which enjoy certain symmetries \cite{BoroczkyDe-StableLogBMWithSymmetries,BoroczkyKalantz-LogBMWithSymmetries,ColesantiLivshyts-LocalpBMUniquenessForBall,CLM-LogBMForBall,HKL-LogBMForSubsets,KolesnikovEMilman-LocalLpBM,Rotem-logBM,Saroglou-logBM1,Saroglou-logBM2}.  
Of particular historical significance was played by the case when $K,L \in \K^2_{+,e}$ are perturbations of $B_2^n$ ($C^2$-perturbations of particular form in \cite{CLM-LogBMForBall,ColesantiLivshyts-LocalpBMUniquenessForBall}, and centro-affine images of general $C^2$- and even $C^0$-perturbations of $B_2^n$ in \cite{KolesnikovEMilman-LocalLpBM}). However, the extent of these admissible $C^2$-perturbations was non-explicit and deteriorated with the dimension $n$. In contrast, note that the $C^2$-perturbations allowed by Theorem \ref{thm:main-Lp-Minkowski} and Corollary \ref{cor:main-log-Minkowski} are entirely explicit and in fact improve with the dimension -- e.g. Corollary \ref{cor:main-log-Minkowski} applies to any $K \in \K^{2,\alpha}_{+,e}$ with $R^2 / r^2 < \frac{n-1}{6}$.

As for the isomorphic and isometric results of Theorems \ref{thm:main-iso-Lp} and \ref{thm:main-isometric}, we are not aware of any prior results of this nature regarding the even $L^p$-Minkowski problem. The best comparison comes from a totally different yet equally fundamental problem posed by J.~Bourgain \cite{Bourgain-LK} regarding a volumetric property of convex bodies -- the Slicing Problem (see \cite{Milman-Pajor-LK,GreekBook}). The Slicing Problem has been confirmed for numerous families of convex bodies, and there has been recent dramatic advancement in the best known estimates for general convex bodies (obtained by combining the recent results of Chen \cite{Chen-AlmostKLS} on the Kannan--Lov\'asz--Simonovits conjecture \cite{KLS} with the results of Eldan--Klartag from \cite{EldanKlartagThinShellImpliesSlicing}). While the Slicing Problem remains open in general, the \emph{isomorphic version} of the Slicing Problem was fully resolved by B.~Klartag in \cite{KlartagPerturbationsWithBoundedLK}. Our results in Corollary \ref{cor:main-iso-log} and Theorem \ref{thm:main-isometric} can be seen as the log-Minkowski analogues of Klartag's results for the Slicing Problem (despite the two problems being very different, and having no apparent relation between our corresponding proofs). Note that we are not aware of an analogous result for the (also spectral) KLS conjecture (apart from an isometric quantitative stability result established in \cite{EMilman-RoleOfConvexity}).

\subsection{Centro-affine differential geometry}

Perhaps more important than our main results described above, is our rediscovery 
of the significance of affine differential geometry to the Brunn--Minkowski theory, and our apparently new observation about the crucial role played by the centro-affine normalization
(we refer to \cite{NomizuSasaki-Book,Schneider-Book-2ndEd} for more background, and to Section \ref{sec:ADG} for an introduction to affine differential geometry). Historically, the Brunn--Minkowski theory of convex sets was initiated by Brunn and subsequently Minkowski towards the end of the 19th century, and further developed by Blaschke, Berwald, Kubota, Favard, Alexandrov, Bonnesen, Fenchel and others in the first third of the 20th century (a singular but especially relevant contribution to the theory was also made by Hilbert at the turn of the century). 
In parallel, the origins of affine differential geometry are often attributed to the works of Tzitzeica \cite{Tzitzeica1908} circa 1908, following the axiomatization of affine geometry in Felix Klein's Erlangen program. A systematic study of affine differential geometry was subsequently undertaken between 1916 and 1923 by Blaschke in collaboration with Pick, Radon, Berwald and Thomsen, among others, and this was followed up in the work by Cartan, Kubota, S\"uss, \'Slebodzi\'nski, Salkowski and others in the late 1920's and 1930's. 
It is apparent from the large overlap in mathematicians working on both theories during those formative years that these theories interacted quite significantly.

However, after this initial period, each theory developed along its own respective trajectory, with little to no overlap with the other. Some notable exceptions include 
the study of the affine surface area and affine isoperimetric inequality, initiated by Blaschke and further developed and extended by 
Deicke, Hug, Leichtweiss, Ludwig, Lutwak, Meyer, Petty, Reitzner, Santal\'o, Sch\"utt, Werner, Ye and others (see \cite{Hug-ContributionsToAffineSurfaceArea,Hug-AffineSurfaceAreaOfPolar,Leichtweiss-HistoryOfAffineSurfaceArea,Ludwig-GeneralAffineSurfaceAreas,Lutwak-ExtendedAffineSurfaceArea,Lutwak-Firey-Sums-II} and the references therein), the early work by R.~Schneider  in the 1960's on global affine differential geometry \cite{Schneider-ZurAffinen}, 
and a more recent work of Klartag on convex affine hemispheres \cite{Klartag-AffineHemispheres}; all of these works pertain to the affine differential geometry obtained by equipping a convex set with \emph{Blaschke's equiaffine normal}, which is equivariant with respect to volume-preserving affine transformations. The equiaffine \emph{normalization} is the most prevalent one used in affine differential geometry, and some of the highlights of the resulting theory include the works by Calabi (see \cite{Calabi-CompleteAffineHyperspheres,Calabi-Bourbaki}) and Cheng--Yau \cite{ChengYau-CompleteAffineHypersurfacesI} on classification of equiaffine spheres. 
However, we will make the case in this work that a much more natural normalization for studying the Brunn--Minkowski theory is the \emph{centro-affine normalization}, which is equivariant with respect to centro-affine transformations (fixing the origin).  Contrary to Blaschke's equiaffine normalization, where the classification of non-compact equiaffine spheres has proven to be a major challenge, the centro-affine normalization is in a sense trivial, since the boundary of every $K \in \K^2_+$ is a centro-affine sphere. However, it is precisely this property which makes the centro-affine normalization so useful for our purposes. 

Before describing the relevance and usefulness of the centro-affine normalization to our setting, one should note that it has already been utilized in convex geometry through the notion of centro-affine surface area $\Omega_n(K)$, which coincides with the $L^p$-affine surface area for $p=n$ \cite{Hug-ContributionsToAffineSurfaceArea,Hug-AffineSurfaceAreaOfPolar,Lutwak-Firey-Sums-II} -- see Subsection \ref{subsec:CA-SelfDuality}. As is well-known, the centro-affine surface area is self-dual $\Omega_n(K) = \Omega_n(K^{\circ})$, and furthermore, the centro-affine metric of a hypersurface is isometric to that of the polar (or dual) hypersurface -- see Subsections \ref{subsec:CA-duality} and \ref{subsec:CA-SelfDuality} for additional self-duality properties enjoyed by the centro-affine normalization, and for suggestions regarding further research in this direction. 
In addition, the critical case $p=-n$ of the $L^p$-Minkowski problem (\ref{eq:intro-Lp-Minkowski}) was interpreted in \cite{ChouWang-LpMinkowski} as the Minkowski problem for the centro-affine Gauss-curvature. 
However, we are not aware of any other previously known connections between the centro-affine normalization and the Brunn--Minkowski inequality or any its variants -- this appears to be a novel observation, which is the main insight we would like to put forth and emphasize in this work. 

\medskip

Our first observation is that the Hilbert--Brunn--Minkowski operator $\Delta_K$ precisely coincides with the centro-affine Laplacian associated with $K \in \K^2_+$. We provide all the relevant details in Sections \ref{sec:ADG} and \ref{sec:CA}, and for now only explain what the latter notion entails. Any selection of a normal vector-field on $\partial K$ defines a Riemannian metric $g_K$ and a torsion-free affine connection $\nabla_K$ which in general is \emph{not} the Levi-Civita connection for $g_K$. The Laplacian associated with the given normalization $\Delta^{\nabla_K,g_K} f$ is then defined as the connection divergence $\div^{\nabla_K}$ of the metric gradient $\grad_{g_K} f$ (the vector field obtained by identification with the covector $df$ via the metric $g_K$). In addition, any (relative) normalization produces a volume measure $\nu_K$, which in general does not coincide with the Riemannian volume measure $\nu_{g_K}$, but is parallel with respect to $\nabla_K$ ($\nabla_K \nu_K = 0$); this allows us to integrate by parts:
\begin{equation} \label{eq:intro-parts}
\int (-\Delta^{\nabla_K,g_K} f) h \; d\nu_K = \int g_K(\grad_{g_K} f, \grad_{g_K} h) d\nu_K = \int f (-\Delta^{\nabla_K,g_K}  h) d\nu_K . 
\end{equation}
It turns out that for the centro-affine normalization of $\partial K$, defining $x$ itself to be the normal to $\partial K$ at $x \in \partial K$, the above objects boil down to some familiar ones from the Brunn--Minkowski theory (after parametrizing $\partial K$ on $\S^*$ via the Gauss map): the centro-affine volume measure $\nu_K$ coincides (up to constants) with the cone-volume measure $V_K$, the Riemannian volume measure $\nu_{g_K}$ for the centro-affine metric $g_K$ coincides (up to constants) with the centro-affine surface area measure $\Omega_{n,K}$, and the centro-affine Laplacian $\Delta^{\nabla_K,g_K}$ coincides with the Hilbert--Brunn--Minkowski operator $\Delta_K$.

 In \cite[Section 5.1]{KolesnikovEMilman-LocalLpBM}, we had originally (implicitly) identified the metric $g_K = \frac{D^2 h_K}{h_K} > 0$ on $\S^*$ by starting with the Hilbert--Brunn--Minkowski operator $\Delta_K$, performing the intergration-by-parts in (\ref{eq:intro-parts}) with respect to $V_K$ and computing the corresponding Dirichlet form, thereby interpreting $\Delta_K$ as the weighted Laplacian on $(\S^*, g_K , V_K)$. However, it was not entirely clear whether the choice of measure $V_K$ and thus construction of the metric $g_K$ are canonical, or what is the direct relation between these two objects; we now finally have a satisfactory answer coming from the centro-affine normalization. In addition, this gives a satisfactory explanation for the centro-affine equivariance property of the Hilbert--Brunn--Minkowski operator, originally observed in \cite[Section 5.2]{KolesnikovEMilman-LocalLpBM} following a lengthy computation. Furthermore, we deduce the centro-affine equivariance of all of the above differential objects ($g_K$, $\nabla_K$, $\nu_K$, etc...), as well as their behaviour under duality. In particular, we deduce (the known fact) that $(\S^*,g_K)$ and $(\S^*,g_{K^{\circ}})$ are isometric, and so any quantity derived from $g_K$ is the same for $K$ and $K^{\circ}$ (for example, $\Omega_n(K) = \frac{1}{n}\norm{\nu_{g_K}} = \frac{1}{n}\norm{\nu_{g_{K^{\circ}}}} = \Omega_n(K^{\circ})$). 

\medskip

 One of the key takeaways of our work is that in the context of the Brunn--Minkowski theory (and perhaps in other geometric problems), it is actually beneficial to use a calculus based on a well-suited non-Levi-Civita connection, instead of the usual weighted Levi-Civita calculus. 
As already mentioned, the boundary of any $K \in \K^2_+$ is a centro-affine ($(n-1)$-dimensional) unit-sphere, and so in particular, its centro-affine Ricci curvature is constant and equal to $n-2$. We stress that this is in stark contrast to the weighted Ricci curvature of $(\S^* , g_K , V_K)$, which will depend on third derivatives of $h_K$ and so will not be positive in general. A classical theorem of Lichnerowicz \cite{LichnerowiczBook} states that having a positive lower bound on the Ricci curvature of the Levi-Civita connection implies a lower bound on the first non-trivial eigenvalue of the associated Laplace-Beltrami operator. Lichnerowicz's proof is an immediate consequence of the $L^2$-method and an integrated Bochner formula for the Levi-Civita connection. It is possible to extend Bochner's formula to completely general affine connections, deriving an ``Asymmetric Bochner Formula". Applying this to the centro-affine connection $\nabla_K$, integrating with respect to $\nu_K$, and using that the centro-affine Ricci curvature is $n-2$, we obtain in Section \ref{sec:Bochner} the following centro-affine Bochner formula: 
\[ \int (\Delta_K f)^2 d\nu_K - \int \norm{\Hess^*_K f}_{g_K}^2 d\nu_K = (n-2) \int |\grad_{g_K} f|^2 d\nu_K 
\] (here $\Hess^*_K f$ denotes the Hessian with respect to the conjugate connection to $\nabla_K$ -- see Sections \ref{sec:ADG} and \ref{sec:CA}). 
As an immediate consequence, by verbatim repeating Lichnerowicz's argument, we obtain a new proof of the Brunn--Minkowski inequality (in its equivalent infinitesimal form) $\lambda_1(-\Delta_K) = n-1$, including the more delicate characterization of the corresponding $n$-dimensional eigenspace (originally due to Hilbert). 

Using the centro-affine Bochner formula, it easily follows that the conjectured even log-Brunn--Minkowski / log-Minkowski inequalities for $K \in \K^2_{+,e}$ are equivalent to the following new  inequality, which should hold for all \emph{even} test functions $f$:
\begin{equation} \label{eq:intro-factor-2}
\int \norm{\Hess_K^* f}_{g_K}^2 d\nu_K \geq 2 \int |\grad_{g_K} f|^2 d\nu_K .
\end{equation}
A particularly attractive feature of this new formulation is that the above inequality always holds
 for any $K \in \K^2_+$ and (not necessarily even) test function $f$ with constant $1$ instead of $2$ above, in which case it becomes equivalent to the usual (infinitesimal) Brunn-Minkowski inequality:
 \begin{equation} \label{eq:intro-Poincare}
\int z d\nu_K = 0 \;\; \Rightarrow \;\; \int (-\Delta_K z) z d\nu_K =  \int |\grad_{g_K} z|^2 d\nu_K \geq (n-1) \int |z|^2 d\nu_K 
 \end{equation}
 (see Remark \ref{rem:one-is-trivial}).  Consequently, the challenge is to use the evenness of the data in (\ref{eq:intro-factor-2}) to get a two-fold increase in the ``trivial" estimate, a factor which seems less mysterious than our previous local formulation from \cite{KolesnikovEMilman-LocalLpBM}, where the goal was to pass from the known $\lambda_1(-\Delta_K) = n-1$ to the conjectured $\lambda_{1,e}(-\Delta_K) \geq n$. This is now very reminiscent of the challenge in the resolution of the B-conjecture by Cordero-Erausquin--Fradelizi--Maurey \cite{CFM-BConjecture}, where the evenness of the data was used to gain a factor of two in the corresponding even eigenvalue estimate.  In some sense, the centro-affine normalization allows us to implement the strategy from \cite{CFM-BConjecture}, but we are still missing the final ingredient (\ref{eq:intro-factor-2}). Roughly speaking, the difficulty lies in the incompatibility between the centro-affine and Euclidean metrics, and so contrary to the Euclidean setting of \cite{CFM-BConjecture}, applying (\ref{eq:intro-Poincare}) to $z_v = df(v)$ for some fixed vector $v \in \S$ and averaging over $v$ does not yield the expressions appearing in (\ref{eq:intro-factor-2}). We are however able to verify (\ref{eq:intro-factor-2}) under the assumptions of Corollary \ref{cor:main-log-Minkowski} (and more generally, Theorem \ref{thm:curvature-implies-local-BM}). To this end, the centro-affine geometric interpretation plays a crucial role.

\bigskip

The rest of this work is organized as follows. In Section \ref{sec:prelim} we begin with some notation and required preliminaries, establishing in particular Theorem \ref{thm:intro-equiv}. In Section \ref{sec:ADG}, we provide the required background from affine differential geometry. In Section \ref{sec:CA}, we specialize the general theory to the centro-affine normalization for several useful parametrizations of $\partial K$, and compute various differential objects of interest. In Section \ref{sec:Bochner}, we derive the centro-affine Bochner formula, the equivalent local formulation (\ref{eq:intro-factor-2}), and a proof of the classical Brunn--Minkowski \`a-la Lichnerowicz. In Section \ref{sec:curvature} we provide a proof of Theorem \ref{thm:main-Lp-Minkowski}. In Section \ref{sec:iso}, we obtain our isomorphic and isometric Theorems \ref{thm:main-iso-Lp} and \ref{thm:main-isometric}.

\medskip

\noindent{\textbf{Acknowledgment}.} 
I thank Gaoyong Zhang for his comments regarding an earlier version of this manuscript.

\section{Preliminaries} \label{sec:prelim}

We begin with some preliminaries, referring to \cite{Schneider-Book-2ndEd,KolesnikovEMilman-LocalLpBM} and the references therein for additional information. 

\subsection{Notation} \label{subsec:prelim-notation}

Let $E = E^n$ denote an $n$-dimensional vector space over $\R$, which we will often identify with $\R^n$ via a fixed basis. 
The dual space to $E$ is $E^* = (\R^n)^*$, which we identify with $E$ via a fixed isomorphism $i : E \rightarrow E^*$. By abuse of notation, we also use $i : E^* \rightarrow E$ to denote the inverse isomorphism, and use $\scalar{\cdot,\cdot}$ to denote both the natural pairing between $E^*$ and $E$ and the induced Euclidean scalar product on $E$ and $E^*$ via $i$ (so that $\scalar{i(v),i(w)} = \scalar{i(v),w} = \scalar{v,w}$ for all $v,w \in E$). We denote the Euclidean norm by 
$|x| = \sqrt{\scalar{x,x}}$. The Euclidean unit-spheres in $(\R^n,\scalar{\cdot,\cdot})$, $E$ and $E^*$ are denoted by $S^{n-1}$, $\S$ and $\S^*$, respectively; they are equipped with their induced Lebesgue measures $\Leb$, $\Leb^{\S}$ and $\Leb^{\S^*}$ (or simply $\Leb$).

A convex body in $\R^n$ is a convex, compact set with non-empty interior. We denote by $\K = \K(E)$ the collection of convex bodies in $E$ having the origin in their interior.  The support function $h_K : E^* \rightarrow \R_+$ of $K \in \K(E)$ is defined as: 
\[
h_K(x^*) := \max_{x \in K} \scalar{x^*,x} \; ~,~ x^* \in E^* . 
\]
It is easy to see that $h_K$ is continuous, convex and positive outside the origin. Clearly, it is $1$-homogeneous, so we will mostly consider its restriction to $\S^*$.  Conversely, a convex $1$-homogeneous function $h : E^* \rightarrow \R_+$ which is positive outside the origin is necessarily a support function of some $K \in \K$. The dual body $K^* \in \K(E^*)$ of $K \in \K(E)$ is defined as the convex body in $E^*$ given by the level-set $\{ h_K \leq 1 \}$; duality implies that $(K^{*})^{*} = K$. The Minkowski gauge function of $K \in \K(E)$ is defined as:
\[
\norm{x}_K := \inf \{ t > 0 \; ; \; x \in t K \} ~,~ x \in E . 
\]
Note that $h_K = \norm{\cdot}_{K^{*}}$ on $E^*$ and $h_{K^{*}} = \norm{\cdot}_K$ on $E$. 
Given $K \in \K(E)$, we define the polar body $K^{\circ} \in \K(E)$ by identifying it with $K^* \in \K(E^*)$ via $i$, i.e. $i (K^{\circ}) = K^*$. 
The Minkowski sum $K_1 + K_2$ of two convex bodies is defined as $\{ x_1 + x_2 \; ; \; x_i \in K_i \}$. Note that this operation is additive on the level of support-functions: $h_{K_1 + K_2} = h_{K_1} + h_{K_2}$.

\medskip

We denote by $C^k(S^{n-1})$ and $C^{k,\alpha}(S^{n-1})$, $k = 0,1,2,\ldots$ and $\alpha \in (0,1)$, the space of $k$-times continuously and $\alpha$-H\"older differentiable functions on $S^{n-1}$, respectively, equipped with their usual corresponding topologies. When $k=0$, we simply write $C(S^{n-1})$ and $C^{\alpha}(S^{n-1})$. It is known \cite[Section 1.8]{Schneider-Book-2ndEd} that convergence of elements of $\K$ in the Hausdorff metric is equivalent to convergence of the corresponding support functions in the $C(\S^*)$ norm. 
\medskip

Given a smooth differentiable manifold $M$, the tangent and cotangent bundles are denoted by $TM$ and $T^* M$, respectively, and $\Gamma^k(TM)$ and $\Gamma^k(T^* M)$ denote the collection of $C^k$-smooth vector and covector fields on $M$. We use $X^i$ and $\omega_j$ to denote $X \in \Gamma^k(TM)$ and $\omega \in \Gamma^k(T^* M)$ in a local frame, and similarly for higher order contravariant and covariant tensors. A metric $(0,2)$ tensor $g$ is denoted by $g_{ij}$, and its inverse $(2,0)$ tensor by $g^{ij}$, so that $g^{ij} g_{jk} = \delta^i_k$, the Kronecker delta. Given a $C^1$-smooth function $f$ on $M$, we use $f_j$ to denote the $1$-form $(df)_j$ in a local frame.

\medskip

The standard flat affine connection on $\R^n$ is denoted by $\bar D$. 
Given a Euclidean structure $\scalar{\cdot,\cdot}$ on $\R^n$ and a closed smooth hypersurface $H$ with outer unit-normal $\n^H$ in $(\R^n,\scalar{\cdot,\cdot})$, we denote by ${}^H \nabla$ the induced Euclidean connection on $H$ and by $\II^H$ the corresponding second fundamental form, given by the Gauss equation:
\begin{equation} \label{eq:II}
\bar D_U V = {}^{H} \nabla_U V - \II^{H}(U,V) \n^H \;\;\;\; U \in T H \; , \; V \in \Gamma^1(T H) . 
\end{equation}
The induced Euclidean metric $\delta^H$ on $H$ is given by $\delta^H_p(u,v) = \scalar{u,v}$ for $u,v \in T_p H$. 
As usual, the (non-tensorial) Christoffel symbols associated to a local coordinate frame $\{e_1,\ldots,e_{n-1}\}$
 are defined via:
 \[
 {}^H \nabla_{e_i} e_j = {}^H \Gamma_{ij}^k e_k,
 \]
 and we have for any smooth function $f$ on $H$:
\[
{}^{H} \nabla^2_{ij} f = \partial^2_{ij} f - {}^H \Gamma_{ij}^k \partial_k f . 
\]
In addition, for any smooth extension of $f$ to a neighborhood of $H$:
\begin{equation} \label{eq:II-2nd-deriv}
\bar D^2 f(u,v) = {}^{H} \nabla^2 f(u,v) + \II^{H}(u,v) \n^H_p(f)  \;\;\; u,v \in T_p H .
\end{equation}

\medskip

Given $h \in C^2(\S^*)$, we extend $h$ as a $1$-homogeneous function on $E^*$, and define the symmetric $2$-tensor $D^2 h$ on $\S^*$ as the restriction of $\bar D^2 h$ onto $T \S^*$. Recalling (\ref{eq:II-2nd-deriv}) and using Euler's identity for $1$-homogeneous functions $\n^{\S^*}(h) = \scalar{\bar D h, \n^{\S^*}} = h$, it follows that in a local  frame $\{e_1,\ldots,e_{n-1}\}$ on $\S^*$:
\[
D^2_{ij} h = \bar D^2 h(e_i,e_j) = {}^{\S^*}\nabla^2_{ij} h + h \delta^{\S^*}_{ij}  ~,~ i,j=1,\ldots,n-1 . 
\]
Denoting by $C^k_{>0}(\S^*)$ the subset of positive functions in $C^k(\S^*)$, note that $h \in C^2_{>0}(\S^*)$ is a support-function of $K \in \K$ if and only if $D^2 h_K \geq 0$. 

\medskip

We denote by $\K^m_+$ the subset of $\K$ of convex bodies with $C^m$-smooth boundary and strictly positive curvature. By \cite[pp.~115-116,120-121]{Schneider-Book-2ndEd}, for $m \geq 2$, $K \in \K^m_+$ if and only if $h_K \in C^m_{>0}(\S^*)$ and $D^2 h_K > 0$. Similarly, $\K^{m,\alpha}_+$ denotes the subset of $\K^{m}_+$ of convex bodies with $C^{m,\alpha}$-smooth boundary ($\alpha \in (0,1]$), and for $m \geq 2$, $K \in \K^{m,\alpha}_+$ if and only if $h_K \in C^{m,\alpha}_{>0}(\S^*)$ and  $D^2 h_K > 0$. Consequently, by identifying elements of $\K^m_+$ and $\K^{m,\alpha}_+$ with their support functions whenever $m \geq 2$, we equip these spaces with their corresponding $C^m$ and $C^{m,\alpha}$ topologies, respectively. It is well-known that $\K^{\infty}_+$ is dense in $\K$ with respect to the Hausdorff metric (e.g. \cite[pp. 184-185]{Schneider-Book-2ndEd}).

\medskip

A convex body $K$ is called origin-symmetric if $K = -K$. We will always use $S_e$ to denote the origin-symmetric (or even) members of a set $S$, e.g. $\K_e$ and $\K^2_{+,e}$ denote subset of origin-symmetric bodies in $\K$ and $\K^2_+$, respectively, and $C^2_{e}(\S^*)$ denotes the subset of even functions in $C^2(\S^*)$. 

\medskip

We use $GL(E)$ to denote the group of non-singular linear (or centro-affine) transformations on $E$, and $SL(E)$ to denote the subgroup of volume and orientation preserving elements. When $E = \R^n$, we simply write $GL_n$.

\subsection{Brunn--Minkowski theory}

Given a convex body $K$ in Euclidean space $(E^n,\scalar{\cdot,\cdot})$, its surface-area measure $S_K$ is defined as the push-forward under the Gauss map $\n^{\partial K} : \partial K \rightarrow \S^*$ of $\H^{n-1}|_{\partial K}$. Recall that $\n^{\partial K}$ denotes the outer unit-normal to $K$ and $\H^{n-1}$ is the $(n-1)$-dimensional Hausdorff measure. 
When $K \in \K^2_{+}$, we have:
\[
S_K = \det(D^2 h_K) \Leb . \]
More generally, Lutwak introduced in \cite{Lutwak-Firey-Sums} the $L^p$ surface-area measure of $K$ as:
\[
 S_p K := h^{1-p}_K S_K . 
\]
The cone-volume measure $V_K$ on $\S^*$ is defined as:
\[
V_K = V_K^{\S^*} := \frac{1}{n} h_K S_K ;
\]
it is obtained by first pushing forward the Lebesgue measure on $K$ via the cone-map $K \ni x \mapsto x / \norm{x}_K \in \partial K$, and then pushing forward the resulting cone-measure $V^{\partial K}_K$ on $\partial K$ via the Gauss map $\n^{\partial K} : \partial K \rightarrow \S^*$. For completeness, note that if we instead push-forward the Lebesgue measure on $K$ via the radial-projection map $K \ni x \mapsto x / |x| \in \S$, we obtain:
\[
V^{\S}_K := \frac{1}{n} \frac{1}{\norm{\theta}_K^n} \Leb(d \theta) . 
\]

\medskip

Given two convex bodies $K_0,K_1$ in $E^n$, the classical Brunn--Minkowski inequality states that:
\begin{equation} \label{eq:BM}
V(K_0+K_1)^{\frac{1}{n}} \geq V(K_0)^{\frac{1}{n}} + V(K_1)^{\frac{1}{n}} ,
\end{equation}
where $V$ denotes volume (Lebesgue measure) and $K_0+K_1$ denotes the Minkowski sum of $K_0$ and $K_1$. The $L^p$-Minkowski sum $a \cdot K_0 +_p b \cdot K_1$ of $K_0, K_1 \in \K$ ($a,b \geq 0$) was defined by Firey for $p \geq 1$ \cite{Firey-Sums}, and extended by B\"{o}r\"{o}czky--Lutwak--Yang--Zhang \cite{BLYZ-logMinkowskiProblem,BLYZ-logBMInPlane} to all $p \in \R$, as the largest convex body (with respect to inclusion) $L$ so that:
\[
h_L \leq \brac{a h_{K_0}^p + b h_{K_1}^p}^{1/p} 
\]
(with the case $p=0$ interpreted as $h_{K_0}^{a} h_{K_1}^b$ when $a + b = 1$). 
Note that for $p \geq 1$ one has equality above, that the case $p=1$ coincides with the usual Minkowski sum, and that for $p < 1$ the resulting convex body $a \cdot K_0 +_p b \cdot K_1$ is the Alexandrov body associated to the continuous function on the right-hand-side. 

As a consequence of the Brunn--Minkowski and Jensen inequalities, Firey showed that for any $K_0,K_1 \in \K$ and $p \geq 1$:
\begin{equation} \label{eq:Lp-BM}
\forall \lambda \in [0,1] \;\;\; V((1-\lambda) \cdot K_0+_p \lambda \cdot K_1)^{\frac{p}{n}} \geq (1-\lambda) V(K_0)^{\frac{p}{n}} + \lambda V(K_1)^{\frac{p}{n}}  .
\end{equation}
It is not hard to show that the above statement for any $p < 1$ is false for general $K_0,K_1 \in \K$. However, it was conjectured by B\"{o}r\"{o}czky--Lutwak--Yang--Zhang \cite{BLYZ-logBMInPlane} that for \emph{origin-symmetric} $K_0,K_1 \in \K_e$, (\ref{eq:Lp-BM}) does in fact hold for all $p \in [0,1)$ -- we refer to this as the (even) $L^p$-Brunn--Minkowski conjecture. The validity of (\ref{eq:Lp-BM}) for all $K_0,K_1 \in \K_e$ and a given $p$ implies the validity for all $K_0,K_1 \in \K_e$ and any $q > p$, and so the case $p=0$, called the (even) log-Brunn--Minkowski conjecture, is the strongest in this hierarchy. 
As described in Theorem \ref{thm:intro-equiv} from the Introduction, the even $L^p$--Brunn--Minkowski conjecture is intimately related to the even $L^p$-Minkowski inequality  (\ref{eq:intro-equiv-M}) and to the uniqueness question in the even $L^p$-Minkowski problem (\ref{eq:intro-equiv-uniqueness}).
It turns out that the conjecture is also related to a certain spectral problem, described next.

\subsection{Hilbert--Brunn--Minkowski operator}

Following the work of \cite{CLM-LogBMForBall}, the \emph{local} version of the $L^p$-Brunn--Minkowski inequality (\ref{eq:Lp-BM}) was studied in our previous work with Kolesnikov \cite{KolesnikovEMilman-LocalLpBM}. Given $K \in \K^{2}_{+}$, the Hilbert--Brunn--Minkowski operator $\Delta_K : C^2(\S^*) \rightarrow C(\S^*)$ was defined in \cite{KolesnikovEMilman-LocalLpBM} as:
\begin{align*}
\Delta_K z = \Delta^{\S^*}_K z & := ((D^2 h_K)^{-1})^{ij} D^2_{ij}( z h_K) - (n-1) z \\
& =  ((D^2 h_K)^{-1})^{ij} ( h_K  {}^{\S^*}\nabla^2_{ij} z + (h_K)_i z_j + (h_K)_j z_i)  .
\end{align*}
Note that we are using a slightly different normalization than in \cite{KolesnikovEMilman-LocalLpBM}, where the Hilbert--Brunn--Minkowski operator (denoted $L_K$) was defined as $L_K := \frac{1}{n-1} \Delta_K$. Introducing the following Riemannian (positive-definite) metric on $\S^*$:
\begin{equation} \label{eq:prelim-metric}
g_K = g_K^{\S^*} := \frac{D^2 h_K}{h_K} > 0 ,
\end{equation}
we may also write:
\begin{equation} \label{eq:HBM}
\Delta_K z = g_K^{ij} ( {}^{\S^*}\nabla^2_{ij} z + (\log h_K)_i z_j + (\log h_K)_j z_i) .
\end{equation}
Clearly, $\Delta_K$ is an elliptic second-order differential operator with vanishing zeroth order term, and in particular $\Delta_K 1 = 0$. Up to gauge transformations, $\Delta_K$ coincides with the operator defined by Hilbert in his proof of the Brunn--Minkowski inequality \cite{BonnesenFenchelBook}.  

It was shown in \cite{KolesnikovEMilman-LocalLpBM} that the following integration-by-parts formula holds:
\[
 \int_{\S^*} (-\Delta_K z) w \; dV_K =  \int_{\S^*} g_K(\nabla z , \nabla w) dV_K = \int_{\S^*} (-\Delta_K w) z \; dV_K \;\;\; \forall z,w \in C^2(\S^*) .
\]
We use the notation $g_K(\nabla z, \nabla w) = g_K^{ij} z_i w_j$ and $|\nabla z|^2_{g_K} = g_K(\nabla z, \nabla z)$.
It follows that we may interpret $\Delta_K$ as the weighted Laplacian on the weighted Riemannian manifold $(\S^*, g_K, V_K)$ (see e.g. \cite{KolesnikovEMilmanReillyPart1,KolesnikovEMilmanReillyPart2}). 
Consequently, $-\Delta_K$ is a symmetric positive semi-definite operator on $L^2(V_K)$. It uniquely extends to a self-adjoint positive semi-definite operator with Sobolev domain $H^2(\S^*)$ and compact resolvent, which we continue to denote by $-\Delta_K$. Its (discrete) spectrum is denoted by $\sigma(-\Delta_K)$, and   its first non-zero eigenvalue is denoted by $\lambda_1(-\Delta_K)$. 

As known already to Minkowski, the Brunn--Minkowski inequality (\ref{eq:BM}) is equivalent to its local form (when $K_1$ is an infinitesimal perturbation of $K_0$). This local form was interpreted by Hilbert in a spectral language as:
\[
\lambda_1(-\Delta_K) \geq n-1 ,
\]
or equivalently:
\[
\int_{\S^*} z \; dV_K = 0 \;\; \Rightarrow \;\;\int_{\S^*} (-\Delta_K z) z \; dV_K \geq  (n-1) \int_{\S^*} z^2 dV_K \;\;\; \forall z \in C^2(\S^*)  .
\]
Hilbert showed that in fact $\lambda_1(-\Delta_K) = n-1$, characterizing in addition the corresponding eigenspace (see \cite[Section 5]{KolesnikovEMilman-LocalLpBM} or Subsections \ref{subsec:HBM} and \ref{subsec:BM} for more information).

\medskip

Now assume in addition that $K$ is origin-symmetric, i.e. that $K \in \K^{2}_{+,e}$. 
Denote by $H^2_e(\S^*)$ the even elements of the Sobolev space $H^2(\S^*)$ and by $\mathbf{1}^{\perp}$ those elements $f$ for which $\int f \; dV_K = 0$. The first non-trivial \emph{even} eigenvalue of $-\Delta_K$ is defined as:
\begin{align}
\nonumber \lambda_{1,e}(-\Delta_K) & := \min \sigma(-\Delta_K|_{H^2_e(\S^*) \cap \mathbf{1}^{\perp}}) \\
\label{eq:RR} & = \inf \set{ \frac{\int_{\S^*}|\nabla z|_{g_K}^2 dV_K}{\int_{\S^*} z^2 dV_K} \; ; \;  0 \neq z \in C^2_{e}(\S^*) , \int_{\S^*} z \; dV_K = 0 } .
\end{align}
It was shown in \cite{KolesnikovEMilman-LocalLpBM} that the validity of the \emph{local} form of the even $L^p$-Brunn--Minkowski inequality (\ref{eq:Lp-BM}) for $K \in \K^{2}_{+,e}$ is equivalent to the validity of the statement:
\[
\lambda_{1,e}(-\Delta_K) \geq n-p . 
\]
That the validity of the local form for all $K \in \K^2_{+,e}$ implies the validity of the global form (\ref{eq:Lp-BM}) for all $K_0,K_1 \in \K_e$ is trivial for $p \geq 1$ but not obvious at all when $p < 1$. The latter was conjectured in \cite{KolesnikovEMilman-LocalLpBM} and proved by Putterman in \cite{Putterman-LocalToGlobalForLpBM}, after a prior local-to-global result for the uniqueness question in the even $L^p$-Minkowski problem by Chen--Huang--Li--Liu \cite{ChenEtAl-LocalToGlobalForLogBM}.

\subsection{Proof of Theorem \ref{thm:intro-equiv}} \label{subsec:A}

We can now finally formulate an expanded version of Theorem \ref{thm:intro-equiv} from the Introduction, utilizing the full array of results from \cite{BLYZ-logMinkowskiProblem,BLYZ-logBMInPlane,BCD-PowerOfGaussCurvatureFlow,KolesnikovEMilman-LocalLpBM,ChenEtAl-LocalToGlobalForLogBM}, which we will require for establishing our results.

\begin{thm} \label{thm:equiv-full}
For a fixed $p \in (-n,1)$, statements (\ref{it:main1}), (\ref{it:main2}) and (\ref{it:main3}) of Theorem \ref{thm:intro-equiv} are equivalent to each other and to the following additional statements (with the usual interpretation when $q = 0$):
\begin{enumerateb}[start=2]
\item \label{it:main2b}
For all $q \in (p,1)$ and $K,L \in \K_{e}$, the even $L^q$-Brunn--Minkowski inequality holds:
\begin{equation} \label{eq:equiv-strong-BM}
\forall \lambda \in [0,1] \;\;\; V((1-\lambda) \cdot K+_q \lambda \cdot L) \geq \brac{(1-\lambda) V(K)^{\frac{q}{n}} + \lambda V(L)^{\frac{q}{n}}}^{\frac{n}{q}} ,
\end{equation}
with equality for some $\lambda \in (0,1)$ if and only if $L = c K$ for some $c > 0$. 
\item \label{it:main3b}
For all $q \in (p,1)$ and $K \in \K^{2,\alpha}_{+,e}$, the even $L^q$-Minkowski inequality holds:
\begin{equation} \label{eq:equiv-strong-M}
 \forall L \in \K_e \;\;\; \frac{1}{q} \int_{\S^*} h_L^{q} dS_q K  \geq \frac{n}{q} V(K)^{1-\frac{q}{n}} V(L)^{\frac{q}{n}} ,
\end{equation}
with equality if and only if $L = c K$ for some $c > 0$.
\end{enumerateb}
\begin{enumerate}
\setcounter{enumi}{3}
\item \label{it:main4}
For all $K \in \K^{2,\alpha}_{+,e}$, $\lambda_{1,e}(-\Delta_K) \geq n-p$. 
\end{enumerate}

Moreover, let $\F \subset \K^{2,\alpha}_{+,e}$ be any subfamily containing $B_2^n$ which is path-connected in the $C^{2,\alpha}$ topology. Namely, for any $K \in \F$, there exists $[0,1] \ni t \mapsto K_{t} \in \F$ a continuous path in the $C^{2,\alpha}$ topology so that $K_0 = B_2^n$ and $K_1 = K$. 
Then the implications (\ref{it:main4}) $\Rightarrow$ (\ref{it:main1}) $\Rightarrow$ (\ref{it:main3b}) remain valid 
 if we replace $\K^{2,\alpha}_{+,e}$ in these statements by $\F$. \end{thm}

Before providing a proof of Theorem \ref{thm:equiv-full}, we need to first collect several known ingredients from the literature. 
First, as explained in \cite[Section 6]{EMilman-Isospectral-HBM}, the standard regularity theory of the Monge-Amp\`ere equation implies that any solution $L \in \K$ to:
\begin{equation} \label{eq:regularity}
S_p L = f \Leb ~,~  f \in C^\alpha(\S^*) ~,~ f > 0 ,
\end{equation}
 necessarily satisfies $L \in \K^{2,\alpha}_{+}$ (note that without \emph{a-priori} assuming that $L \in \K$, so that $h_L > 0$ on $\S^*$, the asserted regularity is false and $L$ may not be $C^2$ smooth \cite[Section 6]{ChouWang-LpMinkowski}). In our context, this means that whenever $K \in \K^{2,\alpha}_{+,e}$, uniqueness in the $L^p$-Minkowski problem:
 \[
 S_p L = S_p K 
 \]
 is the same when considering solutions $L$ in either of the classes $\K_e$ or $\K^{2,\alpha}_{+,e}$. 
 
\medskip

Next, we need the following local uniqueness statement from \cite[Theorem 11.2]{KolesnikovEMilman-LocalLpBM}:
\begin{thm}[Kolesnikov--Milman] \label{thm:local-uniqueness}
Assume that the local even $L^p$-Brunn--Minkowski inequality (\ref{eq:local-Lp-BM}) holds for $K \in \K^2_{+,e}$ and some $p_0 < 1$. Then for any $p \in (p_0,1)$, 
the even $L^p$-Minkowski problem has a locally unique solution in a neighborhood of $K$ in the following sense: 
there exists a $C^2_e$-neighborhood $N_{K,p}$ of $K$ in $\K^2_{+,e}$, so that:
\begin{equation} \label{eq:locally-unique}
\forall L_1,L_2 \in N_{K,p} ~,~  S_{p} L_1 = S_{p} L_2 \;\; \Rightarrow \;\; L_1 = L_2 . 
\end{equation}
\end{thm}

The next ingredient we need was obtained in \cite{Firey-ShapesOfWornStones} (for $p=0$ and $L \in \K^\infty_e$), \cite{AndrewsGuanNi-PowerOfGaussCurvatureFlow} (for $p \in [0,1)$ and $L \in \K_e$) and finally completely resolved in \cite{BCD-PowerOfGaussCurvatureFlow} (see also \cite{Andrews-FateOfWornStones,ChoiDaskalopoulos-UniquenessInLpMinkowski,Chow-PowerOfGaussCurvatureFlow,HuangLiuXu-UniquenessInLpMinkowski} for additional contributions):
\begin{thm}[Firey, Andrews--Guan--Ni, Brendle--Choi--Daskalopoulos] \label{thm:Acta}
Let $-n < p < 1$ and $c > 0$. Then the $L^p$-Minkowski problem:
\[
L \in \K \; , \; S_p L = c \cdot \Leb 
\]
has a unique solution $L$ given by a centered Euclidean ball. 
\end{thm}

An additional crucial ingredient is the following theorem, which is the main new ingredient in the results of \cite{ChenEtAl-LocalToGlobalForLogBM}; as it is not explicitly stated in the manner formulated below, we sketch its proof for completeness. 
\begin{thm}[Chen--Huang--Li--Liu] \label{thm:local-to-global}
Let $p  < 1$, and let $[0,1] \ni t \mapsto K_{t} \in \K^{2,\alpha}_{+,e}$ be a continuous path in the $C^{2,\alpha}$ topology. Assume that the even $L^p$-Minkowski problem has a globally unique solution for $K_0$:
\[
\forall L \in \K^{2,\alpha}_{+,e} \;\;\; S_{p} L = S_{p} K_0 \;\; \Rightarrow \;\;\ L = K_0 . 
\]
Assume that for all $t \in [0,1]$, the even $L^p$-Minkowski problem has a locally unique solution in a $C^{2,\alpha}_e$-neighborhood $N_{K_t,p}$ of $K_t$ in the sense of (\ref{eq:locally-unique}). Then the even $L^p$-Minkowski problem has a globally unique solution for $K_1$:
\[
\forall L \in \K^{2,\alpha}_{+,e}  \;\;\; S_{p} L = S_{p} K_1 \;\; \Rightarrow \;\;\ L = K_1 . 
\]
\end{thm}
\begin{proof}[Sketch of Proof]
Let $K \in \K^{2,\alpha}_{+,e}$, and assume that there exists a $C^{2,\alpha}_{e}$ neighborhood $N_{K,p}$ of $K$ so that (\ref{eq:locally-unique}) holds.
 It was shown in \cite[Lemma 3.1]{ChenEtAl-LocalToGlobalForLogBM} that if the equation $S_p L = S_p K$ has a globally unique solution $L=K$ among all $L \in \K^{2,\alpha}_{+,e}$, then the equation $S_p L = S_p \tilde K$ has a globally unique solution $L = \tilde K$ for all $\tilde K$ in a $C^{2,\alpha}_{e}$ sub-neighborhood of $K$ in $N_{K,p}$. On the other hand, as explained in the proof of \cite[Theorem 1.4]{ChenEtAl-LocalToGlobalForLogBM}, it follows from \cite[Lemmas 3.2--3.4]{ChenEtAl-LocalToGlobalForLogBM} that if the equation $S_p L = S_p K$ has multiple distinct solutions $L \in \K^{2,\alpha}_{+,e}$, then the equation $S_p L = S_p \tilde K$ also has multiple distinct solutions  for all $\tilde K$ in a $C^{2,\alpha}_{e}$ sub-neighborhood of $K$ in $N_{K,p}$. 
 
Now apply the method of continuity following \cite{ChenEtAl-LocalToGlobalForLogBM}: define $I \subset [0,1]$ to be the subset of $t$'s for which $S_p L = S_p K_t$ has a globally unique solution $L = K_t$ among all $L \in \K^{2,\alpha}_{+,e}$. The results above imply that $I$ is both relatively open and closed in $[0,1]$, and hence $I$ is either empty or the entire $[0,1]$. But our assumption was that $0 \in I$, and hence $I = [0,1]$. 
\end{proof}

Finally, we will use the existence of a global minimizer in the following optimization problem \cite[Section 5]{ChouWang-LpMinkowski}. Once it is shown that a global minimum is attained, a very general variational argument \cite[Theorem 3.3]{Lutwak-Firey-Sums},\cite[Lemma 4.1]{BLYZ-logMinkowskiProblem}  ensures that any local minimizer satisfies the corresponding Euler-Lagrange equation (\ref{eq:LpMinkowski-EL}) (compare with the original argument of \cite[Theorem D]{ChouWang-LpMinkowski}):

\begin{thm}[Chou--Wang, Lutwak, B\"or\"oczky--Lutwak--Yang--Zhang] \label{thm:CW}
Let $f \in C^{\alpha}_e(\S^*)$ be a strictly positive even density, and denote $\mu = f \Leb$. Given $-n < p < 1$, consider the $0$-homogeneous functional:
\[
\K_e \in L \mapsto F_{\mu,p}(L) := \begin{cases} \frac{ \frac{1}{p} \int h_L^{p} d\mu }{ V(L)^{p/n} } & p \neq 0 \\ 
\frac{\exp(\int \log h_L d\tilde \mu)}{V(L)^{\frac{1}{n}} } & p = 0 \end{cases} ,
\]
where $\tilde \mu$ denotes the normalized measure $\mu/ \norm{\mu}$.
Then $F_{\mu,p}$ attains a global minimum. 

Moreover, any local minimizer (in the Hausdorff topology) $L$ of $F_{\mu,p}$ satisfies:
\begin{equation} \label{eq:LpMinkowski-EL}
S_p L = c \cdot \mu ,
\end{equation}
for some $c > 0$. In particular, by (\ref{eq:regularity}), necessarily $L \in \K^{2,\alpha}_{+,e}$. 
\end{thm}
\begin{rem}
Without the origin-symmetry assumption above, it is imperative to incorporate an additional maximization over all possible translations of $L$ so that the origin remains in $L$, rendering the analysis much more delicate \cite{ChouWang-LpMinkowski,ChenLiZhu-LpMongeAmpere,ChenLiZhu-logMinkowski,BBCY-SubcriticalLpMinkowski}. Nevertheless, one can still guarantee the existence of a global minimizer under even more general conditions on $\mu$ than the ones stated above: this was shown for densities $f$ satisfying $0 < c \leq f \leq C$ in \cite{ChouWang-LpMinkowski}, for densities $f \in L^{\frac{n}{n+p}}(\Leb)$ when $-n<p<0$ in \cite{BBCY-SubcriticalLpMinkowski}, and for finite Borel measures $\mu$ which are not concentrated on any hemisphere when $p \in (0,1)$ or which satisfy the subspace concentration condition when $p=0$ in \cite{ChenLiZhu-LpMongeAmpere,ChenLiZhu-logMinkowski}. 
\end{rem}

We can now finally provide a proof of Theorem \ref{thm:equiv-full}. 

\begin{proof}[Proof of Theorem \ref{thm:equiv-full}]
Statement (\ref{it:main1}) implies (\ref{it:main3b}) for each individual $q \in (-n,1)$ and $K \in \K^{2,\alpha}_{+,e}$. To see this, denote $\mu = S_q K$, and note that $\mu = f \Leb$ with a positive density $f \in C^{\alpha}_{e}(\S^*)$. Recall from Theorem \ref{thm:CW} that a global minimizer of $\K_e \ni L \mapsto F_{\mu,q}(L)$ always exists, and that any global minimizer $L$ must satisfy $S_q L  = c \cdot S_q K$ and is therefore in $\K^{2,\alpha}_{+,e}$. Consequently, if statement (\ref{it:main3b}) regarding the $L^q$-Minkowski inequality or its cases of equality were wrong, it would follow that there exists a global minimizer $L \in \K^{2,\alpha}_{+,e}$ which is different than $K$ so that (after rescaling) $S_q L = S_q K$, in contradiction to the uniqueness in the even $L^q$-Minkowski problem asserted in (\ref{it:main1}). 

The converse implication (\ref{it:main3b}) $\Rightarrow$ (\ref{it:main1}) also holds for each individual $q$ and \emph{both} pairs $(K,L)$ and $(L,K)$, for any fixed $K,L \in \K_e$. While we do not require this here, we provide a quick proof for completeness following Lutwak \cite{Lutwak-Firey-Sums}. Let $q \neq 0$; the case $q=0$ is treated in an identical manner. Assume that (\ref{it:main3b}) holds for both pairs $(K,L)$ and $(L,K)$ and that $S_q K = S_q L$. Denote $V_q(K,L) := \frac{1}{q} \int_{\S^*} h_L^q dS_q K$. Then by (\ref{it:main3b}):
\[
V_q(K,L) \geq V_q(K,K) = V_q(L,K) \geq V_q(L,L) = V_q(K,L) . 
\]
Consequently, equality holds throughout, and since $\frac{n}{q} V(L) = V_q(L,L) = V_q(K,K) = \frac{n}{q} V(K)$, it follows that we have equality in (\ref{eq:equiv-strong-M}), and hence $L = c K$ for some $c > 0$. But since $V(L) = V(K)$, we conclude that $L = K$, as asserted in (\ref{it:main1}).

Statement (\ref{it:main3b}) obviously implies (\ref{it:main3}) for general $K,L \in \K_e$ after recalling Remark \ref{rem:intro-weakly-cont} and taking the limit as $q \searrow p$.  Similarly, statement (\ref{it:main2b}) trivially implies (\ref{it:main2}) by taking the limit $q \searrow p$.
The equivalence of statements (\ref{it:main3}) and (\ref{it:main2}) was shown by B\"or\"oczky--Lutwak--Yang--Zhang in \cite{BLYZ-logBMInPlane}. 
That the global statement (\ref{it:main2}) implies the local one in (\ref{it:main4}) was shown in \cite{KolesnikovEMilman-LocalLpBM}. The local-to-global converse direction was established by Putterman in \cite{Putterman-LocalToGlobalForLpBM}, and also follows by the implications (\ref{it:main4}) $\Rightarrow$ (\ref{it:main1}) $\Rightarrow$ (\ref{it:main3b}) $\Rightarrow$ (\ref{it:main3}) $\Rightarrow$ (\ref{it:main2}). 

That (\ref{it:main3}) implies (\ref{it:main3b}) for any $q > p$ is a simple consequence of Jensen's inequality (after rescaling $K$ for convenience so that $\vol(K)=1$):
\begin{equation} \label{eq:tricky-Jensen}
\brac{\frac{1}{n} \int_{\S^*} h_L^q dS_q K}^{\frac{1}{q}} = \brac{ \int_{\S^*} \brac{\frac{h_L}{h_K}}^q dV_K }^{\frac{1}{q}} \geq \brac{\int_{\S^*} \brac{\frac{h_L}{h_K}}^p dV_K }^{\frac{1}{p}} = \brac{\frac{1}{n} \int_{\S^*} h_L^p dS_p K}^{\frac{1}{p}} . 
\end{equation}
To establish the characterization of equality in (\ref{it:main3b}), one may argue as in the proof of \cite[Theorem 1.8]{BLYZ-logBMInPlane}. Indeed, by (\ref{eq:tricky-Jensen}) and (\ref{it:main3}) we have for all $K \in \K^{2,\alpha}_{+,e}$ (say with $\vol(K) = 1$) and $L \in \K_e$:
\[
\brac{\frac{1}{n} \int_{\S^*} h_L^q dS_q K}^{\frac{1}{q}} \geq \brac{\frac{1}{n} \int_{\S^*} h_L^p dS_p K}^{\frac{1}{p}} \geq \vol(L)^{\frac{1}{n}} . 
\]
Consequently, if equality holds between the left and right most terms, we must have equality in Jensen's inequality (\ref{eq:tricky-Jensen}), and hence $h_L$ and $h_K$ must be proportional $V_K$-a.e.. But as $K \in \K^{2,\alpha}$, $V_K$ is absolutely continuous with respect to $\Leb$, and so by continuity $h_L(\theta) = c h_K(\theta)$ for some $c > 0$ and all $\theta \in \S^*$. 

Similarly, it is well-known and easy to check (see \cite[Theorem 2]{Firey-Sums}) that (\ref{it:main2}) implies (\ref{it:main2b}), after rescaling $K$ and $L$ by homogeneity so that $\vol(K) = \vol(L) = 1$ and noting that by Jensen's inequality, whenever $q > p$:
\begin{equation} \label{eq:tricky-inclusion}
(1-\lambda) \cdot K +_q \lambda \cdot L \supset (1-\lambda) \cdot K +_p \lambda \cdot L \;,
\end{equation}
with equality for some $\lambda \in (0,1)$ if and only if $K$ and $L$ are dilates. 
If equality holds in (\ref{it:main2b}) for some $\lambda_0 \in (0,1)$ and $K_0,L_0 \in \K_e$, then after rescaling $K_0, L_0$ into $K,L$ so that $\vol(K) = \vol(L) =1$, it follows by homogeneity that there exists $\lambda \in (0,1)$ so that equality holds in (\ref{it:main2b}) for $\lambda$, $K$ and $L$. By (\ref{eq:tricky-inclusion}) and (\ref{it:main2}) we know that:
\[
\vol((1-\lambda) \cdot K +_q \lambda \cdot L) \geq \vol((1-\lambda) \cdot K +_p \lambda \cdot  L) \geq 1 ,
\]
and as equality holds between the left and right most terms, we must have equality in (\ref{eq:tricky-inclusion}) (up to null-sets, and as the corresponding compact sets have non-empty interior, pointwise equality). It follows that $K$ and $L$ must be dilates, and hence so are $K_0$ and $L_0$. 

It remains to show that statement (\ref{it:main4}) implies (\ref{it:main1}) for a path-connected $\F \subset \K^{2,\alpha}_{+,e}$ containing $B_2^n$. Given $K \in \F$, there exists a continuous path in $\F$ (equipped with the $C^{2,\alpha}$ topology), denoted $[0,1] \ni t \mapsto K_t$, so that $K_0 = B_2^n$ and $K_1 = K$. Fix $q \in (p,1)$. Statement (\ref{it:main4}) and Theorem \ref{thm:local-uniqueness} imply that for all $t \in [0,1]$, the even $L^q$-Minkowski problem has a locally unique solution in a neighborhood of $K_t$ in the sense of (\ref{eq:locally-unique}). Consequently, as $K_0 = B_2^n$ satisfies the global uniqueness in the even $L^q$-Minkowski problem by Theorem \ref{thm:Acta}, it follows by Theorem \ref{thm:local-to-global} that $K_1 = K$ also satisfies the global uniqueness in the even $L^q$-Minkowski problem in the class $\K^{2,\alpha}_{+,e}$. As explained in the beginning of this subsection, the regularity theory for (\ref{eq:regularity}) implies that the uniqueness extends to the entire $\K_e$, thereby establishing (\ref{it:main1}). 
\end{proof}

\begin{rem} \label{rem:ellipsoids}
An immediate corollary of Theorem \ref{thm:equiv-full} is that uniqueness in the even $L^p$-Minkowski problem (\ref{eq:main-Lp-Minkowski-Uniqueness}) holds for all $p \in (-n,1)$ whenever $K$ is a centered ellipsoid $\EE$. Indeed, it well-known that $\lambda_{1,e}(-\Delta_{B_2^n}) = 2n$, as $\Delta_{B_2^n}$ coincides with the usual Laplace-Beltrami operator on $\S^*$ (see e.g. \cite{KolesnikovEMilman-LocalLpBM}). As the spectrum of $-\Delta_K$ is invariant under centro-affine transformations \cite[Section 5]{KolesnikovEMilman-LocalLpBM}, it follows that $\lambda_{1,e}(-\Delta_{\EE}) = 2n$ for all centered ellipsoids $\EE$. Hence, applying Theorem \ref{thm:equiv-full} to the family $\F = \{ \EE \}\subset \K^{2,\alpha}_{+,e}$ of all centered ellipsoids, the implication (\ref{it:main4}) $\Rightarrow$ (\ref{it:main1}) concludes the proof. 
\end{rem}

\section{Affine differential geometry} \label{sec:ADG}

In this section, we collect facts from affine differential geometry which we will need for this work; note that our sign choices in various places may be different from the standard ones. We refer to \cite{NomizuSasaki-Book,Bokan-Survey} for a detailed exposition and further information regarding affine differential geometry. For a development of the theory from the point of view of relative normalizations, we refer to \cite{SSV-Book,LSZH-Book,OlikerSimon-Polarity},
and from the point of view of statistical structures, we refer to \cite{Opozda-BochnerMethod,Opozda-CompletenesOfStatisticalStructures}.

\subsection{Normalization and structure equations}

Recall that $E = E^n$ denotes an $n$-dimensional linear vector space over $\R$. More general treatments assume that $E$ is an $n$-dimensional affine space and distinguish between $E$ and its tangent spaces, but for simplicity we will not require this here and identify $T_x E$ with $E$. $E$ is equipped with its standard flat affine connection $\bar D$ and a determinant volume form $\Det$ (note that all determinant volume forms coincide up to a scalar multiple). 

Let $M = M^{n-1}$ denote a smooth connected $(n-1)$-dimensional differentiable manifold. 
In our context, $M$ will always be orientable and closed, i.e. compact without boundary. Let $x : M^{n-1} \rightarrow E^n$ be a smooth immersion, that is a smooth map so that $d_p x$ is of maximal rank for all $p \in M$. In our context, $x : M \rightarrow E$ will \textbf{always be an embedding of a convex hypersurface with strictly positive curvature (``strongly convex")}. 
Let $\xi : M \rightarrow E$ denote a smooth transversal normal field to $x$, meaning that 
$\text{rank}(d_p x , \xi(p)) = n$ for all $p \in M$. 
 The transversal normal $\xi$ induces a volume form $\nu_{\xi}$ on $M$:
\[
\nu_{\xi}(e_1,\ldots,e_{n-1}) := \Det(dx(e_1),\ldots,dx(e_{n-1}) , \xi) ~,~ e_i \in T_p M .
\]
It also induces a connection $\nabla = \nabla^{\xi}$ and bilinear form $g = g^{\xi}$ on $M$ via the Gauss structure equation: 
\begin{equation} \label{eq:Gauss}
\bar D_U \; dx(V) = dx(\nabla^{\xi}_U V) - g^{\xi}(U,V) \xi ~,~ U \in TM , V \in \Gamma^1(T M) .
\end{equation}
It turns out that $\nabla^{\xi}$ is always a torsion-free affine connection, and that $g^{\xi}$ is a symmetric $(0,2)$ tensor, which is called the second fundamental form. In our context, since $x$ is strongly convex, $g^{\xi}$ is always definite, and so multiplying $\xi$ by $-1$, we can always make sure that $g^{\xi}$ is positive-definite, and hence defines a Riemannian metric on $M$. 

In addition, $\xi$ induces a $(1,1)$ tensor $S = S^{\xi} : TM \rightarrow TM$ called the shape operator and a $1$-form $\theta^{\xi}$ on $M$ via the Weingarten structure equation:
\begin{equation} \label{eq:Weingarten}
d\xi(V) = dx(S^{\xi}(V)) + \theta^{\xi}(V) \xi ~,~ V \in T M. 
\end{equation}

\subsection{Conormalization and structure equations}

The dual space to $E$ is denoted by $E^*$, and $\scalar{\cdot,\cdot} : E^* \times E \rightarrow \R$ denotes the corresponding pairing. $E^*$ is equipped with the same standard flat connection $\bar D$ and the dual volume form $\Det^*$, uniquely defined by requiring that:
\[
\Det^*(w^1, \ldots, w^n)  \Det(v_1,\ldots,v_n) = \det((\scalar{w^i,v_j})_{ij}) \;\;\; \forall w^i \in E^* , v_j \in E ,
\]
where $\det$ is the usual $n$ by $n$ determinant. 

A conormal field $\xi^* : M \rightarrow E^*$ is a smooth vector-field so that $\scalar{\xi^* , dx} = 0$. We will always normalize $\xi^*$ so that in addition $\scalar{\xi^* , \xi} = 1$; since $\text{rank}(dx,\xi) = n$, we see that $\xi$ determines $\xi^*$ uniquely. 

Observe that:
\[
g^{\xi}(u,v) = \scalar{d\xi^*(u),dx(v)} \;\;\; \forall u,v \in T_p M ;
\]
in particular, the right-hand side is symmetric in $u,v$. Indeed, using that $\scalar{\xi^* , dx} = 0$ twice and (\ref{eq:Gauss}), we have for $U \in TM, V \in \Gamma^1(TM)$:
\begin{align*}
& \scalar{d\xi^*(U),dx(V)} = U(\scalar{\xi^*, dx(V)}) - \scalar{\xi^*, \bar D_U \; dx(V)} \\
& = - \scalar{ \xi^* , dx(\nabla^{\xi}_U V) - g^{\xi}(U,V) \xi}  = \scalar{\xi^*,\xi} g^{\xi}(U,V) . 
\end{align*}

The conormal field $\xi^*$ induces a volume form $\nu^*_{\xi}$ on $M$:
\[
\nu^*_{\xi}(e_1,\ldots,e_{n-1}) := \Det^*(d\xi^*(e_1),\ldots,d\xi^*(e_{n-1}) , \xi^*) ~,~ e_i \in T_p M . 
\]
Since we assume that $x$ is strongly convex, it always holds that $\xi^* : M \rightarrow E^*$ is an immersion, and that $\xi^*$ is transversal to $\xi^*(M)$; in particular, $\nu^*_{\xi}$ is non-trivial.

Repeating the same construction as before, $\xi^*$  induces a torsion-free affine connection $\nabla^* = (\nabla^\xi)^*$ and symmetric $(0,2)$ tensor $\hat S = \hat S^{\xi}$ called the Weingarten form, via the Gauss structure equation:
\[
\bar D_U \; d \xi^*(V) = d\xi^* ((\nabla^\xi)^*_U V) - \hat S^{\xi}(U,V) \xi^* ~,~ U \in T M , V \in \Gamma^1(TM) .
\]
The Weingarten form and shape operator are 
related by:
\[
 \hat S (U,V) = g (S (U) , V) .
 \]

\subsection{Affine and equiaffine invariance} \label{subsec:ADG-AI}

$(M,\xi,\xi^*)$ is called a normalization of the hypersurface $x : M \rightarrow E$.

Let $\alpha : E \rightarrow E$ be a regular affine transformation given by $\alpha z = A z + b$. Then the hypersurface $x$ with normalization $(\xi,\xi^*)$ and the hypersufrace $\alpha x$ with normalization $(A \xi , A^{-*} \xi^*)$ induce the following exact same structures on $M$: $\nabla$, $\nabla^*$, $g$, $\theta$, $S$ and $\hat S$. We will say that these structures are affine-invariant. 

Note that the volume forms $\nu$ and $\nu^*$ are invariant under the above transformations only when $\det A = 1$, i.e. when $\alpha$ belongs to the unimodular (or equiaffine) group -- we will say in this case that they are equiaffine-invariant.

\subsection{Relative Normalization}

The transversal normal $\xi$ is called a \emph{relative-normal}, and $(M,\xi)$ is called a \emph{relative-normalization},
 if the $1$-form $\theta^{\xi}$ from the Weingarten equation (\ref{eq:Weingarten}) vanishes identically $\theta^{\xi} = 0$. The following statements are easily shown to be equivalent:
 \begin{enumerate}
 \item $\xi$ is a relative-normal: 
 \begin{equation} \label{eq:shape-operator}
 d\xi(V) = dx(S^\xi(V))   ~,~ \forall V \in T M . 
 \end{equation}
 \item $\xi$ is equiaffine, meaning that:
 \[
 \nabla^{\xi} \nu_{\xi} = 0.
 \]
 Note that in general one always has $\nabla^{\xi} \nu_{\xi} = \theta^{\xi} \nu_{\xi}$. We will not use the term equiaffine in this context, since it may be confused with Blaschke's notion of affine normal, which is a particular choice of relative-normalization described below. 
 \item The cubic form $A := -\frac{1}{2} \nabla^{\xi} g^{\xi}$ is a totally symmetric $(0,3)$ tensor. Since $g^{\xi}$ is already symmetric, it is enough to verify the symmetry with respect to the first two variables:
 \begin{equation} \label{eq:cubic-symmetric}
 (\nabla^{\xi}_X g^{\xi})(Y,Z) = (\nabla^{\xi}_Y g^{\xi})(X,Z) \;\;\; \forall X,Y,Z \in T_p M .
 \end{equation}
 A torsion-free affine connection $\nabla$ on a Riemannian manifold $(M,g)$ satisfying the Codazzi equation (\ref{eq:cubic-symmetric}) 
  is called a \emph{statistical connection} for $g$, and $(g,\nabla)$ is called a \emph{statistical structure} on $M$. This nomenclature is derived from the influential work of Amari (see e.g. \cite{Amari-StatisticalInference,Amari-ParametricFamilies}) regarding applications of such structures in Statistics, but is otherwise highly misleading, and so we will mostly avoid using it here. 
 \end{enumerate}
 
For a strongly convex hypersurface there are inifinitely many different relative normalizations. As for the conormal $\xi^*$, it turns out that we always have:
\[
(\nabla^{\xi})^* \nu^*_{\xi} = 0,
\]
i.e. the conormal always gives rise to a relative (or equiaffine) conormalization $(M,\xi^*)$.

\medskip

Recall that $\scalar{\xi^*, dx}=0$. For a relative-normalization, we also have the important property:
\[
\scalar{d\xi^*, \xi} = 0 .
\]
Indeed, this follows by differentiating $\scalar{\xi^*,\xi} = 1$ and using that $d\xi = dx \circ S^{\xi}$. 

\medskip

\textbf{From here on we assume that $(M,\xi,\xi^*)$ is a relative normalization of the hypersurface $x$.}

\subsection{Conjugation via the metric}

The connections $\nabla = \nabla^{\xi}$ and $\nabla^* = (\nabla^{\xi})^*$ are conjugate connections with respect to $g = g^{\xi}$, i.e. they satisfy:
\[
U g(V_1,V_2) = g(\nabla_U V_1 , V_2) + g(V_1 , \nabla^*_U V_2) \;\;\; \forall U \in TM \;\; \forall V_1,V_2 \in \Gamma^1(T M) . 
\]
In a local frame, it is straightforward to check that this is equivalent to:
\begin{equation} \label{eq:conjugation}
\nabla^*_i V^j = g^{j a} \nabla_i (g_{a b} V^b)  ~,~ \nabla^*_i \omega_j = g_{j a} \nabla_i (g^{a b} \omega_b) . 
\end{equation}
Denoting by $\nabla^g$ the Levi-Civita connection associated to the metric $g$, namely the unique torsion-free affine connection which is metric ($\nabla^g g = 0$), it easily follows that:
\[
\nabla^g = \frac{1}{2}( \nabla + \nabla^* ). 
\]

In addition, the volume forms $\nu = \nu_\xi$ and $\nu^* = \nu^*_{\xi}$ are conjugate with respect to the Riemannian volume form $\nu_{g}$:
\begin{equation} \label{eq:conjugate-measures}
\nu(e_1,\ldots,e_{n-1}) \nu^*(e_1,\ldots,e_{n-1}) = \det ( (g^{\xi}(e_i, e_j))_{ij} ) = \nu^2_{g}(e_1,\ldots,e_{n-1}). 
\end{equation}
Here and elsewhere, $\nu_g$ denotes the Riemannian volume form associated to the metric $g$. 

\subsection{Differential calculus} \label{subsec:ADG-calculus}

Recall that the divergence of a vector field $X$ on $M$ relative to an affine connection $\nabla$ is defined as:
\[
\div^{\nabla} X := \tr \{ Y \mapsto \nabla_Y X \} = \nabla_i X^i ,
\]
and that the Hessian of a function $f \in C^2(M)$ is defined as:
\[
\Hess^{\nabla} f (X,Y) := \nabla^2_{X,Y} f = (\nabla_X df) (Y) =  X(Y(f)) - (\nabla_X Y)(f) . 
\]
The Hessian $\Hess^{\nabla} f$ is a $(0,2)$ tensor, which is in addition symmetric if the connection $\nabla$ is torsion-free. By definition $\nabla_X f := X(f) = df(X)$. 

Assume that a volume-form $\nu$ satisfies $\nabla \nu = 0$. The divergence theorem implies that $\int_M \div^{\nabla} X \; d\nu = 0$, and so we have the integration-by-parts formula:
\[
\int_M f \div^{\nabla} X  \; d\nu = - \int_M \nabla_X f \; d\nu \;\;\; \forall X \in \Gamma^1(T M) \;\; \forall f \in C^1(M) . 
\]

\medskip

While an affine connection is the only structure needed to define the above differential operators, this is not the case whenever a trace over two simultaneouesly covariant or contravariant coordinates is required; in particular, there is no intrinsic definition of the Laplacian of $f \in C^2(M)$ as the trace of its Hessian. To make sense of this, one needs an extra metric structure $g$ on $M$. In that case, we denote by $\grad_g f \in \Gamma^1(TM)$ the unique vector-field satisfying:
\[
g(\grad_g f, X) = X(f) \;\;\; \forall X \in TM ,
\]
and define:
\[
\Delta^{\nabla,g} f := \div^{\nabla} \grad_g f . 
\]
In particular, we have the following integration-by-parts formula for all $f,h \in C^2(M)$:
\[
\int_M (\Delta^{\nabla,g} f) h  \; d\nu = - \int_M (\grad_g f)(h) d\nu = - \int_M g(\grad_g f,\grad_g h) d\nu = \int_M f (\Delta^{\nabla,g} h)  d\nu .
\]
In out setting, all of the above applies to both pairs $(\nabla^{\xi}, \nu_\xi)$ and $((\nabla^{\xi})^* , \nu_\xi^*)$.

\medskip

In a local frame:
\[
(\grad_g f)^i = g^{ij} f_j ~,~ \Delta^{\nabla,g} f = \nabla_i (g^{ij} f_j) . 
\]
Using (\ref{eq:conjugation}), we see that:
\[
\Delta^{\nabla,g} f = \nabla_i (g^{ij} f_j) = g^{ij} \nabla^*_i f_j = \tr_g \; \Hess^{\nabla^*} f ,
\]
where $\nabla^*$ is the $g$-conjugate connection to $\nabla$. In our setting, this applies to our $g^{\xi}$-conjugate connection pair $\nabla^{\xi}$ and $(\nabla^{\xi})^*$.

\subsection{Curvature}

Recall that the curvature $\RR$ of an affine connection $\nabla$ on $M$ is defined as the following $(1,3)$ tensor:
\[
\RR(X,Y)Z = \nabla_X \nabla_Y Z - \nabla_Y \nabla_X Z - \nabla_{[X,Y]} Z ,
\]
and that the Ricci $(0,2)$ tensor $\Ric$ is defined by tracing:
\[
\Ric(Y,Z) = \tr \{ X \mapsto \RR(X,Y) Z \} . 
\]
We denote the curvature and Ricci tensors of $\nabla = \nabla^{\xi}$ and $\nabla^* = (\nabla^{\xi})^*$ by $\RR, \Ric$ and $\RR^*, \Ric^*$, respectively. 
We subsequently omit the superscripts $\xi$ and $\xi^*$ in our various differential structures. 

The Gauss equations for $\RR$ and $\RR^*$ are:
\begin{align*}
\RR(X,Y) Z &= g(Y,Z) S X  - g(X,Z) S Y , \\
\RR^*(X,Y) Z & = \hat S(Y, Z) X - \hat S (X,Z) Y . 
\end{align*}
In particular, $\nabla^*$ is always projectively flat, and we have:
\[
g(\RR(X,Y) Z , W) = -g(\RR^*(X,Y) W , Z) . 
\]
Note that the usual symmetries of the Riemann curvature tensor need not hold for the curvature tensor $\RR$ of a general torsion-free affine connection, and that in general the Ricci tensor will not be symmetric; however, in our context, $\Ric, \Ric^*$ turn out to always be symmetric:
\[
\Ric(Y,Z) = \tr S \; g(Y,Z) - \hat S(Y,Z) ~,~ \Ric^*(Y,Z) = (n-2) \hat S(Y,Z) . 
\]

It easily follows that $\RR = \RR^*$ iff $\Ric = \Ric^*$ iff $S = \lambda \Id$, i.e. the hypersurface $x : M \rightarrow E$ is a relative affine sphere; a particular instance of this is when $x$ is a Blaschke (equi)affine sphere -- see below. 
In the context of statistical structures, when $\RR = \RR^*$ then $(g,\nabla,\nabla^*)$ is called a conjugate symmetric statistical structure. 

\subsection{Structure equations in a local frame}

Recall that in a local frame $\{e_1,\ldots,e_{n-1}\}$ on $M$, we use $x_i$ to denote the $E$-valued $1$-form $(dx)_i$, and similarly for the $E$- and $E^*$-valued $\xi_i$ and $\xi^*_i$. Let us stress that $x_i$ should not be confused with the $i$-th coordinate of $x$ in $E$ (especially since $E$ is not $(n-1)$-dimensional and since no coordinate system has been introduced on $E$). We summarize the (vector-valued) structure equations for a relative normalization in a local frame \cite[p. 33]{LSZH-Book}:
\begin{align*}
\Hess^{\nabla^{\xi}}_{ij} x  = \nabla^{\xi}_j x_i = - g^{\xi}_{ij} \xi \;\;\; & \text{Gauss equation for $x$} , \\
\Hess^{(\nabla^{\xi})^*}_{ij} \xi^* = (\nabla^{\xi})^*_j \xi^*_i = - \hat S^{\xi}_{ij} \xi^* \;\;\; & \text{Gauss equation for $\xi^*$} , \\
\xi_i  = (S^\xi)_i^k\; x_k \;\;\;&  \text{Weingarten equation for $\xi$} .
\end{align*}
We also have:
\begin{align*}
g^{\xi}_{ij} & = \scalar{\xi^*_j , x_i} =  \scalar{ d\xi^*(e_j) , dx(e_i)} ,\\
\hat S^{\xi}_{ij} & = g^{\xi}_{ik} (S^{\xi})^k_j .
\end{align*}

\subsection{Blaschke's equiaffine normalization}

The fundamental theorem of affine differential geometry states that there exists a unique relative-normal $\xi_0$ so that the Riemannian volume measure $\nu_{g^{\xi_0}}$ associated to the metric $g^{\xi_0}$ coincides with the induced volume measure $\nu_{\xi_0}$. Equivalently (up to orientation), this is the same as requiring that $|\nu_{\xi_0}| = |\nu^*_{\xi_0}|$. 
This unique $\xi_0$ is called the Blaschke affine normal, $g^{\xi_0}$ is called the Blaschke metric (or second fundamental form), and $(M,\xi_0)$ or $(M,g^{\xi_0},\nabla^{\xi_0})$ are called a Blaschke hypersurface. Clearly, the Blaschke normalization is equiaffine invariant, and is sometimes called the equiaffine normalization. There are various natural geometric and analytic ways to explicitly define the Blaschke affine normal 
\cite{LSZH-Book, NomizuSasaki-Book, OlikerSimon-Polarity, Klartag-AffineHemispheres}, and various related problems such as characterizing all Blaschke affine spheres have been an extremely active avenue of research \cite{Calabi-CompleteAffineHyperspheres,Calabi-Bourbaki,ChengYau-CompleteAffineHypersurfacesI}. However, in this work, we focus on a different natural relative-normalization.

\subsection{Centro-affine normalization} \label{subsec:ADG-CA}

Recall that $x : M \rightarrow E$ is assumed strongly convex, and assume further that the origin of $E$ lies on the inside of $x(M)$ (i.e. in the interior of the bounded component of $E \setminus x(M)$). 
The centro-affine normalization of the hypersurface $x$ is given by:
\[
\xi := x ,
\]
which is a transversal normal field thanks to our assumptions. The induced centro-affine metric is denoted by $g = g^x$. 
Clearly, this normalization is centro-affine invariant, namely, invariant under regular linear transformations (but not affine ones). Furthermore, inspecting (\ref{eq:shape-operator}), it is clearly a relative normalization with identity shape operator, and consequently:
\[
S = \Id ~,~ \hat S = g . 
\]
This means that any strongly convex hypersurface $x : M \rightarrow E$ is always a centro-affine sphere. In particular, the centro-affine sectional and Ricci curvatures are always constant:
\begin{align*}
\RR(X,Y) Z = \RR^*(X,Y) Z & = g(Y,Z) X  - g(X,Z) Y  \\
\Ric = \Ric^* & = (n-2) g . 
\end{align*}

If $\xi^*$ is the associated conormal field, we define:
\[
x^* := \xi^* . 
\]
Note that:
\begin{equation} \label{eq:duality}
\scalar{x^* , x} = 1 ~,~ \scalar{x^*, dx} = 0 ~,~ \scalar{dx^* , x} = 0 . 
\end{equation}
The symmetry between $x$ and $x^*$ immediately implies that the dual and primal centro-affine normalizations are related by conjugation
(cf. \cite{Laugwitz-DifferentialGeometryBook}, \cite[Prop. 7.2.1]{OlikerSimon-Polarity}). By this we mean the following: denote the metric and pairs of conjugate connections and volume forms on $M$ for the hypersurface $x : M \rightarrow E$ equipped with the normalization $(M, x, x^*)$ by
$g^x$, $\nabla^x$, $(\nabla^x)^*$, $\nu_x$ and $\nu^*_x$, and for the hypersurface $x^* : M \rightarrow E^*$ equipped with the normalization $(M,x^*,x)$ by $g^{x^*}$, $\nabla^{x^*}$, $(\nabla^{x^*})^*$, $\nu_{x^*}$ and $\nu^*_{x^*}$, respectively. Then:
\begin{equation} \label{eq:CA-dual-conjugate}
g^x = g^{x^*} ~,~ \nabla^x = (\nabla^{x^*})^* ~,~ (\nabla^x)^* = \nabla^{x^*} ~,~ \nu_x = \nu^*_{x^*} ~,~ \nu^*_x = \nu_{x^*} . 
\end{equation}
Note that the induced metric remains invariant under duality. 

The structure equations for the centro-affine normalization in a local frame $\{e_1,\ldots,e_{n-1}\}$ are:
\begin{equation} \label{eq:CA-structure}
\begin{array}{rl}
\Hess^{\nabla^x}_{ij} x = \nabla^x_j x_i = - g^{x}_{ij} x \;\;\; & \text{Gauss equation for $x$}  \\
\Hess^{\nabla^{x^*}}_{ij} x^* = \nabla^{x^*}_j x^*_i  = - g^{x^*}_{ij} x^* \;\;\; & \text{Gauss equation for $x^*$} ,
\end{array}
\end{equation}
where:
\[
g^{x}_{ij} = g^{x^*}_{ij} =\scalar{x^*_j, x_i} = \scalar{dx^*(e_j), dx(e_i) } . 
\]

\section{Centro-affine differential geometry of convex bodies} \label{sec:CA}

Fix a smooth convex body $K$ with strictly positive curvature in $E$ having the origin in its interior, $K \in \K^\infty_+$. 
Recall that we use a fixed isomorphism $i$ to identify between $E$ and $E^*$, and use $\scalar{\cdot,\cdot}$ to denote both the natural pairing between $E^*$ and $E$ and the induced Euclidean scalar product on $E$ and $E^*$ via $i$. Having fixed $i$ and thus the Euclidean structures on $E$ and $E^*$, we uniquely select the determinant form $\Det$ on $E$ so that $\Det(v_1,\ldots,v_n) = \sqrt{\det(\scalar{v_i,v_j})}$ and thus $\Det^*(w^1,\ldots,w^n) = \sqrt{\det(\scalar{w^i,w^j})}$ for all $v_i \in E$ and $w^j \in E^*$. 

Recall that $K^* \subset E^*$ denotes the dual body to $K$, and that $K^{\circ} \subset E$ is the corresponding polar body given via $i (K^{\circ}) = K^*$. Also recall that $\bar D$ denotes the standard flat covariant derivative on $E$ and $E^*$.

\medskip

We equip the strongly convex hypersurface $\partial K$ with the centro-affine normalization.

\subsection{Parametrizations}

It will be instructive to consider a parametrization $x^M_K : M \rightarrow \partial K \subset E$ for several natural manifolds $M$:
\[
M \in \M_K := \{ \S^*, \partial K , \S , \partial K^{*} \} . 
\]
We denote the induced metric and pairs of conjugate connections and volume forms on $M$ by:
\begin{equation} \label{eq:structures}
g^M_K , \nabla^M_K , (\nabla^M_K)^* , \nu^M_K , (\nu^M_K)^* . 
\end{equation}
We will not distinguish between the above volume forms $\nu$ and the corresponding volume measures $|\nu|$, using $\nu$ to denote both variants. 
We will sometimes write $O^{M \; *}_K$ instead of $(O^M_K)^*$ for $O \in \{g, \nabla, \nu \}$, especially when concatenating with another operation. 

\medskip

The above parametrizations of $\partial K$ are naturally obtained by appropriately composing the Gauss maps on $\partial K,\partial K^*$ and their inverses with the radial spherical projection $v \mapsto v / |v|$ in $E$ and $E^*$. Formally, for all $M_i,M_j \in \M_K$, we specify diffeomorphisms $T^{M_i \rightarrow M_j}_{K} : M_i \rightarrow M_j$ so that 
\begin{equation} \label{eq:transitivity}
T^{M_1 \rightarrow M_3}_{K} = T^{M_2 \rightarrow M_3}_{K}  \circ T^{M_1 \rightarrow M_2}_{K}  . 
\end{equation}
They are obtained by composing the following diffeomorphisms:
\[
\begin{array}{c c l c l c l c l}
\S^* & \rightarrow & \partial K & \rightarrow & \S & \rightarrow & \partial K^* & \rightarrow & \S^* \\
\inup & & \inup & & \inup & & \inup & & \inup \\
\theta^* & \mapsto & x = \bar D h_K(\theta^*) & \mapsto & \theta = \frac{x}{|x|} & \mapsto & x^* = \bar D \norm{\theta}_K & \mapsto & \theta^* = \frac{x^*}{|x^*|} .
\end{array}
\]
It will be useful to also explicitly specify the inverse diffeomorphisms:
\[
\begin{array}{c c l c l c l c l}
\S^* & \rightarrow & \partial K^* & \rightarrow & \S & \rightarrow & \partial K & \rightarrow & \S^* \\
\inup & & \inup & & \inup & & \inup & & \inup \\
\theta^* & \mapsto & x^* = \frac{\theta^*}{h_K(\theta^*)} & \mapsto & \theta = \frac{\bar D h_K(x^*)}{|\bar D h_K(x^*)|}  & \mapsto & x = \frac{\theta}{\norm{\theta}_K} & \mapsto & \theta^* = \frac{\bar D \norm{x}_K}{|\bar D \norm{x}_K|} .
\end{array}
\]
It is well-known and straightforward to check that the above cycles close up, so that indeed $T_K^{M \rightarrow M} = \Id$ for all $M$ and (\ref{eq:transitivity}) holds. Note that $T_K^{\partial K \rightarrow \S^*}$ and $T_K^{\partial K^* \rightarrow \S}$ are the Gauss maps for $\partial K$ and $\partial K^*$, respectively. 

It should already be clear (and will be verified below) that our parametrizations are understood in the following natural sense: $x \in \partial K$ is both the hypersurface and the centro-affine normal, $\theta^* \in \S^*$ is the unit outer-normal, $x^* \in \partial K^*$ is the centro-affine conormal (the dual point to $x$ on $\partial K^*$), and $\theta \in \S$ is the unit outer-normal to $\partial K^*$ at $x^*$, pointing in the direction of $x$ and thereby closing the cycle.

Setting:
\[
 x_K^M := T_K^{M \rightarrow \partial K},
 \]
 our definitions ensure that the following diagram commutes:
\[
\begin{tikzcd}[row sep = 40pt , column sep = 40pt]
 M_1  \arrow{d}[left]{T^{M_1 \rightarrow M_2}_{K}} \arrow{dr}{x^{M_1}_K} &  \\
 M_2  \arrow{r}{x^{M_2}_K} & \partial K \subset E .
\end{tikzcd}
\]
Consequently, $T^{M_1 \rightarrow M_2}_{K}$ induces an isomorphism between the objects $O^{M_i}_K$ defined on $M_i \in \M_K$ for each $O \in \{g, \nabla, \nabla^*, \nu, \nu^*\}$, and so for each of these, it is enough to calculate $O^{M_1}_K$ on a single convenient parametrization $M_1$, thereby obtaining $O^{M_2}_K$ for all other $M_2 \in \M_K$ by pushing-forward:
\[
O^{M_2}_K = (T^{M_1 \rightarrow M_2}_{K})_* O^{M_1}_K . 
\]
Our main object of interest will be $O_K$, regardless of the parametrization $M$, and so we will often omit the superscript $M$.

\subsection{Duality} \label{subsec:CA-duality}

Our parametrizations are compatible with the natural duality operation $\D_K$. For every $M \in \M_K$, denote by $M^* \in \M_K$ its obvious dual counterpart (e.g. $(\partial K)^* = \partial K^*$ and $(\S^*)^* = \S$). Note that for $M \in \{\S , \S^*\}$, $M^* = i(M)$ but not in general. By abuse of notation, we use the same notation $\D_K$ (omitting the reference to $M$) to denote the diffeomorphism:
\[
\D_K := T_K^{M \rightarrow M^*} : M \rightarrow M^*  ~,~ M^* = \D_K( M ) . 
\]
It is worthwhile to note that:
\[
\D_K  : \partial K \ni x \mapsto x^* \in \partial K^{*} ~,~ x^* = \D_K x = \bar D \norm{x}_{K} .
\]
To quickly see this, note that $\bar D \norm{x}_K$ is clearly perpendicular to $\partial K$, and that
\begin{equation} \label{eq:x*x=1}
\scalar{x^* , x} = \scalar{\bar D \norm{x}_{K} , x} = \norm{x}_K = 1 \;\;\;  \forall x \in \partial K,
\end{equation}
 by Euler's identity for the $1$-homogeneous function $\norm{x}_K$. Hence:
\[
h_K(\bar D \norm{x}_K) = \scalar{\bar D  \norm{x}_K , x} = 1 \;\;\; \forall x \in \partial K ,
\]
and we confirm that $\D_K$ maps $\partial K$ onto $\partial K^*$. Similarly,
\begin{align*}
\D_K : \partial K^* \ni x^* \mapsto x \in \partial K & ~,~ x = \D_K x^* = \bar D h_K(x^*) ,\\
\D_K : \S^* \ni \theta^* \mapsto \theta \in \S & ~,~\theta = \D_K \theta^* = \frac{\bar D h_K}{|\bar D h_K|}(\theta^*) . 
\end{align*}

Next, we tautologically extend our construction to strongly convex bodies in $E^*$ (and not just in $E$). Recall that $K^{\circ} \subset E$ and $K^* \subset E^*$ are related by $i(K^{\circ}) = K^*$, and hence $i(\partial K) = \partial (K^{\circ})^* = (\partial K^{\circ})^*$. We consequently define:
\[
\M_{K^*} := \M^*_K = i(\M_{K^{\circ}}) ,
\]
and set:
\[
T_{K^*}^{M_1 \rightarrow M_2} = i \circ T_{K^{\circ}}^{i(M_1) \rightarrow i(M_2)} \circ i \;\;\; \forall M_1,M_2 \in \M_{K^*} . 
\]
In particular:
\[
x_{K^*}^M = i \circ x_{K^{\circ}}^{i(M)} \circ i . 
\]
As the identification between $E$ and $E^*$ via $i$ is tautological and does not change any differential structure, we have for every $M \in \M_{K}$:
\[
g^{i(M)}_{K^{\circ}} = i_* g^{M}_{K^*} ~,~ \nabla^{i(M)}_{K^{\circ}} = i_* \nabla^{M}_{K^*} ~,~ (\nabla^{i(M)}_{K^{\circ}})^* = i_* (\nabla^{M}_{K^*})^* ~,~ 
\nu^{i(M)}_{K^{\circ}} = i_* \nu^{M}_{K^*} ~,~ (\nu^{i(M)}_{K^{\circ}})^* = i_* (\nu^M_{K^*})^* ,
\]
where $i_*$ denote the push-forward via $i$. In addition, we see that our construction is compatible with the duality operator $\D_K$, in the sense that the following diagram commutes:
\begin{equation} \label{eq:duality-commutes}
\begin{tikzcd}[row sep = 50pt , column sep = 50pt]
 M \arrow{r}{x_K^M} \arrow{rd}[outer sep=-2.0,pos=0.3]{x_{K^*}^M} \arrow{d}[left]{\D_K} & \partial K  \arrow{d}{\D_K}  \\
 M^* \arrow{r}{x_{K^*}^{M^*}} \arrow{ru}[outer sep=-4.0,pos=0.3]{x_K^{M^*}} & \partial K^*
\end{tikzcd} .
\end{equation}

As discussed in Subsection \ref{subsec:ADG-CA}, the duality operation is important in view of its conjugation role for the centro-affine normalization. Observe that the conormal $(x^M_K)^* : M \rightarrow E^*$ corresponding to the centro-affine normal $x^M_K : M \rightarrow E$ of the hypersurface $\partial K$ is precisely given by:
\[
(x^M_K)^* := \D_K \circ x^M_K = x^M_{K^*}. 
\]
This follows immediately by verifying the validity of the defining equations (\ref{eq:duality}); indeed, as already explained in (\ref{eq:x*x=1}), $x^* = \D_K x$ is perpendicular to $\partial K$ and satisfies $\scalar{x^*,x} = 1$. 

Consequently, it follows from (\ref{eq:CA-dual-conjugate}) that the induced structures on $M$ via $x^M_K : M \rightarrow \partial K$ and via $x^M_{K^*} : M \rightarrow \partial K^{*}$ are conjugate to each other:
\[
g^M_K = g^M_{K^{*}} , \nabla^M_K = (\nabla^M_{K^{*}})^* , (\nabla^M_K)^* =  \nabla^M_{K^{*}} , \nu^M_K =  (\nu^M_{K^{*}})^* , (\nu^M_K)^* = \nu^M_{K^{*}} .
\] 
Note that the induced centro-affine Riemannian metric $g^M_K$ is self-dual, as is well-known \cite{Laugwitz-DifferentialGeometryBook,OlikerSimon-Polarity,Hug-AffineSurfaceAreaOfPolar}. 
We emphasize several equivalent alternative forms of the above duality, which follow immediately from the commutation in (\ref{eq:duality-commutes}):
\[
\begin{array}{l  l  l  l}
 g^{i(M^*)}_{K^\circ} & =  i_* g^{M^*}_{K^*} &=  i_* g^{M^*}_{K} &= (i \circ \D_K)_* g^M_{K} , \\
 \nabla^{i(M^*)}_{K^\circ}  & =  i_* \nabla^{M^*}_{K^*} & =  i_* (\nabla^{M^*}_{K})^* & =  (i \circ \D_K)_* (\nabla^M_K)^* ,\\
 (\nabla^{i(M^*)}_{K^\circ})^* & =  i_* (\nabla^{M^*}_{K^*})^* & =  i_* \nabla^{M^*}_{K} & =  (i \circ \D_K)_*  \nabla^M_K ,\\
 \nu^{i(M^*)}_{K^{\circ}} & = i_* \nu^{M^*}_{K^*} & =  i_*(\nu^{M^*}_K)^* & = (i \circ \D_K)_* (\nu^M_K)^* , \\
 (\nu^{i(M^*)}_{K^{\circ}})^*  & =  i_* (\nu^{M^*}_{K^*})^* & =  i_*\nu^{M^*}_K & =  (i \circ \D_K)_* \nu^M_K . 
\end{array}
\]
In other words, up to conjugation, $i : M^* \rightarrow i(M^*)$ pushes forward  $O^{M^*}_K$ onto $O^{i(M^*)}_{K^{\circ}}$, and 
$i \circ \D_K : M \rightarrow i(M^*)$ pushes forward $O^M_K$ onto $O^{i(M^*)}_{K^{\circ}}$, for $O \in \{g,\nabla,\nabla^*,\nu,\nu^*\}$. The latter is particularly useful when $M \in \{ \S , \S^* \}$ since then $i(M^*) = M$. 

It is a good exercise to verify directly that, e.g., $i \circ \D_K$ pushes forward $g^{\S^*}_K$ onto $g^{\S^*}_{K^{\circ}}$.

\subsection{Centro-affine invariance} \label{subsec:CA-invariance}

Let $A \in GL(E)$. The centro-affine invariance of the centro-affine normalization immediately implies that $A$ pushes forward $O^{\partial K}_K$ onto $O^{\partial A(K)}_{A(K)}$ for all $O \in \{g, \nabla, \nabla^*\}$, as well as $O \in \{\nu,\nu^*\}$ whenever $A \in SL(E)$. 
As $A(K)^* = A^{-*}(K^*)$, it is also easy to see that $A^{-*}$ pushes forward $O^{\partial K^*}_K$ onto $O^{\partial A(K)^*}_{A(K)}$ in the same manner as above. 

The situation with our other parametrizations $M \in \{\S,\S^*\}$ requires a bit more thought. While these parametrizations are very natural from a geometric perspective, they are not centro-affine co- or contra-variant, e.g. it is not true that $x^{\S}_{A(K)} = A \circ x^{\S}_K$. Consequently, the centro-affine invariance from Subsection \ref{subsec:ADG-AI} only holds after a suitable change-of-variables: 

\begin{prop} \label{prop:CA-Invariance}
Given $A \in GL(E)$, denote:
\begin{align*}
A^{(0)} : \S  \rightarrow \S ~,~ & A^{(0)}(\theta) = \frac{A \theta}{|A \theta|} \\
 (A^{-*})^{(0)} : \S^* \rightarrow \S^* ~,~&  (A^{-*})^{(0)}(\theta^*) = \frac{A^{-*} \theta^*}{|A^{-*} \theta^*|} .
\end{align*}
Then:
\begin{equation} \label{eq:cov-formula}
x^{\S}_{A(K)} = A \circ x_K^{\S} \circ (A^{(0)})^{-1} ~,~ x^{\S^*}_{A(K)} = A \circ x_K^{\S^*} \circ ((A^{-*})^{(0)})^{-1} . \end{equation}
Consequently:
\begin{enumerate}
\item \label{it:pf1}
$A^{(0)}$ pushes forward $O^{\S}_K$ onto $O^{\S}_{A(K)}$ for all $O \in \{g, \nabla, \nabla^*\}$, as well as $O \in \{\nu,\nu^*\}$ whenever $A \in SL(E)$. 
\item \label{it:pf2}
$ (A^{-*})^{(0)}$ pushes forward $O^{\S^*}_K$ onto $O^{\S^*}_{A(K)}$ for all $O \in \{g, \nabla, \nabla^*\}$, as well as $O \in \{\nu,\nu^*\}$ whenever $A \in SL(E)$. 
\end{enumerate}
\end{prop}
\begin{proof}
Recall that $x_K^{\S}(\theta) = \frac{\theta}{\norm{\theta}_K}$, which is $0$-homogeneous on $E$. Hence:
\[
x_{A(K)}^{\S}(\theta) = \frac{\theta}{\norm{\theta}_{A(K)}} = \frac{\theta}{\norm{A^{-1} \theta}_K} = A \frac{A^{-1} \theta}{\norm{A^{-1} \theta}_K} = A \circ x^{\S}_K(A^{-1} \theta / |A^{-1} \theta|) ,
\]
and the first identity in (\ref{eq:cov-formula}) follows. Also recall that $x_K^{\S^*}(\theta^*) = \bar D h_K(\theta^*)$, which is $0$-homogeneous on $E^*$. Hence: \[
x_{A(K)}^{\S^*}(\theta^*) = \bar D h_{A(K)}(\theta^*) = \bar D (h_{K} \circ A^*)(\theta^*) = A \bar D h_K(A^* \theta^*) = A \bar D h_K(A^* \theta^* / |A^* \theta^*|) ,
\]
and the second identity in (\ref{eq:cov-formula}) follows.

As for the second part of the proposition, let us only verify (\ref{it:pf2}) (as the verification of (\ref{it:pf1}) is identical). 
Denoting $x^{\S^*}_{aux} := x_K^{\S^*} \circ ((A^{-*})^{(0)})^{-1}$, it follows by usual centro-affine invariance (as in Subsection \ref{subsec:ADG-AI}) that the two hypersurfaces $x^{\S^*}_{aux},x_{A(K)}^{\S^*} : \S^* \rightarrow E$ (equipped with the centro-affine normalization)
 induce exactly the same metric, normal and conormal connections $g,\nabla,\nabla^*$ on $\S^*$, and also the same volume measures $\nu,\nu^*$ whenever $A \in SL(E)$. It remains to note that the two hypersurfaces $x^{\S^*}_{K}, x^{\S^*}_{aux} : \S^* \rightarrow E$ are identical, up to reparametrization of $\S^*$ via $(A^{-*})^{(0)}$; consequently, whenever these two hypersurfaces are equipped with the same normalization (as is our context), their induced differential structures are isomorphic via $(A^{-*})^{(0)}$. 
\end{proof}

It is a good exercise to verify directly that, e.g., $(A^{-*})^{(0)}$ pushes forward $g^{\S^*}_K$ onto $g^{\S^*}_{A(K)}$. 

\subsection{Explicit formulas}

We now calculate the centro-affine differential structures (\ref{eq:structures}). As explained above, it is enough to perform the calculation on a convenient parametrization $M \in \M_K$. The most convenient choice for us is $M = \S^*$, but we also provide the corresponding expressions in our other parametrizations. 
Recall from Section \ref{sec:prelim} the definitions of $V^{\S^*}_K$, $V^{\partial K}_K$ and $V^{\S}_K$, that $D^2 h_K$ denotes the restriction of $\bar D^2 h_K$ onto $T \S^*$, and the discussion regarding induced Euclidean structures.

\subsubsection{$M = \S^*$}

We work on $M = \S^*$ and omit the superscript $\S^*$ (and often the subscript $K$ as well) in our expressions. 
Recall that:
\begin{align*}
x = x_K & :  \S^* \rightarrow \partial K  ~,~  x(\theta^*) = \bar D h_K(\theta^*) , \\
x^* = (x_K)^* = x_{K^*} & :  \S^* \rightarrow \partial K^* ~,~ x^*(\theta^*) = \frac{\theta^*}{h_K(\theta^*)} .
\end{align*}
We perform all calculations at a fixed $\theta^* \in \S^*$. Then:
\begin{align*}
d  x & : T_{\theta^*} \S^* \rightarrow T_x \partial K ~,~ d x(u)  = \bar D_u \bar D h_K  ,\\
d x^* & : T_{\theta^*} \S^* \rightarrow T_{x^*} \partial K^* ~,~  dx^*(v) = \frac{1}{h_K} ( v - v(\log h_K) \theta^* ) . 
\end{align*}
Since $h_K$ is $1$-homogeneous, it follows that $\bar D_{\theta^*} \bar D h_K  = 0$. Consequently:
\begin{equation} \label{eq:gK1}
g_K^{\S^*}(u,v) = \scalar{dx^*(v) , dx(u)} = \frac{\bar D^2 h_K(u,v)}{h_K} =  \frac{D^2 h_K(u,v)}{h_K} ~,~ u,v \in T_{\theta^*} \S^* . 
\end{equation}

For a Euclidean orthonormal basis $e_1,\ldots,e_{n-1}$ in $T_{\theta^*} \S^*$, we have:
\begin{align*}
\frac{d \nu_K}{d \Leb^{\S^*}} & = \nu_K(e_1,\ldots,e_{n-1}) = \Det(dx(e_1),\ldots,dx(e_n),x) \\
& = h_K \Det(dx(e_1),\ldots,dx(e_{n-1}) , \theta^*)   = h_K \det(D^2 h_K) = n \frac{dV^{\S^*}_K}{d \Leb^{\S^*}} ,
\end{align*}
where we used that $P_{(T \partial K)^{\perp}} x = \scalar{\theta^*,x} \theta^* = h_K \theta^*$. Similarly:
\begin{align*}
\frac{d \nu^*_K}{d \Leb^{\S^*}} &  = \nu^*_K(e_1,\ldots,e_{n-1}) = \Det^*(dx^*(e_1),\ldots,dx^*(e_n),x^*) \\
& = \frac{1}{h_K^n} \Det^*( e_1 - e_1(\log h_K) \theta^* , \ldots, e_{n-1} - e_{n-1}(\log h_K) \theta^* , \theta^*) =  \frac{1}{h_K^n} = n \frac{d \; i_*V^{\S}_{K^{\circ}}}{d \Leb^{\S^*}} .
\end{align*}

By the Gauss equation for $x$, we have for $U\in T \S^*, V \in \Gamma^1(T \S^*)$:
\[
\bar D_U \; dx(V) = dx(\nabla_U V) -  g_K(U,V)  x. 
\]
Consequently:
\[
\bar D^2_{U,V} \bar D h_K + \bar D_{\bar D_U V} \bar D h_K = \bar D_{\nabla_U V} \bar D h_K  -  g_K(U,V) x . 
\]
It follows that:
\[
\bar D^2 h_K (\nabla_U V - \bar D_U V, \xi^*) = \bar D^3 h_K(U,V,\xi^*) + g_K(U,V) \scalar{\xi^*,x} \;\;\; \forall \xi^* \in E^* . 
\]
Recalling that $g_K = \frac{\bar D^2 h_K}{h_K}$, $x = \bar D h_K$,   $\bar D_U V = {}^{\S^*} \nabla_U V -\II^{\S^*}(U,V) \theta^*$ and that $\bar D^2 h_K \cdot \theta^* = 0$, it follows that:
\[
g_K(\nabla_U V - {}^{\S^*} \nabla_U V, \xi^*)  = \frac{\bar D^3 h_K(U,V,\xi^*)}{h_K} + g_K(U,V) \scalar{\xi^*, \bar D( \log h_K) } \;\;\; \forall \xi^* \in T_{\theta^*} \S^* .
\]
Introducing a local frame $\{e_1,\ldots,e_{n-1}\}$ on $\S^*$,  it follows that:
\[
(\nabla_U V - {}^{\S^*} \nabla_U V)^i = g_K^{i j} \brac{\frac{\bar D^3 h_K(U,V,e_j)}{h_K} + g_K(U,V) (\log h_K)_j } . 
\]
As expected, this expression depends on third derivatives of $h_K$. 

A much more useful expression is obtained for the conjugate connection. By the Gauss equation for $x^*$ (recall that $\hat S = g$ in the centro-affine normalization):
\[
\bar D_U \; dx^*(V) = dx^*(\nabla^*_U V) - g_K(U,V) x^*  ,
\]
i.e.:
\[
\bar D_U \brac{ \frac{V}{h_K} - \frac{V(h_K)}{h_K^2} \theta^*  } = \frac{1}{h_K} ( \nabla^*_U V - (\nabla^*_U V)(\log h_K) \theta^*) - g_K(U,V) \frac{\theta^*}{h_K(\theta^*)} . 
\]
Consequently, applying the Leibniz rule and multiplying by $h_K$, we have:
\[
-\frac{1}{h_K} U(h_K) V + \bar D_U V - h_K U(V(h_K)/h_K^2) \theta^* - V(\log h_K) \bar D_U \theta^* = \nabla^*_U V - (\nabla^*_U V)(\log h_K) \theta^* - g_K(U,V) \theta^* . 
\]
Recall that $U \in T_{\theta^*} \S^*$ so that $\bar D_U \theta^* = U$. Orthogonally projecting onto $(\theta^*)^{\perp}$, we obtain:
\[
- U(\log h_K) V + {}^{\S^*} \nabla_U V - V (\log h_K) U =  \nabla^*_U V . 
\]
In particular, we see that the conjugate connection only depends on first derivatives of $h_K$, which is already reassuring. 
By projecting onto $\theta^*$ and recalling (\ref{eq:II}), one rederives (\ref{eq:gK1}); for completeness, let us verify this:
\begin{align*}
g_K(u,v) = & \II^{\S^*}(U,V) +h_K U(V(h_K)/h_K^2) -  (\nabla^*_U V)(\log h_K) \\
 = & \II^{\S^*}(U,V) + \frac{{}^{\S^*} \nabla^2_{U,V} h_K}{h_K} + \frac{{}^{\S^*} \nabla_U V (h_K)}{h_K} - 2 \frac{U(h_K) V(h_K)}{h_K^2} \\
& + 2 U(\log h_K) V(\log h_K) - {}^{\S^*} \nabla_U V (\log h_K) \\
 = & \II^{\S^*}(U,V) + \frac{{}^{\S^*} \nabla^2_{U,V} h_K}{h_K} = \frac{\bar D^2 h_K(U,V)}{h_K} ,
\end{align*}
where in the last transition we used (\ref{eq:II-2nd-deriv}) and the fact that $h_K$ is $1$-homogeneous so that $\theta^*(h_K) = h_K$.

We summarize all of these computations in the following:
\begin{prop} \label{prop:S*-param}
The differential structures (\ref{eq:structures}) for the centro-affine normalization of $K$ are given on $\S^*$ by:
\begin{align*}
g^{\S^*}_K & = \frac{D^2 h_K}{h_K} , \\
\nu^{\S^*}_K & =  h_K S_K = h_K \det (D^2 h_K) \Leb^{\S^*} = n V^{\S^*}_K  , \\
(\nu^{\S^*}_K)^* & = \frac{1}{h_K^n} \Leb^{\S^*} = n\; i_* V^{\S}_{K^\circ}  , \\
((\nabla^{\S^*}_K)_U V)^{i} & = ({}^{\S^*} \nabla_U V)^i + (g^{\S^*}_K)^{i j} \brac{\frac{\bar D^3 h_K(U,V,e_j)}{h_K} + g^{\S^*}_K(U,V) (\log h_K)_j } ,\\
(\nabla^{\S^*}_K)^*_U V & =  {}^{\S^*} \nabla_U V - U(\log h_K) V - V (\log h_K) U  
\end{align*}
(for any $U \in T\S^*, V \in \Gamma^1(T \S^*)$ and local frame $\{e_1,\ldots,e_{n-1}\}$ on $\S^*$). \\
In particular, the centro-affine metric $g^{\S^*}_K$ coincides with the metric (\ref{eq:prelim-metric}) introduced in Section \ref{sec:prelim}, and up to normalization, the centro-affine volume-measure $\nu^{\S^*}_K$ coincides with the cone-volume measure $V^{\S^*}_K$. 
\end{prop}

In addition, recall that the following useful properties hold, regardless of parametrization:
\begin{itemize}
\item $\nabla_K \nu_K = \nabla^*_K \nu^*_K = 0$.
\item $K$ is a centro-affine unit-sphere; in particular, the Ricci curvatures of $\nabla_K$ and $\nabla^*_K$ are constant and equal to $n-2$.
\end{itemize}

In our opinion, it is quite remarkable that the Ricci curvature turns out to even just be positive, let alone constant, for the centro-affine connection, in view of the fact that it involves three derivatives of $h_K$. 
It is a good but tedious exercise to verify this for both our connections directly from the associated Christoffel symbols on $\S^*$:
\begin{align}
\label{eq:Christoffel-on-S*} {}^{\nabla^*_K} \Gamma_{ij}^k & = {}^{\S^*} \Gamma_{i j}^k - \delta_{i}^k (\log h_K)_j - \delta_{j}^k (\log h_K)_i  \\
\nonumber {}^{\nabla_K} \Gamma_{ij}^k & ={}^{\S^*} \Gamma_{i j}^k +  g_K^{k p} \brac{\frac{\bar D^3_{ijp} h_K}{h_K} + (g_K)_{ij} (\log h_K)_p} .
\end{align}
Another good exercise is to verify that our connections are indeed $g_K$-conjugates using (\ref{eq:conjugation}). 

\subsubsection{Other parametrizations}

It will be convenient to also work on the parametrization $M = \partial K$ (we henceforth omit the corresponding superscript in our notation). 
Recall that:
\begin{align*}
x = x_K & :  \partial K \rightarrow \partial K  ~,~ x = \Id , \\
x^* = (x_K)^* = x_{K^*} & :  \partial K\rightarrow \partial K^* ~,~ x^*(x) = \bar D \norm{x}_K = \bar D h_{K^{*}}(x) .
\end{align*}
We fix a point on $\partial K$, which by abuse of notation we denote by $x$ (there should be no confusion with the map $x$). 
Hence:
\begin{align*}
d  x & : T_{x} \partial K \rightarrow T_x \partial K ~,~ d x = \Id  ,\\
d x^* & : T_{x} \partial K \rightarrow T_{x^*} \partial K^* ~,~  dx^*(v) = \bar D_v \bar D h_{K^{*}}  . 
\end{align*}
Hence:
\[
g^{\partial K}_K(u,v) := \scalar{dx^*(v) , dx(u)}  = \bar D^2_x h_{K^{*}}(u,v) ~,~ u , v \in T_x \partial K . 
\]
Note that $\bar D^2_x h_{K^{*}} \cdot x = 0$, whereas $u,v \perp \theta^*$; to emphasize this point, we write:
\begin{equation} \label{eq:g-partial-K-1}
g^{\partial K}_K = P_{(\theta^*)^{\perp}} \bar D^2_x h_{K^{*}} P_{(\theta^*)^{\perp}} ,
\end{equation}
where $P_H$ denotes orthogonal projection onto the corresponding subspace $H$. A more convenient expression is derived in:

\begin{lemma}
For all $x \in \partial K$:
\[
g^{\partial K}_K(x) = |x^*| \II^{\partial K}_x  = \frac{\II^{\partial K}_x}{h_K(\theta^*)}   . 
\]
\end{lemma}
\begin{proof}
It is well-known (e.g. \cite[(2.48)]{Schneider-Book-2ndEd}) that $\II^{\partial K}_x = (D^2_{\theta^*} h_K)^{-1} = \frac{1}{|x^*|} (D^2_{x^*} h_K)^{-1}$ (naturally identifying between the corresponding tangent spaces); using this and the duality between $K$ and $K^{*}$, it is a nice exercise to derive the assertion from (\ref{eq:g-partial-K-1}) by verifying that:
\[ P_{(\theta^*)^{\perp}} \bar D^2_x h_{K^{*}} P_{(\theta^*)^{\perp}} = (D^2_{x^*} h_K)^{-1} . 
\]  Alternatively, it is simpler to pull-back $g^{\S^*}_K = \frac{D^2_{\theta^*} h_K}{h_K(\theta^*)}$ via the Gauss map $\partial K \ni x \mapsto \theta^* \in \S^*$ since $d_{\theta^*} x = D^2_{\theta^*} h_K$. Probably the simplest argument is to recall that 
 the centro-affine conormalization of $\partial K$ by $x^*$ coincides with the Euclidean conormalization by $\theta^*$ up to a multiplicative factor of $|x^*|$, and hence the corresponding induced second-fundamental forms $g^{\partial K}_K(x)$ and $\II^{\partial K}_x$ also coincide up to this factor \cite[Proposition 1.23 (ii)]{LSZH-Book}. Note that $|x^*| h_K(\theta^*) = |x^*| \scalar{\theta^*,x} = \scalar{x^*,x} = 1$. 
\end{proof}

We leave the rest of the computations of our structures for the reader, as they will not be needed, and only state them. One can use all of the tools developed in the previous subsections to transfer the information from $\S^*$ to any other $M \in \M_K$: direct computation, pushing-forward via our diffeomorphisms $T^{\S^* \rightarrow M}_K$, conjugation and duality. See also \cite[Proposition 1.23]{LSZH-Book}. 

\begin{prop} \label{prop:partial-K-param}
The differential structures (\ref{eq:structures}) for the centro-affine normalization of $K$ are given on $\partial K$ at $x \in \partial K$ by:
\begin{align*}
g^{\partial K}_K & = P_{(\theta^*)^{\perp}} \bar D^2_x h_{K^{\circ}} P_{(\theta^*)^{\perp}} = |x^*| \II^{\partial K}_x  = \frac{\II^{\partial K}_x}{h_K(\theta^*)} , \\
\nu^{\partial K}_K & = \scalar{\theta^*,x} \H^{n-1}|_{\partial K}(dx) = n V^{\partial K}_K  , \\
(\nu^{\partial K}_K)^* & = \frac{\kappa^{\partial K}_x}{\scalar{\theta^*,x}^n} \H^{n-1}|_{\partial K}(dx) =:  n V^{\partial K}_{K^\circ} , \\
(\nabla^{\partial K}_K)_U V & = {}^{\partial K} \nabla_U V  + g^{\partial K}_K(U,V) P_{(\theta^*)^{\perp}} x  ,\\
((\nabla^{\partial K}_K)^*_U V)^{i} & = ({}^{\partial K} \nabla_U V)^{i} + (g^{\partial K}_K)^{ij} \brac{  \bar D^3 h_{K^*} (U,V,e_j)   +  g^{\partial K}_K(U,V) (h_{K^*})_j } 
\end{align*}
(for any $U \in T_x \partial K, V \in \Gamma^1(T \partial K)$ and local frame $\{e_1,\ldots,e_{n-1}\}$ on $\partial K$).
Here $\kappa^{\partial K}_x = \det \II^{\partial K}_x = 1 / \det(D^2_{\theta^*} h_K)$ denotes the Gauss curvature of $\partial K$ at $x$. 
\end{prop}

\begin{prop} \label{prop:S-param}
The differential structures (\ref{eq:structures}) for the centro-affine normalization of $K$ are given on $\S$ at $\theta \in \S$ by:
\begin{align*}
g^{\S}_K & = g^{\S}_{K^*} =  i_* g^{\S^*}_{K^{\circ}} = \frac{D^2 h_{K^*}}{h_{K^*}} , \\
\nu^{\S}_K & = (\nu^{\S}_{K^*})^*  = i_* (\nu^{\S^*}_{K^{\circ}})^*  =  \frac{1}{h_{K^{*}}^n} \Leb^{\S}  = \frac{1}{\norm{\cdot}_K^n} \Leb^{\S} = n V^{\S}_K  , \\
(\nu^{\S}_K)^* & =  \nu^{\S}_{K^*}  = i_* \nu^{\S^*}_{K^{\circ}} = h_{K^{*}}  S_{K^{*}} = h_{K^{*}} \det (D^2 h_{K^{*}}) \Leb^{\S} = n\; i_* V^{\S^*}_{K^\circ} , \\
(\nabla^{\S}_K)_U V  & = (\nabla^{\S\; *}_{K^*})_U V = (i_* \nabla^{\S^* \; *}_{K^{\circ}})_U V = {}^{\S} \nabla_U V - U(\log h_{K^{*}}) V - V (\log h_{K^{*}}) U , \\
((\nabla^{\S}_K)^*_U V)^i  & = ((\nabla^{\S}_{K^*})_U V)^i  = ((i_* \nabla^{\S^*}_{K^{\circ}})_U V)^i  \\
& =({}^{\S} \nabla_U V)^i + (g^{\S}_{K^{*}})^{i j} \brac{\frac{\bar D^3 h_{K^{*}}(U,V,e_j)}{h_{K^{*}}} + g_{K^{*}}(U,V) (\log h_{K^{*}})_j } 
\end{align*}
(for any $U \in T\S, V \in \Gamma^1(T \S)$ and local frame $\{e_1,\ldots,e_{n-1}\}$ on $\S$). 
\end{prop}

\begin{prop} \label{prop:partial-K*-param}
The differential structures (\ref{eq:structures}) for the centro-affine normalization of $K$ are given on $\partial K^*$ at $x^* \in \partial K^*$ by:
\begin{align*}
g^{\partial K^*}_K & = g^{\partial K^*}_{K^*} = i_* g^{\partial K^{\circ}}_{K^{\circ}} = P_{\theta^{\perp}} \bar D^2_{x^*} h_K P_{\theta^{\perp}} = |x| \II^{\partial K^{*}}_{x^*}  = \frac{\II^{\partial K^{*}}_{x^*}}{h_{K^*}(\theta)} ,\\
\nu^{\partial K^*}_K & = (\nu^{\partial K^*}_{K^*})^* = i_* (\nu^{\partial K^{\circ}}_{K^{\circ}})^* = 
\frac{\kappa^{\partial K^*}_{x^*}}{\scalar{x^*,\theta}^n} \H^{n-1}|_{\partial K^*}(dx^*) = n \; i_*  V^{\partial {K^{\circ}}}_K ,\\
(\nu^{\partial K^*}_K)^* & = \nu^{\partial K^*}_{K^*} = i_* \nu^{\partial K^{\circ}}_{K^{\circ}} =  \scalar{x^*,\theta} \H^{n-1}|_{\partial K^*}(dx^*) = n \; i_*  V^{\partial K^{\circ}}_{K^\circ}  , \\
((\nabla^{\partial K^*}_K)_U V)^i & = ((\nabla^{\partial K^*}_{K^*})^*_U V)^i = ((i_* \nabla^{\partial K^{\circ} \; *}_{K^{\circ}})_U V)^i \\
 & = ({}^{\partial K^*} \nabla_U V)^{i} + (g^{\partial K^*}_K)^{ij} \brac{  \bar D^3 h_{K} (U,V,e_j)   +  g^{\partial K^*}_K(U,V) (h_{K})_j } ,\\
(\nabla^{\partial K^*}_K)^*_U V & = (\nabla^{\partial K^*}_{K^*})_U V = (i_* \nabla^{\partial K^{\circ}}_{K^{\circ}})_U V = 
{}^{\partial K^*} \nabla_U V  + g^{\partial K^*}_K(U,V) P_{\theta^{\perp}} x^*
\end{align*}
(for any $U \in T_{x^*}\partial K^*, V \in \Gamma^1(T \partial K^*)$ and local frame $\{e_1,\ldots,e_{n-1}\}$ on $\partial K^*$).
Here $\kappa^{\partial K^*}_{x^*} = \det \II^{\partial K^*}_{x^*} = 1 / \det(D^2_{\theta} h_{K^*})$ denotes the Gauss curvature of $\partial K^*$ at $x^*$. 
\end{prop}

\subsection{Differential calculus and the Hilbert-Brunn-Minkowski operator} \label{subsec:HBM}

We denote the centro-affine divergence and Hessian operators by:
\begin{align*}
\div^M_K := \div^{\nabla^M_K} & ~,~ (\div^M_K)^* = \div^{(\nabla^M_K)^*} \\
 \Hess^M_K := \Hess^{\nabla^M_K} & ~,~ (\Hess^M_K)^* = \Hess^{(\nabla^M_K)^*},
\end{align*}
 omitting the superscript $M$ when the context is clear.
 
 \begin{lemma} \label{lem:Hess-loc}
 In a local frame on $\S^*$ we have for any $f \in C^2(\S^*)$:
 \[
 (\Hess_K^*)_{ij}  f= {}^{\S^*} \nabla^2_{ij} f + (\log h_K)_i f_j  + (\log h_K)_j f_i . 
 \]
 \end{lemma}
 \begin{proof}
 Recall that:  \[
 (\Hess_K^*)_{ij} f = \partial^2_{ij} f -  {}^{\nabla_K^*} \Gamma_{ij}^k \partial_k f .
 \]
 Plugging the expression for the Christoffel symbols on $\S^*$ derived in Proposition \ref{prop:S*-param} and recorded in (\ref{eq:Christoffel-on-S*}), the assertion immediately follows. 

 \end{proof}

  Recalling the discussion in Subsection \ref{subsec:ADG-calculus}, we denote the corresponding centro-affine Laplacian operators by:
 \[
 \Delta^M_K := \div^M_K \; \grad_{g^M_K} = \tr_{g^M_K} (\Hess^M_K)^* ~,~ (\Delta^M_K)^* := (\div^M_K)^* \; \grad_{g^M_K} = \tr_{g^M_K} \Hess^M_K . 
 \]
 
 Recall that the Hilbert--Brunn--Minkowski operator on $\S^*$ was introduced in Section \ref{sec:prelim} as the weighted Laplacian on $(\S^*,g^{\S^*}_K, V^{\S^*}_K)$, and was denoted by $\Delta_K^{\S^*}$. While the reader may be concerned that there will be some ambiguity due to our identical notation
 for the Hilbert--Brunn--Minkowski and centro-affine Laplacian operators,  
  our most important observation in this section is that there is no ambiguity. We omit the particular parametrization $M$, as it is irrelevant in the statement below. 

\begin{thm} \label{thm:DeltaK}
The centro-affine Laplacian $\Delta_K$ coincides with the Hilbert-Brunn-Minkowski operator. 
\end{thm}
\begin{proof}
We verify the claim on $\S^*$. By Lemma \ref{lem:Hess-loc}:
\[
\Delta_K^{\S^*} f =  g_K^{ij}(\Hess_K^* f)_{ij} = g_K^{ij} \brac{ {}^{\S^*} \nabla^2_{ij} f + (\log h_K)_i f_j + (\log h_K)_j f_i } . 
\]
Recalling that centro-affine metric $g_K = g^{\S^*}_K$ coincides with the metric (\ref{eq:prelim-metric}), we confirm that the right-hand-side coincides with the Hilbert--Brunn--Minkowski operator (\ref{eq:HBM}). 
\end{proof}

Theorem \ref{thm:DeltaK} finally gives a satisfactory explanation for the centro-affine equivariance property of the Hilbert--Brunn--Minkowski operator, originally observed in \cite[Section 5.2]{KolesnikovEMilman-LocalLpBM} following a lengthy computation, but now an immediate consequence of Proposition \ref{prop:CA-Invariance}. 

Recall from Subsection \ref{subsec:ADG-calculus} that (regardless of the parametrization $M$):
\begin{equation} \label{eq:CA-Dirichlet}
\int (-\Delta_K f) h \; d\nu_K = \int g_K(\grad_{g_K} f, \grad_{g_K} h) \; d\nu_K = \int f (-\Delta_K h) \; d\nu_K ,
\end{equation}
for all $f,h \in C^2(M)$. Recall from Proposition \ref{prop:S*-param} that on $\S^*$, $\nu^{\S^*}_K$ coincides (up to normalization) with the cone-volume measure $V^{\S^*}_K$. In \cite[Section 5.1]{KolesnikovEMilman-LocalLpBM}, we had originally (implicitly) identified the metric $g_K^{\S^*}$ by performing the intergration-by-parts in (\ref{eq:CA-Dirichlet}) with respect to $V_K$ and computing the Dirichlet form, thereby interpreting the Hilbert--Brunn--Minkowski operator as the weighted Laplacian on $(\S^* , g^{\S^*}_K , V^{\S^*}_K)$. However, it was not entirely clear whether the choice of measure $V^{\S^*}_K$ and thus construction of the metric $g^{\S^*}_K$ are canonical, or what is the direct relation between these two objects (as in general $V_K$ is not the Riemannian volume measure for $g_K$); we now finally have a satisfactory answer coming from the centro-affine normalization.

Consequently, regardless of the parametrization $M$, $-\Delta_K$ uniquely extends to a self-adjoint positive semi-definite operator on $L^2(\nu_K)$ with domain $H^2$ (the Sobolev space on $M$), as explained in \cite[Section 5.1]{KolesnikovEMilman-LocalLpBM}. Its spectrum $\sigma(-\Delta_K)$ is thus inherently centro-affine invariant, and may be studied regardless of parametrization. The spectrum is discrete, consisting of a countable sequence of eigenvalues of finite multiplicity starting at $0$ and tending to $\infty$. The first (trivial) eigenvalue $\lambda_0(-\Delta_K)$ is zero, corresponding to the constant eigenfunctions. As shown by Hilbert \cite{BonnesenFenchelBook,KolesnikovEMilman-LocalLpBM}, the next eigenvalue $\lambda_1(-\Delta_K)$ is $n-1$, and this fact is equivalent to the classical Brunn--Minkowski inequality; moreover, Hilbert showed that the multiplicity of the eigenvalue $n-1$ is precisely $n$. We will give a new proof of both of these statements 
 in the next section using Lichnerowicz's method and Bochner's formula, utilizing the fact that $\partial K$ is a centro-affine sphere having constant centro-affine Ricci curvature equal to $n-2$. 

\medskip

In Hilbert's original definition of his differential operator, the eigenfunctions corresponding to the first non-trivial eigenvalue $\lambda_1 = n-1$ on $\S^*$ were (the restriction to $\S^*$ of) linear functionals on $E^*$. However, with our definition of $\Delta_K$, originating in \cite{KolesnikovEMilman-LocalLpBM} and further studied in \cite{EMilman-Isospectral-HBM}, the corresponding eigenfunctions are the $K$-adapted linear functions:
\[
\lin^{\S^*}_{K,\xi} := \frac{\scalar{\cdot, \xi}}{h_K}  ~,~ \xi \in E . 
\]
When $K$ is a centered Euclidean ball (or ellipsoid), these coincide with the usual linear functionals, but not in general. While this originally appeared to us to be a caveat of our definition (compared to the one used by Hilbert), we now observe that this is in fact very natural. Indeed, the natural extension from a Euclidean ball to a general $K$ should be to restrict the linear functionals on $E^*$ to $\partial K^*$ instead of $\S^*$. In a parametrization-free language this means using the conormal $x^*_K$ (which on $\partial K^*$ is just the identity map, and so $\scalar{x^*_K,\xi}$ are just linear functions on $\partial K^*$):

\begin{prop} \label{prop:first-eigenfuncs}
\begin{equation} \label{eq:lin-is-x*}
\lin^{\S^*}_{K,\xi} = \scalar{(x^{\S^*}_K)^*,\xi} \;\;\; \forall \xi \in E ,
\end{equation}
and regardless of parametrization:
\begin{align}
 \label{eq:Hess-of-x*}
\Hess^*_K x^*_K & = - x^*_K  g_K, \\
\label{eq:first-eigenfunctions}
-\Delta_K \scalar{x^*_K,\xi} & = (n-1) \scalar{x^*_K,\xi} \;\;\; \forall \xi \in E .
\end{align}
\end{prop}
\begin{proof}
Recalling that $(x^{\S^*}_K)^*(\theta^*) = \frac{\theta^*}{h_K(\theta^*)}$ and the definition of $\lin^{\S^*}_{K,\xi}$, (\ref{eq:lin-is-x*}) is immediate. 
The property (\ref{eq:Hess-of-x*}) is a direct consequence of our centro-affine structure equations (\ref{eq:CA-structure}).
Tracing with respect to $g_K$ immediately yields:
\[
\Delta_K x^*_K = \tr_{g_K} \Hess^*_K  x^*_K = -g^{ij} g_{ij} x^*_K = -(n-1) x^*_K . 
\]
\end{proof}

We will see in Theorem \ref{thm:eigenspace} that (\ref{eq:first-eigenfunctions}) is in fact equivalent to the a-priori stronger property (\ref{eq:Hess-of-x*}). 
One can also easily calculate expressions such as $\Delta_K (x^*_K)^p$ and $|\grad_{g_K} (x^*_K)^p|^2$ using the above calculus, which is originally what led us to the key calculation in \cite{EMilman-Isospectral-HBM} (the calculation there was rederived without referring to the centro-affine normalization).

\subsection{Self-Duality} \label{subsec:CA-SelfDuality}

We conclude this section by reflecting a bit more on the remarkable self-duality property of the centro-affine metric $g_K$:
\[
 g^{i(M^*)}_{K^\circ} =  i_* g^{M^*}_{K^*} =  i_* g^{M^*}_{K} = (i \circ \D_K)_* g^M_{K} \;\;\; \forall M \in \M_K . 
\]
Since $i(M^*) = M$ for $M \in \{\S , \S^*\}$, we see that $i \circ \D_K$ is an isometry between $(\S^* , g^{\S^*}_K)$ and $(\S^* , g^{\S^*}_{K^{\circ}})$. Consequently, any geometric quantity which is encoded by the centro-affine metric is the same for $K$ and $K^{\circ}$. Below are a few examples and many questions; we omit the reference to the parametrization $M$ when it is irrelevant.  
\begin{itemize}
\item The Riemannian volume measure $\nu^{\S^*}_{g_K}$ is (up to a constant) the well-known centro-affine surface-area measure $\Omega^{\S^*}_{n,K}$ \cite{Hug-AffineSurfaceAreaOfPolar,Laugwitz-DifferentialGeometryBook}, which coincides with the $L^p$-affine surface-area measure of Lutwak \cite{Lutwak-Firey-Sums-II} for $p=n$:
\[
\frac{d\nu^{\S^*}_{g_K}}{d \Leb^{\S^*}} = \sqrt{\det g_K}  = \sqrt{\det (D^2 h_K / h_K)} = \frac{\sqrt{\det D^2 h_K}}{h_K^{\frac{n-1}{2}}} =: n \frac{d \Omega^{\S^*}_{n,K}}{d \Leb^{\S^*}} . 
\]
This can also be seen from (\ref{eq:conjugate-measures}):
\begin{align}
\label{eq:vol-product} \brac{\frac{ d\nu^{\S^*}_{g_K}}{d\Leb^{\S^*}}}^2 & = \frac{d \nu^{\S^*}_K}{d\Leb^{\S^*}} \frac{d (\nu^{\S^*}_K)^*}{d\Leb^{\S^*}} = \frac{n \; dV^{\S^*}_K}{d\Leb^{\S^*}}  \frac{n \; d \; i_*V^{\S}_{K^{\circ}}}{d \Leb^{\S^*}} \\
\nonumber & = \frac{ h_K dS_K}{d\Leb^{\S^*}} \frac{1}{h_K^{n}} = \frac{\det D^2 h_K}{h_K^{n-1}} . 
\end{align}
By Subsection \ref{subsec:CA-invariance}, the centro-affine surface-area measure is indeed centro-affine invariant, i.e. $A_* \nu^{\partial K}_{g_K} = \nu^{\partial A(K)}_{g_{A(K)}}$ for all $A \in GL(E)$. 
Since $i \circ \D_K : \S^* \rightarrow \S^*$ pushes forward $\nu^{\S^*}_{g_K}$ onto $\nu^{\S^*}_{g_{K^{\circ}}}$, both measures have the same total mass, which up to normalization is the centro-affine surface-area $\Omega_n$:
\[
\Omega_n(K) = \frac{1}{n} \snorm{\nu_{g_K}} = \frac{1}{n} \snorm{\nu_{g_{K^{\circ}}}} = \Omega_n(K^{\circ}) . 
\]
Note that (\ref{eq:vol-product}) and the Cauchy-Schwarz inequality immediately yield:
\begin{equation} \label{eq:vol-product-lower-bound}
\Omega_n(K)^2 \leq V(K) V(K^{\circ}) . 
\end{equation}
Unfortunately, $\Omega_n(K_i)$ converges to $0$ whenever $K_i$ converge in the Hausdorff metric to a polytope, and so this cannot be used to effectively lower bound the volume-product on the right-hand side. All of this is of course well-known.  
\item Let $d_{g_K}$ be the induced geodesic distance on $(\S^*,g_K)$. Given $p > -n+1$, define the quantity:
\[
W_p(K) := \brac{\int_{\S^*} \int_{\S^*} d^p_{g_K}(x,y) \nu_{g_K}(dx) \nu_{g_K}(dy)}^{1/p} . 
\]
Then $W_p(K) = W_p(K^{\circ})$ for all $p > -n+1$. 
\item 
$W_\infty(K) := \text{diam}(\S^*, d_{g_K})$ also satisfies $W_\infty(K) = W_\infty(K^{\circ})$. What is the geometric meaning of this quantity? Is this related to the resolution of Sch\"affer's conjecture by \'Alvarez-Paiva \cite{AlvarezPaiva-SelfDualGirth} (see also Faifman \cite{Faifman-DualGirthOnGrassmanians}), identifying between the girths of $K$ and $K^{\circ}$? This seems unlikely, since the girth is a Finslerian notion. Note that it is not hard to show that $W_\infty(K+L) \leq W_\infty(K) + W_\infty(L)$. 
\item Consider $\Sigma(K) := \sigma(-\Delta_{g_K})$, the spectrum of the Laplace-Beltrami operator on $(\S^*,g_K)$. Then $\Sigma(K) = \Sigma(K^{\circ})$. Is there a geometric meaning to (at least the first) eigenvalues? 
\item Let $F_{g_K}$ denote any non-trivial real-valued function of the Riemann curvature tensor $\RR_{g_K}$ and the metric tensor $g_K$ -- such as the scalar-curvature $\mathfrak{s}_{g_K}$, the Hilbert-Schmidt norm of the Ricci tensor, or some other function of the sectional curvatures. Define $\mathfrak{F}_p(K) = (\int |F_{g_K}|^p \nu_{g_K})^{1/p}$. Then $\mathfrak{F}_p(K) = \mathfrak{F}_p(K^{\circ})$. Of particular interest is the average scalar curvature:
\[
\mathfrak{S}(K) = \int \mathfrak{s}_{g_K} d\nu_{g_K} . 
\]
Does this quantity have a natural geometric meaning? 
One can show that $\mathfrak{s}_{g_K} \geq (n-1)(n-2)$ in the centro-affine normalization \cite[formula (39)]{Opozda-BochnerMethod}, and so $\mathfrak{S}(K) \geq (n-1) (n-2) \Omega_n(K)$. It would be interesting to try and lower bound the volume product $V(K) V(K^{\circ})$ by a function of $\mathfrak{S}(K)^2 = \mathfrak{S}(K) \mathfrak{S}(K^{\circ})$, in view of (\ref{eq:vol-product-lower-bound}). 
\item Is it true that if $(\S^*,g_{K_1})$ and $(\S^*, g_{K_2})$ are isometric then $K_1$ and $K_2$ are congruent up to a centro-affine transformation and polarity? It is not hard to show that if $g_{K_1} = g_{K_2}$ then necessarily $K_2 = c K_1$ for some $c > 0$. 
\end{itemize}

\section{Bochner formula} \label{sec:Bochner}

\subsection{Asymmetric Bochner formula}

The classical Bochner formula \cite[Chapter 9]{PetersenBook3rdEd} states that if $f$ is a smooth function on a Riemannian manifold $(M,g)$ then:
\[
\frac{1}{2} \Delta_g |\grad_g f|^2 = g(\grad_g f , \grad_g \Delta_g f) + \Ric_g (\grad_g f , \grad_g f) + \norm{\Hess_g f}_g^2 . 
\]
In this formula, all higher order differential objects are computed with respect to the Levi-Civita connection $\nabla^g$ (which, recall, is the unique torsion-free connection which is in addition metric, i.e. $\nabla^g g = 0$); namely, $\Delta_g = \div^{\nabla^g} \grad_g$, $\Hess_g f = \nabla^g df$ and $\Ric_g$ denotes the Ricci curvature of $\nabla^g$. We use the notation:
\[
 |X|^2 = g(X,X)  ~,~ \norm{A}_g^2 = g(A,A) = \scalar{A,A}_g =  g^{ij} A_{jk} g^{kl} A_{li},
\]
for the (squared) Riemannian length of a vector-field $X$ and Riemannian Hilbert-Schmidt norm of a symmetric $(0,2)$ tensor $A$, respectively. 

The Bochner formula is local, and applies to arbitrary vector-fields $X$ so that $\nabla^g X$ is a $g$-symmetric $(1,1)$ tensor:
\[
g( \nabla^g_Y X , Z) = g( \nabla^g_Z X , Y) \;\;\; \forall Y,Z \in T_p M \;\; \forall X \in \Gamma^1(T M) 
\]
 (and hence, locally, $X = \grad_g f$ for some function $f$). After polarization, it is equivalent to the following version, valid for any vector-fields $X,Y$ so that $\nabla^g X$ and $\nabla^g Y$ are $g$-symmetric $(1,1)$ tensors:
\begin{equation} \label{eq:classical-Bochner}
\Delta_g \; g(X,Y) = X(\div^{\nabla^g} Y) + Y(\div^{\nabla^g} X) + 2 \Ric_g(X,Y) + 2 \tr(\nabla^g X \circ \nabla^g Y) ,
\end{equation}
where $\nabla X \circ \nabla Y$ denotes the composition as $(1,1)$ tensors, so that in a local coordinates:
\[
\tr( \nabla X \circ \nabla Y) = \nabla_j X^i \nabla_i Y^j . 
\]
 An inspection of the proof of the Bochner formula confirms that it makes use of both the torsion-free and metric properties of the Levi-Civita connection. 

However, it is possible to give a version of Bochner's formula which does not rely on any of these properties, holds for all vector-fields, and in fact does not require a Riemannian metric $g$ at all. 
\begin{prop}[Asymmetric Bochner formula] \label{prop:Bochner}
For any affine connection $\nabla$ on $M$ and (smooth) vector-fields $X,Y$:
\begin{equation} \label{eq:Bochner}
\div^{\nabla}(\nabla_X Y) = X(\div^{\nabla} Y) + \Ric(Y,X) + \tr( \nabla X \circ \nabla Y) .
\end{equation}
\end{prop}

A proof may be found in \cite[Lemma 2.2]{Opozda-BochnerMethod}. For completeness, we provide a simple proof using local coordinates.

\begin{proof}
Denoting as customary the components of the curvature tensors $\RR$  and $\Ric$ by: \[
\RR(\partial_k , \partial_l) \partial_j = \RR_{jkl}^i \partial_i ~,~ \Ric_{jk} = \Ric(\partial_j , \partial _k) ,
\]
so that $\Ric_{jk} = \RR^{i}_{k i j}$, we have in local coordinates: \begin{align*}
\nabla_j (X^i \nabla_i Y^j)  & = \nabla_j X^i \nabla_i Y^j + X^i \nabla_j \nabla_i Y^ j \\
& = \nabla_j X^i \nabla_i Y^j + X^i \nabla_i \nabla_j Y^ j + X^i \RR^j_{i j l} Y^l \\
& = \nabla_j X^i \nabla_i Y^j + X^i \nabla_i \nabla_j Y^ j + \Ric_{l i} X^i  Y^l .
\end{align*}
\end{proof}

Our ``asymmetric" nomenclature stems from the fact that the formula is no longer symmetric in $X,Y$, and moreover, does not assume any symmetry -- not from $\nabla X,\nabla Y$ (these are mixed tensors so in any case symmetry does not make any sense), nor from the affine connection $\nabla$. The classical Bochner formula (\ref{eq:classical-Bochner}) is obtained from Proposition \ref{prop:Bochner} by applying it to the Levi-Civita connection and symmetrizing (i.e. interchanging the roles of $X,Y$ and summing). Indeed, note that by the metric property of $\nabla^g$ and as $\nabla^g X$ and $\nabla^g Y$ in (\ref{eq:classical-Bochner}) are assumed $g$-symmetric:
\[
\nabla^g_Z  \; g(X,Y) = g(\nabla^g_Z X , Y) + g(X , \nabla^g_Z Y) = g(\nabla^g_Y X,Z) + g(\nabla^g_X Y , Z) ,
\]
and hence
\[
 \grad_g \; g(X,Y) = \nabla^g_X Y + \nabla^g_Y X.
\]
 The torsion-free property assures that $\Ric_g$ is symmetric, and (\ref{eq:classical-Bochner}) immediately follows.

\subsection{Centro-affine Bochner formula}

Just as with the classical Bochner formula, integrating the asymmetric Bochner formula (\ref{eq:Bochner}) with respect to an invariant measure $\nu$ yields an integrated Bochner identity \cite[Theorem 9.9]{Opozda-BochnerMethod}. Applying this to our centro-affine differential structures, we immediately deduce:

\begin{thm}[Centro-affine Bochner formula] \label{thm:magic}
Let $K \in \K^{\infty}_+$ be a smooth convex body with strictly positive curvature which contains the origin in its interior. Let $K$ be equipped with the centro-affine normalization. Then, regardless of parametrization $M \in \M_K$, for any function $f \in C^2(M)$:
\begin{equation} \label{eq:magic}
\int (\Delta_K f)^2 d\nu_K - \int \norm{\Hess^*_K f}_{g_K}^2 d\nu_K = (n-2) \int |\grad_{g_K} f|^2 d\nu_K . 
\end{equation}
\end{thm}
\begin{proof}
By approximation, it is enough to prove the identity for smooth functions $f$. We abbreviate $\nabla = \nabla_K$.
Recall that:
\[
\nabla_i (\grad_{g_K} f)^j = \nabla_i (g^{jk}_K f_k) = g^{j k}_K \nabla^*_i f_k = g^{jk}_K (\Hess^*_K)_{ki} f ,
\]
and hence:
\[
\tr(\nabla \grad_{g_K} f \circ \nabla \grad_{g_K} f) = \norm{\Hess^*_K f}_{g_K}^2 . 
\]
Let $\Ric_K$ denote the Ricci curvature of the centro-affine connection. Recall that $\partial K$ is always a centro-affine ($(n-1)$-dimensional) unit-sphere and thus:
\[
\Ric_K  = (n-2) g_K .
\]
Setting $X = Y = \grad_{g_K} f$ and applying the asymmetric Bochner formula (\ref{eq:Bochner}) for $\nabla = \nabla_K$, we deduce that:
\[
\div^{\nabla_K} (\nabla_X X) = \grad_{g_K} f (\Delta_K f) + (n-2) |\grad_{g_K} f|^2 + \norm{\Hess^*_K f}_{g_K}^2 .
\]
Finally, recall that $\nabla_K \nu_K = 0$ (since the centro-affine normalization is a relative normalization). Consequently, integrating with respect to $\nu_K$ as in Subsection \ref{subsec:ADG-calculus}, the integral on left-hand-side above vanishes, and the first term on the right integrates by parts, yielding the asserted:
\[
0 = - \int (\Delta_K f)^2 d\nu_K + (n-2) \int |\grad_{g_K} f|^2 d\nu_K + \int \norm{\Hess^*_K f}_{g_K}^2 d\nu_K .
\]
\end{proof}

It may also be insightful to give an alternative proof of the centro-affine Bochner formula, not relying on centro-affine differential geometry, which is how we originally discovered it (thus realizing that there must be some underlying Bochner formula and constant Ricci curvature). We were informed by Bo'az Klartag and by Ramon van Handel that they also observed (unpublished argument) a similar, ultimately equivalent, formula, without any explicit reference to $g_K$, $V_K$, $\Hess_K^*$, $\Ric_K$, nor their centro-affine geometric interpretation. 

\begin{proof}[Alternative proof of Theorem \ref{thm:magic}]
Given a smooth $f$ on $\S^*$, extend it as a $0$-homogeneous function on $E^* \setminus \{0\}$ and define the following $(1,1)$ tensor in a local frame on $\S^*$:
\[
A^i_k  := ((D^2 h_K)^{-1})^{ij} D^2_{jk} (f h_K) = g_K^{ij} \frac{D^2_{jk} (f h_K)}{h_K} . 
\]
By the Leibniz rule:
\begin{equation} \label{eq:D2fh}
\frac{D^2_{jk} (f h_K)}{h_K}  = D^2_{jk} f + (\log h_K)_j f_k + (\log h_K)_k f_j + f (g_{K})_{jk} , 
\end{equation}
and since $D^2_{jk} f = {}^{\S^*} \nabla^2_{jk} f$ as $f$ is $0$-homogeneous, it follows by Lemma \ref{lem:Hess-loc} that:
\begin{equation} \label{eq:magic-1}
A^i_k = g_K^{ij} (\Hess^*_K f)_{jk}  +  f \delta^i_k .
\end{equation}
From here on, we can use (\ref{eq:magic-1}) as the definition of $\Hess^*_K f$ on $\S^*$.

Set $m = n-1$. Denoting by $\lambda = (\lambda_1,\ldots,\lambda_{m})$ the eigenvalues of the symmetric matrix:
\[
B := (D^2 h_K)^{-1/2} D^2 (f h_K) (D^2 h_K)^{-1/2} 
\]
at a fixed point $\theta^* \in \S^*$, we clearly have:
\begin{equation} \label{eq:magic-2}
\tr(A^2) = \tr(B^2) =  \sum_{i=1}^m \lambda_i^2 =  m^2 \mathfrak{e}_1^2(\lambda) - m (m-1) \mathfrak{e}_2(\lambda) , 
\end{equation}
where $\mathfrak{e}_1(\lambda) := \frac{1}{m} \sum_{i=1}^m \lambda_i$ and $\mathfrak{e}_2(\lambda) := \frac{1}{m (m-1)} \sum_{1 \leq i \neq j \leq m} \lambda_i \lambda_j$ are the first two normalized symmetric polynomials in $\lambda$. 
Note that:
\[
\mathfrak{e}_1(\lambda) = \frac{1}{m} \tr(B) = \frac{1}{m} \tr(A) = \frac{1}{m} (\Delta_K f +  m f) = : \tilde L_K f . 
\]
In addition (see e.g. \cite[formula (6)]{AFO-MixedDiscriminants}):
\[
\mathfrak{e}_2(\lambda) = D_m(B,B,\Id,\ldots,\Id) = \frac{D_m( D^2 (f h_K) , D^2 (f h_K)  ,  D^2 h_K, \ldots,  D^2 h_K)}{\det  D^2 h_K} =: \frac{S_K(f,f)}{S_K} ;
\]
here $D_m$ denotes the mixed discriminant of an $m$-tuple of $m$ by $m$ matrices, and we use the notation $S_K(f,f) = S_K(f;2)$, $S_K(f) = S_K(f;1)$ and $S_K = S_K(f;0) = \det D^2 h_K$, where:
\[
S_K(f;p) = D_m(\underbrace{D^2 (f h_K) , \ldots, D^2 (f h_K)}_{\text{$p$ times}}, \underbrace{ D^2 h_K , \ldots, D^2 h_K}_{\text{$m-p$ times}}) .
\]
By \cite[Subsection 5.1]{KolesnikovEMilman-LocalLpBM}, $\tilde L_K f = \frac{S_K(f)}{S_K}$, and by \cite[Lemma 4.1]{KolesnikovEMilman-LocalLpBM} we can integrate by parts:
\[
\int_{\S^*} \frac{S_K(f,f)}{S_K} dV_K = \frac{1}{n} \int_{\S^*} h_K S_K(f,f) d\Leb = \frac{1}{n} \int_{\S^*} f h_K S_K(f) d\Leb = \int_{\S^*} f \frac{S_K(f)}{S_K} dV_K = \int_{\S^*} f (\tilde L_K f) dV_K . 
\]

\medskip

We now compute $\int_{\S^*} \tr(A^2) dV_K$ in two different manners. On one hand, by (\ref{eq:magic-1}):
\[
\int_{\S^*} \tr(A^2) dV_K = \int_{\S^*} \norm{\Hess^*_K f}_{g_K}^2  dV_K + 2 \int_{\S^*} f (\Delta_K f) dV_K + m \int_{\S^*} f^2 dV_K . 
\]
On the other hand, by (\ref{eq:magic-2}) and the subsequent calculations:
\begin{align*}
& \int_{\S^*} \tr(A^2) dV_K = \int_{\S^*} (\Delta_K f + m f)^2 dV_K - m (m-1) \int_{\S^*} f (\tilde L_K f) dV_K \\
& = \int_{\S^*} (\Delta_K f)^2 dV_K + 2 m \int_{\S^*} f (\Delta_K f) dV_K + m^2 \int_{\S^*} f^2 dV_K - (m-1) \int_{\S^*} (f \Delta_K f + m f^2) dV_K . 
\end{align*}
Comparing both expressions, we obtain:
\[
\int_{\S^*} (\Delta_K f)^2 dV_K - \int_{\S^*}  \norm{\Hess^*_K f}_{g_K}^2 dV_K = -(m-1) \int_{\S^*} f (\Delta_K f) dV_K = (m-1) \int_{\S^*} |\grad_{g_K} f|^2 dV_K . 
\]
\end{proof}

\subsection{Proof of the Brunn-Minkowski inequality \`a-la Lichnerowicz} \label{subsec:BM}

A classical theorem of Lichnerowicz \cite{LichnerowiczBook} (see also \cite{KolesnikovEMilmanReillyPart1}) states that the first non-zero eigenvalue $\lambda_1(-\Delta_g)$ of the Laplace-Beltrami operator $-\Delta_g$ on a closed connected $(n-1)$-dimensional Riemannian manifold $(M^{n-1},g)$ with $\Ric_g \geq \kappa g$ (as $(0,2)$-tensors) and $\kappa > 0$ satisfies:
\[
\lambda_1(-\Delta_g) \geq \frac{n-1}{n-2} \kappa . 
\]
The proof is immediate after applying the integrated Bochner formula to the first eigenfunction $\varphi_1$, and using Cauchy-Schwarz to relate between $\norm{\Hess_g f}_g^2$ and $(\Delta_g f)^2$. 

Repeating the exact same argument verbatim for the centro-affine normalization on $\partial K$ immediately gives a new proof of the Brunn--Minkowski inequality (note that in our case $\kappa = n-2$). By approximation and translation, it is enough to consider smooth strictly convex bodies in $\R^n$ having their origin in the interior $K \in \K^\infty_+$, and as observed by Minkowski and Hilbert, it is enough to derive the inequality in its infinitesimal (or local) form (see \cite[Section 5]{KolesnikovEMilman-LocalLpBM}):
\begin{equation} \label{eq:local-BM}
 \int z \; d\nu_K = 0 \;\; \Rightarrow \;\; \int |\grad_{g_K} z|^2 d\nu_K = \int (-\Delta_K z) z \; d\nu_K \geq  (n-1) \int z^2 d\nu_K \;\;\; \forall z \in C^2  ,
\end{equation}
which exactly means that:
\begin{equation} \label{eq:lambda1-BM}
\lambda_1(-\Delta_K) \geq n-1 . 
\end{equation}
It is well-known that (\ref{eq:local-BM}) is equivalent to its dual form (via $z = \Delta_K f$):
\begin{equation} \label{eq:local-BM-dual}
\int (\Delta_K f)^2 d\nu_K \geq (n-1) \int |\grad_{g_K} f|^2 d\nu_K = (n-1) \int (-\Delta_K f) f \; d\nu_K \;\;\; \forall f \in C^2 . 
\end{equation}
This may be seen by either testing the first eigenfunction $\varphi_1$ and using $-\Delta_K \varphi_1 = \lambda_1 \varphi_1$, or by using the spectral theorem for  $-\Delta_K \geq 0$ (which is self-adjoint on $L^2(\nu_K)$) and noting that (\ref{eq:local-BM-dual}) is equivalent to $P_2(-\Delta_K) \geq 0$ for the polynomial $P_2(t) = t^2 - (n-1) t$. 

\begin{thm}[Local Brunn--Minkowski inequality] \label{thm:BM}
For all $K \in \K^{\infty}_+$, the local Brunn--Minkowski inequality (\ref{eq:local-BM-dual}) holds. 
\end{thm}
\begin{proof}[Proof \`a-la Lichnerowicz]
Given $f \in C^2$, we have by Cauchy-Schwarz:
\begin{equation} \label{eq:Lich-CS}
\norm{\Hess_K^* f}_{g_K}^2 \geq \frac{1}{n-1} (\tr_{g_K} \Hess_K^* f)^2 = \frac{1}{n-1} (\Delta_K f)^2 . 
\end{equation}
Plugging this into the centro-affine Bochner formula (\ref{eq:magic}), we obtain:
\[
\frac{n-2}{n-1} \int (\Delta_K f)^2 d\nu_K \geq  \int (\Delta_K f)^2 d\nu_K - \int \norm{\Hess_K^* f}_{g_K}^2 d\nu_K = (n-2) \int |\grad_{g_K} f|^2 d\nu_K  ,
\]
yielding (\ref{eq:local-BM-dual}). 
\end{proof}

The above proof also immediately reveals the corresponding equality conditions, yielding  
a characterization of the eigenspace of $-\Delta_K$ corresponding to $\lambda_1 = n-1$, a result due to Hilbert. 

\begin{thm}[Equality in local Brunn-Minkowski inequality] \label{thm:eigenspace}
For all $K \in \K^{\infty}_+$, the eigenspace of $-\Delta_K$ corresponding to the first non-zero eigenvalue $\lambda_1 = n-1$ is precisely $n$-dimensional, spanned by $\{ \scalar{x^*_K , \xi} \; ;\; \xi \in E \}$. 
\end{thm}
\begin{proof}We've already seen in Proposition \ref{prop:first-eigenfuncs} that $\scalar{x^*_K,\xi}$ are eigenfunctions corresponding to $\lambda_1 = n-1$. To show the converse, let $-\Delta_K \varphi_1 = (n-1) \varphi_1$, and hence equality holds in (\ref{eq:local-BM-dual}). Inspecting the proof of Theorem \ref{thm:BM}, it follows that we must have equality in the Cauchy-Schwarz inequality (\ref{eq:Lich-CS}) $\nu_K$-a.e., and hence by continuity (as $\nu_K$ is of full support), at every point. It follows that at every point $p$, $\Hess^*_K \varphi_1$ must be a multiple of $g_K$:
\[
\Hess^*_K \varphi_1 = c(p) g_K = \frac{\tr_{g_K} \Hess^*_K \varphi_1}{n-1} g_K = \frac{\Delta_K \varphi_1}{n-1} g_K = - \varphi_1 g_K . 
\]
It is now convenient to use the parametrization on $M = \S^*$. It follows by (\ref{eq:D2fh}) and Lemma \ref{lem:Hess-loc} that if we extend $\varphi_1$ as a $0$-homogeneous function on $E^* \setminus \{0\}$, we have:
\[
\frac{\bar D^2 (\varphi_1 h_K)}{h_K}  = \Hess^*_K \varphi_1 + \varphi_1 g_K = 0 . 
\]
Consequently, $\varphi_1 h_K$ must be an affine function on $E^*$. But since $h_K$ is $1$-homogeneous it vanishes at the origin, and we deduce that $\varphi_1 h_K$ must be a linear function $\scalar{\cdot,\xi}$ for some $\xi \in E$. Hence $\varphi_1 = \lin_{K,\xi} = \scalar{x_K^*,\xi}$, as asserted. 
\end{proof}

\subsection{Equivalent formulations of the $L^p$-Brunn--Minkowski conjecture}

Armed with the centro-affine Bochner formula, we can now give several equivalent formulations of the even local $L^p$-Brunn--Minkowski inequality, which is conjectured to hold for all $K \in \K^\infty_{+,e}$ and $p \in [0,1)$ (and the validity of which for all $K \in \K^\infty_{+,e}$ is equivalent to the validity of the global $L^p$-Brunn--Minkowski inequality (\ref{eq:Lp-BM})  by \cite{Putterman-LocalToGlobalForLpBM}). 

\begin{thm} \label{thm:equivalences}
The following statements are equivalent for a given $K \in \K^\infty_{+,e}$ and $p \leq 1$: 
\begin{enumerate}
\item \label{it:formulation1}
The even local $L^p$-Brunn--Minkowski conjectured inequality:
\begin{equation} \label{eq:local-Lp-BM}
 \!\!\!\!\!\!\!\!\!\! \int f \; d\nu_K = 0 \;\; \Rightarrow \;\; \int |\grad_{g_K} f|^2 d\nu_K = \int (-\Delta_K f) f \; d\nu_K  \geq (n-p) \int f^2 d\nu_K  \;\;\; \forall f  \in C^2_e  . 
\end{equation}
\item \label{it:formulation2}
\[
\int (\Delta_K f)^2 d\nu_K  \geq (n-p) \int |\grad_{g_K} f|^2 d\nu_K = (n-p) \int (-\Delta_K f) f d\nu_K  \;\;\; \forall f  \in C^2_e .
\]
\item \label{it:formulation3}
\[
\int \norm{\Hess_K^* f}_{g_K}^2 d\nu_K \geq \frac{2-p}{n-p} \int  (\Delta_K f)^2 d\nu_K  \;\;\; \forall f \in C^2_e . 
\]
\item  \label{it:formulation4}
\[
\int \norm{\Hess_K^* f}_{g_K}^2 d\nu_K \geq (2-p) \int |\grad_{g_K} f|^2 d\nu_K \;\;\; \forall f \in C^2_e . 
\]
\end{enumerate}
When $p=0$, these are equivalent to:
\begin{enumerate}
\setcounter{enumi}{4}
\item \label{it:formulation5}
\[
\int \norm{\Hess_K^* f + f g_K}_{g_K}^2 d\nu_K \geq (n-1) \int f^2 d\nu_K \;\;\; \forall f \in C^2_e . 
\]
\end{enumerate} 
\end{thm}
\begin{proof}
The duality (say, via the spectral theorem) between (\ref{it:formulation1}) and (\ref{it:formulation2}) was already explained. The equivalence between (\ref{it:formulation2}), (\ref{it:formulation3}) and (\ref{it:formulation4}) is then immediate by using the centro-affine Bochner formula (\ref{eq:magic}). When $p=0$, (\ref{it:formulation4}) is readily seen to be equivalent to (\ref{it:formulation5}) after noting that:
\[
 \int \scalar{\Hess_K^* f , f g_K}_{g_K} d\nu_K = \int f \Delta_K(f) d\nu_K = - \int |\grad_{g_K} f|^2 d\nu_K . 
 \]
\end{proof}

\begin{rem} \label{rem:one-is-trivial}
As we've already seen in (\ref{eq:Lich-CS}), formulation (\ref{it:formulation3}) holds trivially for any $K \in \K^\infty_+$ with $p=1$ by an application of the Cauchy--Schwarz inequality, regardless of evenness of $f$, and is actually equivalent to the local formulations of the Brunn--Minkowski inequality (\ref{eq:local-BM}) and (\ref{eq:local-BM-dual}). It follows by the centro-affine Bochner formula as in Theorem \ref{thm:equivalences} that the same applies to formulation (\ref{it:formulation4}).
\end{rem}

Formulation (\ref{it:formulation5}) is rather appealing since by (\ref{eq:Hess-of-x*}), we have $\Hess_K^* x^*_K = -x^*_K g_K$, and by Theorem \ref{thm:eigenspace}, the left-hand-side vanishes if and only if  $f = \scalar{x^*_K , \xi}$ for some $\xi \in E$, the (odd) eigenfunction of $\Delta_K$ corresponding to $\lambda_1(-\Delta_K) = n-1$; unfortunately, we do not know how to exploit this. 
Arranging the other conjectured inequalities above as follows:
\[
\int \norm{\Hess_K^* f}_{g_K}^2 d\nu_K \geq_? \frac{2-p}{n-p}  \int  (\Delta_K f)^2 d\nu_K \geq_? (2-p) \int |\grad_{g_K} f|^2 d\nu_K \;\;\; \forall f \in C^2_e ,
\]
we see that formulation (\ref{it:formulation4}) relating between the left-most and right-most terms is formally the weakest, and so we will concentrate our efforts on establishing it, at least for some rich class of convex bodies $K$. 

\subsection{A possible strategy} \label{subsec:strategy}

In view of Remark \ref{rem:one-is-trivial}, we see that the remaining challenge is to exploit the evenness of $f$ and origin-symmetry of $K$ to turn the trivial constant $1$ in (\ref{it:formulation4}) into the conjectured constant $2$, corresponding to $p=0$. This is very reminiscent of the challenge in the resolution of the B-conjecture in \cite{CFM-BConjecture}. As in \cite{CFM-BConjecture}, a natural idea for deriving (\ref{it:formulation4}) is to try and apply (\ref{it:formulation1}) to some partial derivative $z$ of $f$ (which will necessarily be odd and hence integrate to zero as required in (\ref{it:formulation1})), and then sum over all partial derivatives, thereby gaining one order of differentiation. However, in order to obtain $|\grad_{g_K} f|^2 = g_K^{ij} f_i f_j$ on the right-hand-side of (\ref{it:formulation4}), we would need to incorporate the square-root of the metric $g_K$ into the definition of $z$ -- an inherent complication in our non-Euclidean setting. 
Despite having the elegant representation $g_K = \scalar{dx^*, dx}$, we do not know how to effectively take the square-root of $g_K$ (without introducing additional parameters we have no control over). Although we will not use this here, we mention the following additional equivalent formulation, which gives us some more flexibility:

\begin{thm}
Given $K \in \K^\infty_{+,e}$ and $p \leq 1$, the even local $L^p$-Brunn--Minkowski conjectured inequality (\ref{eq:local-Lp-BM}) holds if and only if for all $f \in C^2_e$, there exists  $h \in C^2$ so that:
\begin{equation} \label{eq:equivalent-polarized}
\int f h \; d\nu_K > 0 \text{ and } \int \scalar{\Hess^*_K f,\Hess^*_K h}_{g_K} d\nu_K \geq (2-p) \int g_K(\grad_{g_K} f,\grad_{g_K} h) d\nu_K . 
\end{equation}
\end{thm}
\begin{proof}
The ``only if" direction follows by using $h=f$ in the implication (\ref{it:formulation4}) $\Rightarrow$ (\ref{it:formulation1}) in Theorem \ref{thm:equivalences}. 
For the ``if" direction, simply polarize the centro-affine Bochner formula (\ref{eq:magic}):
\begin{equation} \label{eq:magic-polarized}
\int (\Delta_K f) (\Delta_K h) d\nu_K - \int \scalar{\Hess^*_K f,\Hess^*_K h}_{g_K} d\nu_K = (n-2) \int g_K(\grad_{g_K} f,\grad_{g_K} h) d\nu_K ,
\end{equation}
for all $f,h \in C^2$. By elliptic regularity, the first non-constant even eigenfunction $f$ of $-\Delta_K$ corresponding to the eigenvalue $\lambda_{1,e}$ satisfies $f \in C^\infty_e$. 
Let $h \in C^2$ be the function from our assumption. Note that:
\begin{equation} \label{eq:positive}
\int g_K(\grad_{g_K} f,\grad_{g_K} h) d\nu_K = - \int (\Delta_K f) h \; d\nu_K = \lambda_{1,e} \int f h \; d\nu_K > 0 . 
\end{equation}
Combining (\ref{eq:equivalent-polarized}) and (\ref{eq:magic-polarized}) and integrating by parts, we obtain:
\[
\lambda_{1,e} \int g_K(\grad_{g_K} f,\grad_{g_K} h)  d\nu_K \geq (n-p) \int g_K(\grad_{g_K} f,\grad_{g_K} h) d\nu_K . 
\]
Using the positivity of (\ref{eq:positive}), we deduce $\lambda_{1,e} \geq n-p$, which is equivalent to (\ref{eq:local-Lp-BM}). 
\end{proof}

\section{Curvature pinching implies even $L^p$-Minkowski uniqueness} \label{sec:curvature}

We are now ready to reap the fruits of the insight gained from recasting our problem in the centro-affine differential language and establish the main results of this paper. First, it will be useful to record:
\begin{lemma} \label{lem:curvature}
Let $K \in \K^2_+$. Then for all $0 < A \leq B$:
\begin{enumerate}
\item $\frac{1}{B} \delta^{\partial K} \leq \II^{\partial K} \leq \frac{1}{A} \delta^{\partial K}$ if and only if $A \delta^{\S^*} \leq D^2 h_{K} \leq B \delta^{\S^*}$. 
\item $\frac{1}{B} \delta^{\partial K} \leq g_K^{\partial K} \leq \frac{1}{A} \delta^{\partial K}$ if and only if $A \delta^{\S^*} \leq h_K D^2 h_{K} \leq B \delta^{\S^*}$. 
\end{enumerate}
\end{lemma}
Recall that $\delta^{\partial K}$ denotes the induced Euclidean structure on $\partial K$ (i.e. the first fundamental form), and that
 $g_K^{\partial K}(x) = \II^{\partial K}_x  / h_K(\theta^*)$ by Proposition \ref{prop:partial-K-param}. 
\begin{proof}
Let us pull-back the assertions onto $\S^*$ using the inverse Gauss map $x : \S^* \ni \theta^* \mapsto \bar D h_K(\theta^*) \in \partial K$. Note that $d x(\theta^*) = D^2 h_K(\theta^*)$, and so the pull-back of $\delta^{\partial K}_x$ is $(D^2 h_K(\theta^*))^2$. Recalling that $\II^{\partial K}_x = (D^2 h_K(\theta^*))^{-1}$ under the natural identification between $T_x \partial K$ and $T_{\theta^*} \S^*$, the pull-back of $\II^{\partial K}_x$ is therefore $D^2 h_K(\theta^*)$. Consequently:
\begin{enumerate}
\item $\frac{1}{B} \delta^{\partial K} \leq \II^{\partial K} \leq \frac{1}{A} \delta^{\partial K}$ if and only if $\frac{1}{B} (D^2 h_K(\theta^*))^2 \leq  D^2 h_K(\theta^*) \leq \frac{1}{A} (D^2 h_K(\theta^*))^2$. 
\item $\frac{1}{B} \delta^{\partial K}_x \leq \II^{\partial K}_x / h_K(\theta^*) \leq \frac{1}{A} \delta^{\partial K}_x$ if and only if $\frac{1}{B} (D^2 h_K(\theta^*))^2 \leq D^2 h_K(\theta^*) / h_K(\theta^*) \leq \frac{1}{A} (D^2 h_K(\theta^*))^2$.
\end{enumerate}
\end{proof}

\begin{lemma} \label{lem:strong-curvature}
Let $K \in \K^2_{+,e}$ and assume that:
\begin{equation} \label{eq:strong-curvature-bounds}
\frac{1}{R} \delta^{\partial K} \leq \II^{\partial  K} \leq \frac{1}{r} \delta^{\partial K} .
\end{equation}
Then:
\begin{equation} \label{eq:ball-inclusions}
r B_2^n \subset K \subset R B_2^n ,
\end{equation}
and consequently:
\[
\frac{1}{R^2} \delta^{\partial K} \leq g_K^{\partial K} \leq \frac{1}{r^2} \delta^{\partial K} .
\]
\end{lemma}
\begin{proof} By Blaschke's Rolling Ball theorem \cite[Theorem 2.5.4/3.2.12]{Schneider-Book-2ndEd}, the assumption (\ref{eq:strong-curvature-bounds}) implies that the Euclidean ball $r B_2^n$ rolls freely inside $K$ and that $K$ moves freely inside the Euclidean ball $R B_2^n$. Convexity and origin-symmetry of $K$ then imply (\ref{eq:ball-inclusions}) since $B_2^n \subset \text{conv}((v + B_2^n) \cup (-v + B_2^n))$ and $B_2^n \supset (v+B_2^n) \cap (-v + B_2^n)$ for all $v \in E$. Recalling that $g_K^{\partial K}(x) = \II^{\partial K}_x  / h_K(\theta^*)$, the final assertion follows. 
\end{proof}

\subsection{Curvature pinching implies local even $L^p$-Brunn--Minkowski inequality}

\begin{thm} \label{thm:curvature-implies-local-BM}
Let $K \in \K^\infty_{+,e}$ have a centro-affine image $\tilde K$ so that:
\begin{equation} \label{eq:curvature-bounds}
\frac{1}{B} \delta^{\partial \tilde K} \leq g^{\partial \tilde K}_{\tilde K} \leq \frac{1}{A} \delta^{\partial \tilde K} ~,~ \tilde K \subset R B_2^n ,
\end{equation}
for some $A,B,R > 0$. Then $\lambda_{1,e}(-\Delta_{K}) \geq n-p$, i.e. $K$ satisfies the local even $L^p$-Brunn--Minkowski inequality (\ref{eq:local-Lp-BM}), with:
\begin{equation} \label{eq:curvature-p}
p = 2 - \frac{\frac{n-1}{2} A - R^2}{B} .  
\end{equation}
In particular, if $K$ has a centro-affine image for which (\ref{eq:strong-curvature-bounds}) holds, $K$ satisfies the local even $L^p$-Brunn--Minkowski inequality (\ref{eq:local-Lp-BM}) with:
\[
p = 3 - \frac{n-1}{2} \frac{r^2}{R^2} . 
\]
In particular, whenever $\brac{\frac{R}{r}}^2 \leq \frac{n-1}{6}$, $K$ satisfies the local even log-Brunn--Minkowski inequality. 
\end{thm}
\begin{proof}
As the spectrum $\sigma(-\Delta_{K})$ is centro-affine invariant, it is enough to prove the claim for $\tilde K = K$.

Let $\ee^1,\ldots,\ee^n$ denote a fixed orthonormal basis of $(E,\scalar{\cdot,\cdot})$. Recall that $x_K : M \rightarrow \partial K$ is our centro-affine parametrization of $\partial K$ on $M$.  We abbreviate $g = g_K$, $x = x_K$, $\nabla = \nabla_K$, $\nabla^* = \nabla^*_K$ and $\Hess^* = \Hess^*_K$. 
We also denote:
\[
 x^k := \sscalar{\ee^k, x } : M \rightarrow \R ~,~ x^k_i :=(x^k)_i = dx^k(e_i)
\]
 (for some local frame $e_1,\ldots,e_{n-1}$ on $M$). Define the following $(0,2)$ tensor on $M$:
\begin{equation} \label{eq:delta}
\delta_{ij} = \sscalar{dx(e_i), dx(e_j)}  = \sum_{k=1}^n x^k_i x^k_j  . 
\end{equation}
Note that when $M = \partial K$, as $dx_K^{\partial K} = \Id$, $\delta$ is precisely the induced Euclidean metric $\delta^{\partial K}$ on $\partial K$, ensuring that our notation is consistent. We deduce that, regardless of parametrization, we are given that: 
\begin{equation} \label{eq:metric-bounds}
\frac{1}{B} \delta \leq g \leq \frac{1}{A} \delta .
\end{equation}

Given $f \in C^2_e(M)$ and $k=1,\ldots,n$, introduce the function $F^k \in C^1(M)$ given by: 
\[
F^k := (\grad_{g} f)(x^k) = x^k_i  f^i ~,~  f^i := g^{ip} f_p .
\]
 Since $K$ is origin-symmetric and $f$ is even, $F^k$ is clearly odd, and in particular $\int F^k d\nu_K = 0$. Consequently, applying the local ($L^1$-)Brunn-Minkowski inequality (\ref{eq:local-BM}) to each $F^k$ and summing (it is clearly enough that $F^k \in C^1$), we obtain:
\[
\sum_{k=1}^n \int |\grad_g F^k|^2 d\nu_K \geq (n-1) \sum_{k=1}^n \int  (F^k)^2 d\nu_K . 
\]

Using (\ref{eq:metric-bounds}), we first estimate from below:
\begin{align*}
& \sum_{k=1}^n \int  (F^k)^2 d\nu_K = \int  \sum_{k=1}^n x^k_i x^k_j f^i f^j d\nu_K = \int  \delta_{ij} f^i f^j d\nu_K \\
& \geq A \int g_{ij} f^i f^j d\nu_K = A \int |\grad_{g} f|^2 d\nu_K . 
\end{align*}
Next, note that by (\ref{eq:conjugation}) and (\ref{eq:CA-structure}):
\begin{align}
\nonumber \partial_a (x^k_i f^i) & = \partial_a (x^k_i g^{ip} f_p)  = (\nabla_{a} x^k_i) f^i + x^k_i g^{ip} \nabla^*_{a} f_p  \\
\label{eq:AB} & = -g_{ai} x^k f^i + x^k_i g^{ip} \Hess^*_{a p} f = -x^k f_a + x^k_i g^{ip} \Hess^*_{a p} f =: P^k_a + Q^k_a .
\end{align}
Using (\ref{eq:AB}), (\ref{eq:delta}), (\ref{eq:metric-bounds}) and (\ref{eq:ball-inclusions}), we now estimate from above:
\begin{align*}
& \sum_{k=1}^n \int |\grad_g F^k|^2 d\nu_K \\
& = \sum_{k=1}^n \int g^{ab} \partial_a (x^k_i f^i) \partial_b(x^k_j f^j) d\nu_K \\
& = \sum_{k=1}^n \int g^{ab} (P^k_a + Q^k_a) (P^k_b + Q^k_b) d\nu_K \\
& \leq 2 \sum_{k=1}^n \int g^{ab} (P^k_a P^k_b + Q^k_a Q^k_b) d\nu_K \\
& = 2 \int \brac{ \sum_{k=1}^n (x^k)^2 g^{ab} f_a f_b + (\sum_{k=1}^n x^k_i x^k_j)  g^{ab} g^{ip} g^{jq} \Hess^*_{a p} f \Hess^*_{b q} f } d\nu_K \\
& = 2 \int \brac{ |x|^2 |\grad_{g} f|^2 + \delta_{ij}  g^{ab} g^{ip} g^{jq} \Hess^*_{a p} f \Hess^*_{b q} f } d\nu_K \\
& \leq 2 \int \brac{ R^2 |\grad_{g} f|^2  + B \norm{\Hess^* f}_{g}^2} d\nu_K ,  \end{align*}
where in the last inequality we used the well known fact that $S_{ij} T^{ij} \geq 0$ whenever $S,T \geq 0$. 

Combing the three inequalities derived above, we obtain:
\[
2 \int \brac{R^2 |\grad_{g} f|^2 + B \norm{\Hess^* f}_{g}^2} d\nu_K \geq (n-1) A \int |\grad_{g} f|^2 d\nu_K  . 
\]
Rearranging, we deduce for every $f \in C^2_e$:
\[
\int \norm{\Hess^* f}_{g}^2 d\nu_K \geq \frac{\frac{n-1}{2} A - R^2}{B} \int |\grad_{g} f|^2 d\nu_K .
\]
By Theorem \ref{thm:equivalences}, this is equivalent to the local $L^p$-Brunn--Minkowski inequality (\ref{eq:local-Lp-BM}) with $p$ given by (\ref{eq:curvature-p}). 

Lastly, the ``in particular" part of Theorem \ref{thm:curvature-implies-local-BM} follows immediately by Lemma \ref{lem:strong-curvature}. 
\end{proof}

\subsection{Uniqueness in the even $L^p$-Minkowski problem} \label{subsec:UniquenessIngredients}

We can now translate Theorem \ref{thm:curvature-implies-local-BM} into a global uniqueness result for the even $L^p$-Minkowski problem and an even $L^p$-Minkowski inequality. We obtain the following strengthened version of Theorem \ref{thm:main-Lp-Minkowski}, which we will require for the isomorphic results of the next section.

\begin{thm} \label{thm:Lp-Minkowski}
Let $K \in \K^{2,\alpha}_{+,e}$ have a centro-affine image $\tilde K \subset R B_2^n$ so that the following curvature pinching bounds hold:
\begin{equation} \label{eq:curvature-bounds}
\frac{1}{B} \delta^{\partial \tilde K} \leq \frac{\II^{\partial \tilde K}}{h_{\tilde K}(\n^{\partial \tilde K})} \leq \frac{1}{A} \delta^{\partial \tilde K}  ,
\end{equation}
for some $B \geq A > 0$ and $R > 0$. Then for any:
\begin{equation} \label{eq:AB-cond}
2 - \frac{\frac{n-1}{2} A - R^2}{B} < p < 1 ,
\end{equation}
the even $L^p$-Minkowski problem for $K$ has a unique solution:
\begin{equation} \label{eq:Lp-Minkowski-Uniqueness}
L \in \K_e \;\;\; S_p L = S_p K  \;\; \Rightarrow \;\; L = K ,
\end{equation}
and the even $L^p$-Minkowski inequality holds:
\begin{equation} \label{eq:Lp-Minkowski-Inq}
\forall L \in \K_e  \;\;\; \frac{ \frac{1}{p} \int h_L^{p} dS_p K }{ V(L)^{\frac{p}{n}} } \geq  \frac{ \frac{1}{p} \int h_K^{p} dS_p K }{ V(K)^{\frac{p}{n}} } = \frac{n}{p} V(K)^{1-\frac{p}{n}}  ,
\end{equation}
with equality if and only if $L = c K$ for some $c > 0$. 
\end{thm}

As usual, the case $p=0$ above is interpreted in the limiting sense as in (\ref{eq:main-log-Minkowski-Inq}). Note that Theorem \ref{thm:Lp-Minkowski} immediately implies Theorem \ref{thm:main-Lp-Minkowski}, since the assumption (\ref{eq:main-strong-curvature-bounds}) implies (\ref{eq:curvature-bounds}) with $A = r^2$ and $B = R^2$ by Lemma \ref{lem:strong-curvature}. 

\begin{proof}[Proof of Theorem \ref{thm:Lp-Minkowski}]
As Theorem \ref{thm:curvature-implies-local-BM} was formulated for $K \in \K^{\infty}_{+,e}$, let us first assume that this is the case. 
Denote by $T \in GL_n$ the linear map so that $K = T(\tilde K)$ with
\begin{equation} \label{eq:sandwich}
 r B_2^n \subset \tilde K \subset R B_2^n
 \end{equation}
satisfying (\ref{eq:curvature-bounds}). Without loss of generality, we may assume that $r,R$ are best possible in (\ref{eq:sandwich}), and by scaling, we may assume that $r R = 1$, so that in particular $r \leq 1 \leq R$. By Lemma \ref{lem:strong-curvature}, since $\frac{r}{B} \delta^{\partial \tilde K} \leq \II^{\partial \tilde K} \leq \frac{R}{A} \delta^{\partial \tilde K}$, we have $R \leq B/r$ and $r \geq A/R$, and hence $A \leq 1 \leq B$.

Define for $t \in [0,1]$:
\[
r_t := (1-t) + t r ~,~ R_t := (1-t) + t R ,
\]
and let $\tilde K_t,K_t \in \K^{\infty}_{+,e}$ be defined via:
\[
h_{\tilde K_t} := (1-t) + t  h_{\tilde K} ~,~  T_t := (1-t) \Id + t T ~,~ K_t := T_t(\tilde K_t) ,
\]
so that $K_0 = B_2^n$ and $K_1 = K$.  

 By Lemma \ref{lem:curvature}, the assumption (\ref{eq:curvature-bounds}) is equivalent to:
\[
A \delta^{\S^*} \leq h_{\tilde K} D^2 h_{\tilde K} \leq B \delta^{\S^*}  \text{ on $\S^*$.}
\]
Denote by $A_t,B_t > 0$ the best corresponding constants for $\tilde K_t$:
\[
A_t \delta^{\S^*} \leq h_{\tilde K_t} D^2 h_{\tilde K_t} = \brac{(1-t) + t h_{\tilde K}}\brac{ (1-t) \delta^{\S^*} + t D^2 h_{\tilde K} } \leq B_t \delta^{\S^*}  \text{ on $\S^*$.}
\]
We claim that:
\begin{equation} \label{eq:headache}
A_t \geq A \text{ and } B_t \leq B \;\;\; \forall t \in [0,1] .
\end{equation}
Indeed, by fixing $\theta^*$ and diagonalizing $D^2 h_{\tilde K}(\theta^*) > 0$, it is enough to understand the extremal values of the quadratic polynomial $Q(t) := ((1-t) + t p) ((1-t) + t q)$ on the interval $[0,1]$ when $p,q > 0$. A simple calculation verifies that this polynomial is always monotone (either increasing or decreasing) on $[0,1]$, or equivalently, that the mid-point between its two roots lies outside the interval $(0,1)$:
\[
\frac{-\frac{1}{q-1} - \frac{1}{p-1}}{2} = \frac{1 - \frac{p+q}{2}} {1 + pq - (p+q)} \begin{cases} \geq \\ \leq \end{cases} 
\frac{1 -  \frac{p+q}{2}}{ 1 + (\frac{p+q}{2})^2 -  (p+q)} = \frac{1}{1 -  \frac{p+q}{2}} \begin{cases} \geq 1 &  0 < \frac{p+q}{2} < 1 \\ \leq 0 & \frac{p+q}{2} \geq 1 \end{cases} .
\]
Consequently, the minimum and maximum of $Q$ on $[0,1]$ are attained at the end points. Since $A \leq 1 \leq B$, (\ref{eq:headache}) immediately follows. 

Now, since clearly $R_t \leq R$, if follows that for all $t \in [0,1]$:
\[
2 - \frac{\frac{n-1}{2} A_t - R_t^2}{B_t} \leq 2 - \frac{\frac{n-1}{2} A - R^2}{B} =: p_0 . 
\]
We may assume that $p_0 < 1$, otherwise there is nothing to prove. It follows by Theorem \ref{thm:curvature-implies-local-BM} that for all $t \in [0,1]$, $K_t$ satisfies the local $L^{p_0}$-Brunn--Minkowski inequality (\ref{eq:local-Lp-BM}). 

Note that it always holds that $p_0 > -n$ because $A \leq B$. 

We now claim that the exact same conclusion also holds whenever $K \in \K^{2}_{+,e}$. Indeed, recall that $K_t$ satisfies (\ref{eq:local-Lp-BM}) if and only if $\lambda_{1,e}(-\Delta_{K_t}) \geq n-p_0$. It was shown in \cite[Theorem 5.3]{KolesnikovEMilman-LocalLpBM} that the eigenvalues of $-\Delta_{L}$ are continuous in $L$ with respect to the $C^2$ topology. Consequently, if we approximate $K$ by $\{K^i\} \subset \K^{\infty}_{+,e}$ in $C^{2}$, since $K^i_t$ obviously tend to $K_t$ in $C^{2}$ as $i \rightarrow \infty$ as well, it follows that $\lambda_{1,e}(-\Delta_{K^i_t}) \rightarrow \lambda_{1,e}(-\Delta_{K_t})$. We thus deduce that $\lambda_{1,e}(-\Delta_{K_t}) \geq n-p_0$ for all $t \in [0,1]$ whenever $K \in \K^{2}_{+,e}$. 

We now assume that $K \in \K^{2,\alpha}_{+,e}$, to ensure that $[0,1] \ni t \mapsto K_t \in \K^{2,\alpha}_{+,e}$ is a continuous deformation in the $C^{2,\alpha}$ topology. The implications (\ref{it:main4}) $\Rightarrow$ (\ref{it:main1}) $\Rightarrow$ (\ref{it:main3b}) of Theorem \ref{thm:equiv-full} for $\F := \{ K_t \}_{t \in [0,1]} \subset \K^{2,\alpha}_{+,e}$ now immediately establish the asserted uniqueness  in the even $L^p$-Minkowski problem for $K$  (\ref{eq:Lp-Minkowski-Uniqueness}), as well as the corresponding $L^p$-Minkowski inequality (\ref{eq:Lp-Minkowski-Inq}) along with its equality case, for any $p \in (p_0 , 1)$. 
\end{proof}

\section{Isomorphic and isometric $L^p$-Minkowski problem} \label{sec:iso}

We conclude this work by deducing Theorems \ref{thm:main-iso-Lp} and \ref{thm:main-isometric}. To this end, we require the following:

\begin{proposition} \label{prop:iso}
Let $K \in \K_e$, and set $D = d_{BM}(K,B_2^n)$. Then for any $\alpha, \beta > 0$, there exists $\tilde K \in \K^\infty_{+,e}$ so that:
\begin{equation} \label{eq:iso-BM-distance}
d_{BM}(K,\tilde K) \leq  (1+\beta)  \sqrt{1+\alpha^2} ,
\end{equation}
\begin{equation} \label{eq:iso-radii}
r B_2^n \subset \tilde K \subset R B_2^n ,
\end{equation}
and:
\begin{equation} \label{eq:iso-curvature-distance}
\frac{1}{B} \delta^{\partial \tilde K}\leq \frac{\II^{\partial \tilde K}}{h_{\tilde K}(\n^{\partial \tilde K})} \leq \frac{1}{A} \delta^{\partial \tilde K} ,
\end{equation}
with $A,B,r , R > 0$ given by:
\begin{equation} \label{eq:iso-rR-defs}
r : = \beta + \frac{1}{\sqrt{1 + \alpha^2 / D^2}} ~,~ R := \frac{D}{\sqrt{1+\alpha^2}} + \beta  ,
\end{equation}
\begin{equation} \label{eq:iso-AB-defs}
A := \beta r ~,~ B:= \frac{D^2}{\alpha^2} \brac{1 + \beta \sqrt{1 + \alpha^2/D^2}} + \beta R . 
\end{equation}
\end{proposition}

\begin{rem} \label{rem:isomorphic-conjecture}
It is natural to conjecture that with the above assumptions, for any $\gamma \in [1,D]$, there should exist $\tilde K \in \K^\infty_{+,e}$ so that:
\[
d_{BM}(K, \tilde K) \leq C \gamma ,
\]
and:
\[
\delta^{\partial \tilde K} \leq \II^{\partial \tilde K} \leq \frac{C D}{F(\gamma)} \delta^{\partial \tilde K}  ,
\]
for some universal constant $C > 1$ and some increasing function $F : [1,D] \rightarrow [1,D]$ (probably $F(\gamma) = \gamma$). This would enable us to deduce the isomorphic results of this section directly from Theorem \ref{thm:main-Lp-Minkowski}, without going through the stronger Theorem \ref{thm:Lp-Minkowski}. Unfortunately, we do not see a simple argument for showing the above. The crux of the problem is that $\frac{1}{B} \delta^{\partial \tilde K} \leq \II^{\partial \tilde K} \leq \frac{1}{A} \delta^{\partial \tilde K}$ does not imply (in general)  $A \delta^{\partial \tilde K^{\circ}} \leq \II^{\partial \tilde K^{\circ}} \leq B \delta^{\partial \tilde K^{\circ}}$, as the curvature of the polar body picks up additional factors depending on $D$.
\end{rem}

\begin{proof}[Proof of Proposition \ref{prop:iso}]
By applying a linear transformation, we may assume that $K$ is in John's position, so that:
\begin{equation} \label{eq:D-bounds}
B_2^n \subset K \subset D B_2^n . 
\end{equation}
Furthermore, by standard arguments (such as mollifying $h_K$, Minkowski adding a small Euclidean ball, and rescaling \cite[pp. 184-185]{Schneider-Book-2ndEd}), we may assume that $K \in \K^{\infty}_{+,e}$ without altering (\ref{eq:D-bounds}), and only incurring an extra (arbitrarily small) $\eps > 0$ in the final estimate for $d_{BM}(K,\tilde K)$, as described below. 

We now construct the required body $\tilde K \in \K^{\infty}_{+,e}$ by defining:
\[
\tilde K := L + \beta B_2^n ~,~ L:= \brac{K^{\circ} +_2 (\frac{\alpha}{D} B_2^n)}^{\circ} .
\]
Here $A+_2 B$ denotes the $2$-Firey sum of two convex bodies $A,B \in \K$, defined via:
\[
h_{A+_2 B}^2 := h_A^2 + h_B^2 . 
\]
In other words:
\[
\norm{x}_L^2 = \norm{x}_K^2 + \frac{\alpha^2}{D^2} |x|^2  . 
\]

As:
\[
\frac{1}{D} |x| \leq \norm{x}_K \leq |x| ,
\]
we have:
\[
 \brac{1 + \frac{\alpha^2}{D^2}} \norm{x}_K^2  , \frac{1+\alpha^2}{D^2} |x|^2  \leq \norm{x}_L^2 \leq (1+\alpha^2) \norm{x}_K^2 , \brac{1 + \frac{\alpha^2}{D^2}} |x|^2 , 
\]
or equivalently:
\begin{equation} \label{eq:L-bounds}
\frac{1}{\sqrt{1+\alpha^2}} K , \frac{1}{\sqrt{1 + \alpha^2 / D^2}} B_2^n \subset L \subset \frac{1}{\sqrt{1 + \alpha^2 / D^2}} K , \frac{D}{\sqrt{1+\alpha^2}} B_2^n . 
\end{equation}
Consequently:
\begin{equation} \label{eq:tildeK-bounds}
\brac{\frac{1}{\sqrt{1+\alpha^2}} + \frac{\beta}{D}} K , r  B_2^n \subset \tilde K \subset 
\brac{\frac{1}{\sqrt{1 + \alpha^2 / D^2}} + \beta} K , R B_2^n ,
\end{equation}
with $r,R > 0$ given by (\ref{eq:iso-rR-defs}). 
Hence, reinserting the extra $\eps$-error due to the smoothing of the original $K$, by choosing $\eps > 0$ sufficiently small we may ensure:
\[
d_{BM}(K,\tilde K) \leq \frac{\frac{1}{\sqrt{1 + \alpha^2 / D^2}} + \beta}{\frac{1}{\sqrt{1+\alpha^2}} + \frac{\beta}{D}} + \eps \leq (1+\beta) \sqrt{1+\alpha^2} . 
\]

Next, defining $f(x) := \norm{x}_K^2 + \frac{\alpha^2}{D^2} |x|^2$, note that $\partial L = \{ x \; ; \; \norm{x}_L = 1 \} = \{ f = 1 \}$. Consequently, the unit outer normal to $\partial L$ is $\n^{\partial L}(x) = \frac{\bar D f(x)}{|\bar Df(x)|}$. A standard computation then verifies that:
\[
\II^{\partial L}_x = \frac{P_{(\n^{\partial L})^{\perp}} \bar D^2 f \; P_{(\n^{\partial L})^{\perp}}}{|\bar Df(x)|} . 
\]
As $\norm{x}_K^2$ is convex on $E$, and as:
\[
|\bar D f(x)| h_L(\n^{\partial L}(x)) = \scalar{\bar D f(x), x} = 2 f(x) = 2 \;\;\; \forall x \in \partial L ,
\]
by Euler's identity and the $2$-homogeneity of $f$, it follows that:
\[
\II^{\partial L}_x  \geq \frac{2 \frac{\alpha^2}{D^2} \delta^{\partial L}_x}{2 / h_L(\n^{\partial L}(x))} = \frac{\alpha^2}{D^2} h_L(\n^{\partial L}(x)) \delta^{\partial L}_x .
\]
Recalling that $D^2 h_L(\n^{\partial L}(x)) = (\II^{\partial L}_x)^{-1}$ and applying the same argument as in Lemma \ref{lem:curvature}, we deduce:
\[
D^2 h_{L} \leq \frac{D^2}{\alpha^2} \frac{1}{h_L}  \delta^{\S^*} . 
\]
As $h_{\tilde K} = h_L + \beta$ on $\S^*$, it follows that:
\[
\beta \delta^{\S^*} \leq D^2 h_{\tilde K} \leq \brac{\frac{D^2}{\alpha^2} \frac{1}{h_L} + \beta} \delta^{\S^*}  ,
\]
and hence:
\[
\beta h_{\tilde K} \delta^{\S^*} \leq h_{\tilde K} D^2 h_{\tilde K} \leq \brac{\frac{D^2}{\alpha^2} \brac{1 + \frac{\beta}{h_L}} + \beta h_{\tilde K}} \delta^{\S^*} .
\]
Recalling (\ref{eq:L-bounds}) and (\ref{eq:tildeK-bounds}), we deduce that:
\[
A \delta^{\S^*} \leq h_{\tilde K} D^2 h_{\tilde K} \leq B \delta^{\S^*} ,
\]
with $A,B > 0$ given by (\ref{eq:iso-AB-defs}). A final application of Lemma \ref{lem:curvature} concludes the proof. 
\end{proof}

We are now ready to complete the proofs of Theorems \ref{thm:main-iso-Lp} and \ref{thm:main-isometric}. 

\begin{proof}[Proof of Theorems \ref{thm:main-iso-Lp} and \ref{thm:main-isometric}]

Given $\bar K \in \K_e$ and $\gamma > 0$ where $D := d_{BM}(\bar K,B_2^n)$, we construct $\tilde K \in \K^{\infty}_{+,e}$ as in Proposition \ref{prop:iso} with a certain choice of $\alpha,\beta > 0$ to be determined later. Recall that $d_{BM}(\bar K,\tilde K) \leq (1+\beta) \sqrt{1+\alpha^2}$ and that $\tilde K$ satisfies (\ref{eq:iso-radii}) and (\ref{eq:iso-curvature-distance}), where $r,R,A,B > 0$ are parameters depending on $\alpha,\beta,D$ which are given by (\ref{eq:iso-rR-defs}) and (\ref{eq:iso-AB-defs}). By Theorem \ref{thm:Lp-Minkowski}, for any $p$ in the range:
\[
p_{\alpha,\beta,D} := 2 - \frac{\frac{n-1}{2} A - R^2}{B} < p < 1 ,
\]
$\tilde K$ satisfies the conclusion of Theorem \ref{thm:main-iso-Lp}. Consequently, our task is to show that $p_{\alpha,\beta,D} \leq p_{\gamma,D}$ for some choice of $\alpha,\beta > 0$ so that $(1+\beta) \sqrt{1+\alpha^2} = \gamma$, where $p_{\gamma,D}$ is given by (\ref{eq:main-p-gammaD}). 

Denote:
\[
b := \frac{\beta}{1+\beta} .
\]
We will ensure that the parameters $\alpha,\beta,\gamma$ satisfy:
\begin{equation} \label{eq:guarantee}
b \gamma = \beta \sqrt{1+\alpha^2} \leq D (\sqrt{2} - 1) ~,~ \alpha \leq D .
\end{equation}
Plugging the expressions for $r,R,A,B$, we have:
\[
r > \beta ~,~ R \leq \frac{\sqrt{2} D}{\sqrt{1+\alpha^2}} ~,~ \sqrt{1 + \alpha^2 / D^2} \leq \sqrt{2} .
\]
We now estimate:
\begin{align}
\nonumber
\frac{\frac{n-1}{2} A - R^2}{B} & = \frac{ \frac{n-1}{2} \beta r - R^2}{\beta R + \frac{D^2}{\alpha^2} (1 + \beta \sqrt{1 + \alpha^2/D^2})} 
> \frac{\frac{n-1}{2} \beta^2 - \frac{2 D^2}{1+\alpha^2}}{\beta \frac{\sqrt{2} D}{\sqrt{1+\alpha^2}} + \frac{D^2}{\alpha^2}  (1 + \sqrt{2} \beta)} \\
\label{eq:nicer-expression} & = \frac{ \frac{n-1}{2} b^2 \gamma^2 - 2 D^2}{\sqrt{2} b \gamma D + D^2 \frac{1+\alpha^2}{\alpha^2} (1 + \sqrt{2} \beta) } . 
\end{align}

\bigskip

Let us start with the isomorphic result of Theorem \ref{thm:main-iso-Lp}. It is apparent that in order to maximize the latter expression given $\gamma$ and $D$, it is preferable to choose $\beta$ of the order of a constant. To obtain an aesthetically pleasing expression, we set:
\[
\beta = 1 + \sqrt{2} \;\; \Leftrightarrow \;\; b = \frac{1}{\sqrt{2}} \; ,
\]
yielding:
\[
\frac{\frac{n-1}{2} A - R^2}{B} \geq \frac{ \frac{n-1}{4} \gamma^2 - 2 D^2}{ \gamma D + D^2 \frac{1+\alpha^2}{\alpha^2} (3 + \sqrt{2}) } . 
\]
Recalling our guarantee that:
\[
\gamma \leq \frac{\sqrt{2}-1}{b} D = (2 - \sqrt{2}) D ,
\]
it follows by also ensuring that $\alpha^2 \geq 3 + \sqrt{2}$ that:
\[
p_{\alpha,\beta,D} = 2 - \frac{\frac{n-1}{2} A - R^2}{B} <  2 - \frac{ \frac{n-1}{4} \gamma^2 - 2 D^2}{ 6 D^2} = \frac{7}{3} - \frac{n-1}{24} \frac{\gamma^2}{D^2} = p_{\gamma,D} .
\]
It remains to conveniently summarize our restrictions on $\gamma= (1+\beta) \sqrt{1+\alpha^2}$:
\[
\gamma \geq (2+\sqrt{2}) \sqrt{4 + \sqrt{2}} \simeq 7.95 ~,~ \gamma \leq (2-\sqrt{2}) D \simeq 0.58 D . 
\]
This concludes the proof of Theorem \ref{thm:main-iso-Lp}. 

\bigskip 

As for the isometric result of Theorem \ref{thm:main-isometric}, one just notes that in order to have $p_{\gamma,D} < 0$, it is enough to have the expression in (\ref{eq:nicer-expression}) greater or equal to $2$, i.e.:
\[
 \frac{n-1}{2} \beta^2 (1+\alpha^2) \geq \sqrt{2} 2 \beta \sqrt{1+\alpha^2}  D + 2 D^2 \frac{1+\alpha^2}{\alpha^2} (1 + \sqrt{2} \beta)  + 2 D^2. 
\]
We set $\alpha = \frac{\sqrt{D}}{\sqrt[4]{n}}$ and $\beta = C_\beta \alpha$ for an appropriate universal constant $C_\beta > 1$ to be determined, and recall that $D \leq \sqrt{n}$ and hence $\alpha \leq 1$ and $\beta \leq C_\beta$. It is therefore enough to have:
\[
\frac{n-1}{2} C_\beta^2 \frac{D}{\sqrt{n}} \geq 4 C_\beta D + 4 \sqrt{n} D (1 + \sqrt{2} \beta)   + 2 D^2 .
\]
A sufficient condition for that is:
\[
\sqrt{n} D (\frac{1}{4} C_\beta^2 - 4 \sqrt{2} C_\beta - 4 \frac{C_\beta}{\sqrt{n}}) \geq 4 \sqrt{n} D + 2 D^2 ,
\]
and in turn, a sufficient condition for that is:
\[
\frac{1}{4} C_\beta^2 - 4 \sqrt{2} C_\beta - 4 \frac{C_\beta}{\sqrt{n}} \geq 6 . 
\]
Clearly, this holds for a sufficiently large constant $C_\beta > 1$. Finally, it follows that:
\begin{equation} \label{eq:adjust-constant}
\gamma = (1+\beta) \sqrt{1 + \alpha^2} \leq 1 + C \frac{\sqrt{D}}{\sqrt[4]{n}} ,
\end{equation}
for another universal constant $C > 1$. 

There is still one last point to take care of -- we need to guarantee that the assumptions (\ref{eq:guarantee}) hold. Since $\alpha \leq 1 \leq D$, we just need to make sure that $\sqrt{2} C_\beta  \leq D (\sqrt{2} - 1)$. Since we only ever used $D$ as an upper bound on $d_{BM}(\bar K,B_2^n)$, it follows that we should use $\max(D, \frac{\sqrt{2} C_\beta}{\sqrt{2}-1})$ in place of $D$, which simply translates into using a different value for the constant $C > 1$ in (\ref{eq:adjust-constant}). This concludes the proof of Theorem \ref{thm:main-isometric}. 
\end{proof}

\bibliographystyle{plain}
\bibliography{../../../ConvexBib}

\clearpage
\includepdf[pages=-]{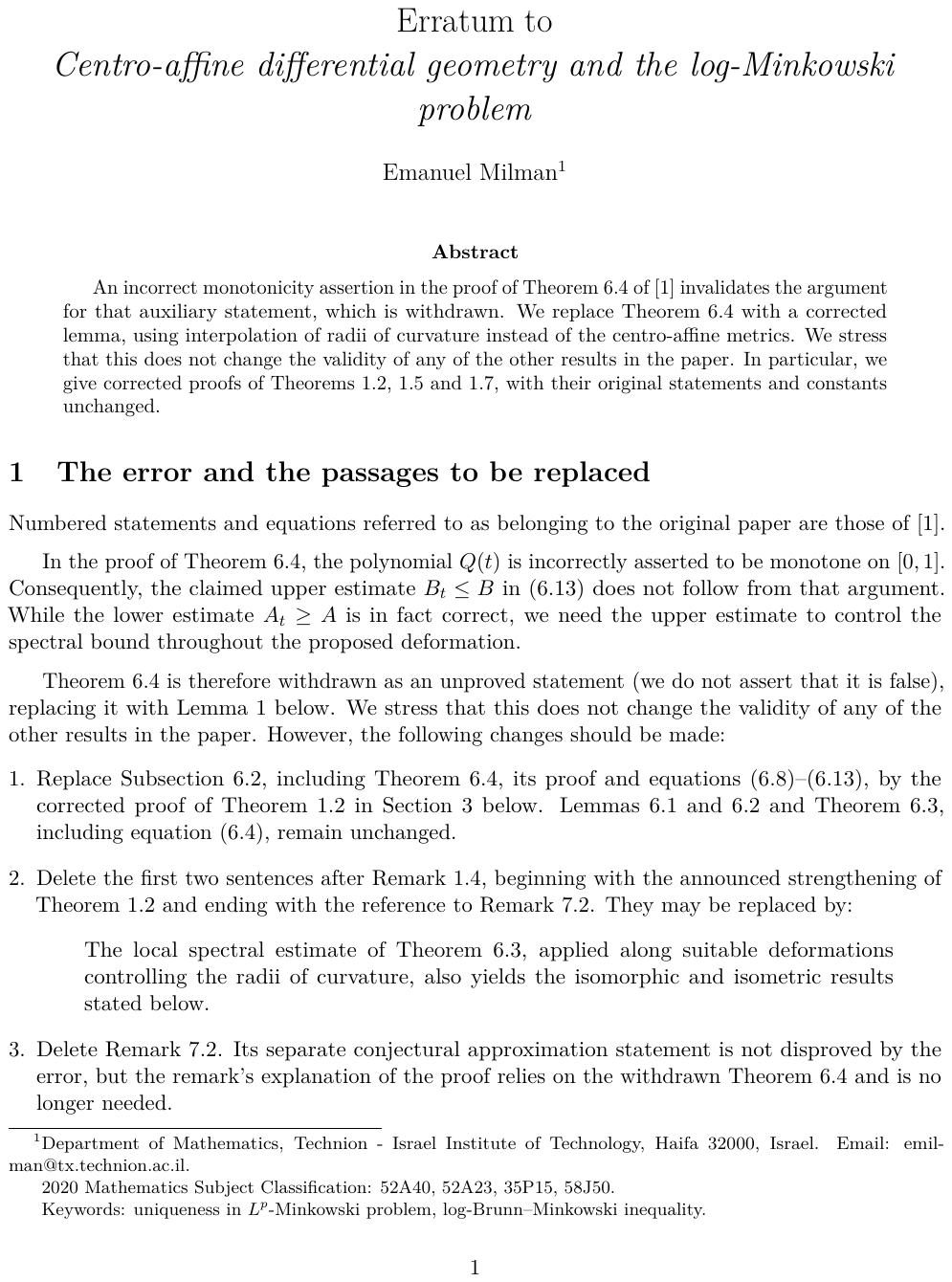}

\end{document}